\documentclass[reqno]{amsart}

\usepackage{amsmath,amsthm,amstext,amssymb,bbm}
\usepackage{hyperref}
\usepackage{enumerate,enumitem}

\theoremstyle{plain}
\newtheorem{theorem}{Theorem}[section]
\newtheorem{lemma}[theorem]{Lemma}

\newtheorem{corollary}[theorem]{Corollary}
\newtheorem{proposition}[theorem]{Proposition}

\newtheorem*{claim*}{Claim}
\newtheorem*{subclaim*}{Subclaim}
\theoremstyle{definition}
\newtheorem{definition}[theorem]{Definition}

\newtheorem{question}[theorem]{Question}

\newcommand{\betrag}[1]{\vert{#1}\vert}

\newcommand{\dom}[1]{{{\rm{dom}}(#1)}}
\newcommand{\crit}[1]{{{\rm{crit}}\left({#1}\right)}}
\newcommand{\cof}[1]{{{\rm{cof}}(#1)}}
\newcommand{\otp}[1]{{{\rm{otp}}\left(#1\right)}}
\newcommand{\ran}[1]{{{\rm{ran}}(#1)}}

\newcommand{\Hull}{{{\rm{Hull}}}}

\newcommand{\length}[1]{{\rm{lh}}({#1})}
\newcommand{\POT}[1]{{\mathcal{P}}({#1})}

\newcommand{\map}[3]{{#1}:{#2}\longrightarrow{#3}}
\newcommand{\Map}[5]{{#1}:{#2}\longrightarrow{#3};~{#4}\longmapsto{#5}}

\newcommand{\Set}[2]{\{{#1}~\vert~{#2}\}}
\newcommand{\seq}[2]{\langle{#1}~\vert~{#2}\rangle}

\newcommand{\goedel}[2]{{\prec}{#1},{#2}{\succ}}

\newcommand{\anf}[1]{{\text{``}\hspace{0.3ex}{#1}\hspace{0.3ex}\text{''}}}

\newcommand{\HH}[1]{{\rm{H}}(#1)}
\newcommand{\Ult}[2]{{\mathrm{Ult}}({#1},{#2})}

\newcommand{\Add}[2]{{\rm{Add}}({#1},{#2})}
\newcommand{\Col}[2]{{\rm{Col}}({#1},{#2})}

\newcommand{\id}{{\rm{id}}}
\newcommand{\Lim}{{\rm{Lim}}}
\newcommand{\On}{{\rm{On}}}

\newcommand{\Reg}{\mathrm{Reg}}
\newcommand{\LL}{{\rm{L}}}
\newcommand{\ZF}{{\rm{ZF}}}
\newcommand{\ZFC}{{\rm{ZFC}}}

\newcommand{\AD}{{\rm{AD}}}

\newcommand{\DC}{{\rm{DC}}}
\newcommand{\PFA}{{\rm{PFA}}}
\newcommand{\BPFA}{{\rm{BPFA}}}

\newcommand{\can}{\text{can}}

\newcommand{\OD}{{\rm{OD}}}
\newcommand{\HOD}{{\rm{HOD}}}

\newcommand{\PPP}{{\mathbb{P}}}

\newcommand{\RRR}{{\mathbb{R}}}

\newcommand{\KK}{{\rm{K}}}

\newcommand{\VV}{{\rm{V}}}

\newcommand{\calE}{\mathcal{E}}

\newcommand{\calG}{\mathcal{G}}

\newcommand{\calO}{\mathcal{O}}

\newcommand{\calR}{\mathcal{R}}

\newcommand{\calW}{\mathcal{W}}

 \newcommand{\CH}{{\rm{CH}}}
 \newcommand{\GCH}{{\rm{GCH}}}

\title[$\Sigma_1$-definability at higher cardinals]{$\mathbf{\Sigma}_1$-definability at higher cardinals: Thin sets, almost disjoint families and long well-orders}

\author{Philipp L\"ucke}
\address{Departament de Matem\`atiques i Inform\`atica, Universitat de Barcelona. Gran Via de les Corts Catalanes, 585, 08007 Barcelona, Spain.}
\email{philipp.luecke@ub.edu}

\author{Sandra M\"uller}
\address{Institut f\"ur Diskrete Mathematik und Geometrie, TU Wien, Wiedner Hauptstra{\ss}e 8-10/104, 1040 Wien, Austria.}
\email{sandra.mueller@tuwien.ac.at}

\thanks{The authors would like to thank William Chan for the permission to include his proof of Theorem \ref{theorem:AD-MAD-omega_1} in this paper. 
 In addition, the authors are thankful to the anonymous referees for the careful reading of the manuscript and several helpful comments. 
   This project has received funding from the European Union’s Horizon 2020 research and innovation programme under the Marie Sk{\l}odowska-Curie grant agreement No 842082 of the first author (Project \emph{SAIFIA: Strong Axioms of Infinity -- Frameworks, Interactions and Applications}). 
   The second author gratefully acknowledges funding from L'OR\'{E}AL Austria, in collaboration with the Austrian UNESCO Commission and in cooperation with the Austrian Academy of Sciences - Fellowship \emph{Determinacy and Large Cardinals}.
   Furthermore, the second author was supported by FWF Elise Richter grant number V844.  
}

\begin{document}

\begin{abstract}
 Given an uncountable cardinal $\kappa$, we consider the question of whether subsets of the power set of $\kappa$ that are usually constructed with the help of the Axiom of Choice are definable by $\Sigma_1$-formulas that only use the cardinal $\kappa$ and sets of hereditary cardinality less than $\kappa$ as parameters. 
 For  limits of measurable cardinals, we prove a \emph{perfect set theorem} for sets definable in this way and use it to  generalize two classical non-definability results to higher cardinals. 
 First, we show that a classical result of Mathias on the complexity of  maximal almost disjoint families of sets of natural numbers can be generalized to measurable  limits of measurables. 
 Second, we prove that for a limit of countably many measurable cardinals, the existence of a simply definable well-ordering of subsets of $\kappa$ of length at least $\kappa^+$ implies the existence of a projective well-ordering of the reals. 
 In addition, we determine the exact consistency strength of the non-existence of $\Sigma_1$-definitions of certain  objects at singular strong limit cardinals. 
 Finally, we show that both  large cardinal assumptions and forcing axioms cause analogs of these statements to hold at the first uncountable cardinal $\omega_1$. 
\end{abstract}

\maketitle



\section{Introduction}

Mathematical objects whose existence are usually proved with the \emph{Axiom of Choice} are often referred to as \emph{pathological sets}.  
 Important examples of such objects are Hamel bases of the vector space of real numbers over the field of rational numbers, non-principal ultrafilters on infinite sets and bistationary ({i.e.} stationary and costationary) subsets of uncountable regular cardinals. 
For many types of pathological sets of real numbers, it is possible to use results from descriptive set theory to show that these objects 
cannot be defined by simple formulas in second-order arithmetic. 
Moreover, many canonical extensions of the axioms of $\ZFC$ prove that these objects are not definable in second-order arithmetic at all and this implication is often viewed as a desirable feature of such extensions, because it allows us to clearly separate pathological sets of real numbers from the  explicitly constructed sets of reals.

In this paper, we study  the \emph{set-theoretic} definability of pathological sets of higher cardinalities. 
More specifically, we aim to generalize classical non-definability results for sets of real numbers to subsets of the power set $\POT{\kappa}$ of an uncountable cardinal $\kappa$ that are definable by $\Sigma_1$-formulas\footnote{See {\cite[p. 5]{MR1994835}} for the definition of the \emph{Levy hierarchy} of formulas.} with parameters in $\HH{\kappa}\cup\{\kappa\}$. 
This bound on the  complexity of the used formulas is motivated by the observation that the assumption $\VV=\HOD$ implies the $\Sigma_2$-definability of various pathological sets (see {\cite[Proposition 3.9]{Sigma1Partitions}}) and this assumption is compatible with many canonical extensions of $\ZFC$. 
The restriction of the  set of parameters is motivated by the existence of highly potent coding forcings at uncountable regular cardinals $\kappa$ that can be used to make highly pathological subsets of $\POT{\kappa}$ definable by a $\Sigma_1$-formula with parameters in $\HH{\kappa^+}$. 
For example, the results of {\cite[Section 3]{MR2987148}} show that, if $\kappa$ is an uncountable cardinal satisfying $\kappa=\kappa^{{<}\kappa}$ and $A$ is a subset $\POT{\kappa}$, then, in some cofinality-preserving forcing extension $\VV[G]$  of the ground model $\VV$, the sets $A$ and  $\POT{\kappa}^{\VV[G]}\setminus A$ are definable by $\Sigma_1$-formulas with parameters in $\HH{\kappa^+}$. 
Moreover, the main result of \cite{HL} shows that for every cardinal $\kappa$ with these properties, there is a cofinality-preserving forcing extension in which a well-ordering of $\POT{\kappa}$ is definable by a $\Sigma_1$-formula with parameters in $\HH{\kappa^+}$, and the results of \cite{MR3320593} show that various large cardinal properties of $\kappa$ can be preserved by such coding forcings. 
Finally, results of Caicedo and Veli{\v{c}}kovi{\'c} in \cite{MR2231126} show that the \emph{Bounded Proper Forcing Axiom $\BPFA$} outright implies the existence of a well-ordering of $\POT{\omega_1}$ that is definable by a $\Sigma_1$-formula with parameters in $\HH{\omega_2}$ (see the proof of Proposition \ref{proposition:SimplyDefOmega2} below).

 Previous work in this direction (see  \cite{Sigma1Partitions}, \cite{MR3694344} and \cite{welch_2020}) has already provided important examples that show that we can achieve the above aim when we work in one of the following scenarios: 
 \begin{itemize}
     \item The cardinal $\kappa$ is a limit of  cardinals possessing certain large cardinal properties, like measurability. 
     
     \item The cardinal $\kappa$ is the first uncountable cardinal $\omega_1$ and either certain large cardinals exist above $\kappa$ or strong forcing axioms hold. 
 \end{itemize}
 
In the following, we will derive structural results for simply definable sets  that will allow us to prove the non-definability of  several types of pathological sets in the above settings. These implications can again be seen as desirable features of the corresponding axiom systems. Moreover, for most of our results about  singular limits of large cardinals, we prove that the used large cardinal assumption is optimal for the corresponding non-definability statement at singular cardinals.

The starting point of our work is a \emph{perfect set theorem} for $\Sigma_1$-definable sets at limits of measurable cardinals. 
 In order to formulate this result, we generalize some basic topological concepts to higher function spaces and power sets. 
 Given a cardinal $\kappa>0$ and an infinite cardinal $\mu$, we equip the set ${}^\mu\kappa$ of all functions from $\mu$ to $\kappa$ with the topology whose basic open sets consists of all functions that extend a given function $\map{s}{\xi}{\kappa}$ with $\xi<\mu$. 
 In the same way, we equip the power set $\POT{\nu}$ of an infinite cardinal $\nu$ with the topology whose basic open sets consists of all subsets of $\nu$ whose intersection with a given ordinal $\eta<\nu$ is equal to a fixed subset of $\eta$. 
 We then say that an injection $\map{\iota}{{}^\mu\kappa}{\POT{\nu}}$ is a \emph{perfect embedding} if it induces a homeomorphism between ${}^\mu\kappa$ and the subspace $\ran{\iota}$ of $\POT{\nu}$.  
 The following result now shows that, analogously to the perfect set property of analytic sets of reals, simply definable \emph{thin} sets of subsets of limits of measurable cardinals have small cardinality.

\begin{theorem}\label{theorem:PerfectSubset}
 Let $\kappa$ be a limit of measurable cardinals and let $D$ be a subset of $\POT{\kappa}$ that is definable by a $\Sigma_1$-formula with parameters in $\HH{\kappa}\cup\{\kappa\}$.  
 If $D$ has cardinality greater than $\kappa$, then there is a perfect embedding $\map{\iota}{{}^{\cof{\kappa}}\kappa}{\POT{\kappa}}$ with $\ran{\iota}\subseteq D$. 
\end{theorem}

In the case of singular limits of measurable cardinals, we will use \emph{core model theory} developed in \cite{MR926749} and, for example, in \cite{Ze02} to show that the  consistency strength of the assumption of this theorem is optimal for its conclusion.

\begin{theorem}\label{theorem:PerfectSubsetOptimal}
Let $\kappa$ be a singular strong limit cardinal with the property that for every subset $D$ of $\POT{\kappa}$ of cardinality greater than $\kappa$ that is definable by a $\Sigma_1$-formula with parameters in $\HH{\kappa}\cup\{\kappa\}$, there is a perfect embedding $\map{\iota}{{}^{\cof{\kappa}}\kappa}{\POT{\kappa}}$ with $\ran{\iota}\subseteq D$. 
Then there is an inner model with a sequence of measurable cardinals of length $\cof{\kappa}$. 
\end{theorem}

The next type of pathological sets that we will study in this paper are \emph{almost disjoint families} of large cardinalities. 
Given an infinite cardinal $\kappa$, a set  $A$ of unbounded subsets of $\kappa$ is an \emph{almost disjoint family in $\POT{\kappa}$} if $x\cap y$ is bounded in $\kappa$ for all distinct $x,y\in A$. 
In addition, we say that such a family $A$ is \emph{maximal} if for every unbounded subset $x$ of $\kappa$, there exists $y\in A$ with the property that $x\cap y$ is unbounded in $\kappa$.  %
Motivated by a classical result of Mathias in \cite{MR491197} that shows that all analytic maximal almost disjoint families in $\POT{\omega}$ are finite and many additional influential results on maximal almost disjoint families by Mathias, A. Miller, Törnquist, Horowitz and Shelah, Neeman and Norwood, Bakke-Haga, Fischer, Schrittesser, Weinert, and others (see \cite{doi:10.1142/S0219061321500264,  FISCHER2021102909, MR3928385, MR491197, MR983001, MR3835078,  MR4012549, MR3156517}), we will use the techniques developed in the proof of Theorem \ref{theorem:PerfectSubset} to prove that, if a cardinal $\kappa$ possesses sufficiently strong large cardinal properties, then every simply definable almost disjoint family in $\POT{\kappa}$ has cardinality at most $\kappa$. In particular, by a simple diagonalization argument, all simply definable maximal almost disjoint families in $\POT{\kappa}$ have cardinality less than $\kappa$ in this case.

In order to reduce the large cardinal assumptions used in our arguments, we recall the notion of \emph{iterable cardinals}, introduced by Sharpe and Welch in \cite{MR2817562} and studied extensively in \cite{MR2830435}. 
 An uncountable cardinal $\kappa$ is \emph{iterable} if for every subset $x$ of $\kappa$, there exists a transitive model $M$ of $\ZFC^-$ of cardinality $\kappa$ with $\kappa,x\in M$  and a weakly amenable $M$-ultrafilter $U$ on $\kappa$ such that the structure $\langle M,U\rangle$ is iterable. 
 Note that all iterable cardinals are weakly compact and  all Ramsey cardinals are iterable (see, for example,  {\cite[Theorem 1.3]{MR2830415}}). In particular, all measurable limits of measurable cardinals satisfy the assumptions of the following result.

\begin{theorem}\label{theorem:AlmostDisjoint}
 Let $\kappa$ be an iterable cardinal that is a limit of measurable cardinals and let $A$ be a subset of $\POT{\kappa}$ that is definable by a $\Sigma_1$-formula with parameters in $\HH{\kappa}\cup\{\kappa\}$.  
  If $A$ has cardinality greater than $\kappa$, then there exist distinct $x,y\in A$ with the property that  $x\cap y$ is unbounded in $\kappa$. 
\end{theorem}

The third type of pathological sets studied in this paper are \emph{long well-orders}, {i.e.} well-orderings of subsets of the power set $\POT{\kappa}$ of an infinite cardinal $\kappa$ of order-type at least $\kappa^+$. 
The study of the definability of these objects is motivated by the classical fact that \emph{Projective Determinacy} implies that all well-orderings  definable in second-order arithmetic have countable length. 
In the case of limits of measurable cardinals $\kappa$, it is possible to use arguments contained in the proof of {\cite[Lemma 1.3]{MR3845129}} to show that for every well-ordering of $\kappa$, the collection of proper initial segments of the given order is not definable by a $\Sigma_1$-formula with parameters in $\HH{\kappa}\cup\{\kappa\}$. 
In Section \ref{section:LongWOSingular} below, we will show that it is possible to use  classical results of Dehornoy in \cite{DEHORNOY1978109} to show that for all such limits $\kappa$, no well-ordering of $\POT{\kappa}$ is definable in the above way (see Corollary \ref{corollary:NoSigma1WOorGoodLong}). 
We will then proceed by using ideas  from the proof of Theorem \ref{theorem:PerfectSubset} to prove results about well-orderings whose domain is a large proper subset of $\POT{\kappa}$. 
 The following theorem provides a scenario in which such orders have no simple definition.

\begin{theorem}\label{MainTheorem:LimitMeasurables}
  Let $\kappa$ be a cardinal of countable cofinality that is a limit of measurable cardinals. 
  If there exists a well-ordering of a subset of $\POT{\kappa}$ of cardinality greater than $\kappa$ that is definable by a $\Sigma_1$-formula with parameter $\kappa$, then there is a $\mathbf{\Sigma}^1_3$-well-ordering of the reals.  
\end{theorem}

In addition, the theory developed in this paper allow us to determine the exact consistency strength of the non-existence of $\Sigma_1$-definable long well-orderings of subsets of a singular strong limit cardinal of countable cofinality. 
The following theorem is proven by combining our techniques with results about \emph{short core models} from \cite{MR926749} in one direction and \emph{diagonal Prikry forcing} in the other direction.

\begin{theorem}\label{theorem:EquiLongWO}
  The following statements are equiconsistent over $\ZFC$: 
  \begin{enumerate}
      \item There exist infinitely many measurable cardinals. 
      
      \item There exists a singular cardinal $\kappa$ with the property that no well-ordering of a subset of $\POT{\kappa}$ of cardinality greater than $\kappa$ is definable by a $\Sigma_1$-formula with parameters in $\HH{\kappa}\cup\{\kappa\}$. 
  \end{enumerate}
\end{theorem}

We now continue by considering analogues of the above results for pathological sets consisting of subsets of the first uncountable cardinal $\omega_1$. 
Using results of Woodin in \cite{MR1713438}, a perfect subset theorem for subsets of $\POT{\omega_1}$ definable by a $\Sigma_1$-formula with parameters in $\HH{\aleph_1}\cup\{\omega_1\}$ is provided by {\cite[Theorem 4.9]{MR3694344}} that shows that if the non-stationary ideal on $\omega_1$ is saturated and there is a measurable cardinal, then every such subset either contains a continuous image of ${}^{\omega_1}\omega_1$ or is a subset of $\LL(\RRR)$. 
 As observed in \cite{MR3694344}, it is, in general, not possible to strengthen the second  alternative to state that the given set has cardinality at most $\aleph_1$, because the failure of $\CH$ implies that $\Set{x\in{}^{\omega_1}\omega_1}{\forall\alpha<\omega_1 ~ x(\omega+\alpha)=0}$ is a subset of ${}^{\omega_1}\omega_1$ of cardinality greater than $\aleph_1$ that is definable by a $\Sigma_1$-formula with parameter $\omega_1$ and does not contain a perfect subset. 
 The following result now shows that analogs of Theorems \ref{theorem:AlmostDisjoint} and \ref{MainTheorem:LimitMeasurables} for $\omega_1$ follow both from strong large cardinal assumptions and the validity of strong forcing axioms.

\begin{theorem}\label{theorem:ResultsOmega1}
 Assume that either there is a measurable cardinal above infinitely many Woodin cardinals or Woodin's Axiom $(*)$ holds. 
 \begin{enumerate}
     \item\label{item:Omega1-1} No well-ordering of a subset of $\POT{\omega_1}$ of cardinality greater than $\aleph_1$ is definable by a $\Sigma_1$-formula  with parameters in $\HH{\aleph_1}\cup\{\omega_1\}$. 
     
     \item\label{item:Omega1-2} If $A$ is a set of cardinality greater than $\aleph_1$ that consists of unbounded subsets of $\omega_1$ and is definable by a $\Sigma_1$-formula with parameters in $\HH{\aleph_1}\cup\{\omega_1\}$, then there exist distinct $x,y\in A$ with the property that $x\cap y$ is unbounded in $\omega_1$. 
     \end{enumerate}
\end{theorem}

We will end this paper by observing that the above results cannot be generalized from $\omega_1$ to $\omega_2$. More specifically, we will show that all large cardinal assumptions are compatible with the existence of an almost disjoint family of cardinality $2^{\aleph_2}$ in $\POT{\omega_2}$ that is definable by a $\Sigma_1$-formula with parameter $\omega_2$ (see Proposition \ref{proposition:SimplyDefOmega2} below).


\section{A perfect subset theorem for limits of measurable cardinals}\label{section:PSPlimitsMeasurables}

 In this section, we prove Theorem \ref{theorem:PerfectSubset} with the help of iterated ultrapowers. 
 Throughout this paper, we will use two types of iterated ultrapower constructions for transitive $\ZFC^-$-models $M$: iterated ultrapowers of $M$ constructed using a single weakly amenable $M$-ultrafilter (as defined in {\cite[Section 19]{MR1994835}}) and iterated ultrapowers of $M$ constructed using a set of normal measures in $M$ (as defined in {\cite[Section 3]{MR3411035}}). 
 In order to establish  notation, we now discuss some details of the second type of construction.  
 Given a transitive model $M$ of $\ZFC^-$ and $\calE\in M$ with $$M\models\anf{\textit{$\calE$ consists of normal ultrafilters on measurable cardinals}},$$ a \emph{linear iteration of $\langle M,\calE\rangle$} is a sequence $I=\seq{U_\alpha}{\alpha<\lambda}$ with $\lambda>0$ and the property that there exists a directed system $$\langle\seq{M_\alpha}{\alpha<\lambda},\seq{\map{i_{\alpha,\beta}}{M_\alpha}{M_\beta}}{\alpha\leq\beta<\lambda}\rangle$$ of transitive $\ZFC^-$-models and elementary embeddings such that the following statements hold: 
 \begin{enumerate}
  \item $M_0=M$. 
  
  \item $U_\alpha\in i_{0,\alpha}(\calE)$  for all $\alpha<\lambda$. 
  
  \item If $\alpha$ is an ordinal with $\alpha+1<\lambda$, then $M_{\alpha+1}$ is the  (transitive collapse) of the ultrapower of $M_\alpha$ constructed using $U_\alpha$ and $i_{\alpha,\alpha+1}$ is the corresponding ultrapower embedding. 
  
  \item If $\eta<\lambda$ is a limit ordinal, then $\langle M_\eta,\seq{i_{\alpha,\eta}}{\alpha<\eta}\rangle$ is a direct limit of the directed system $\langle\seq{M_\alpha}{\alpha<\eta},\seq{\map{i_{\alpha,\beta}}{M_\alpha}{M_\beta}}{\alpha\leq\beta<\eta}\rangle$. 
 \end{enumerate}
 The ordinal $\lambda$ is then called the \emph{length of $I$} and we use $\length{I}$ to refer to this ordinal.

 It is easy to see that the above system is uniquely determined by the sequence $I$ and therefore we write $U^I_\alpha=U_\alpha$, $M^I_\alpha=M_\alpha$ and $i^I_{\alpha,\beta}=i_{\alpha,\beta}$ for all $\alpha\leq\beta<\length{I}$. We then let $\langle M^I_\infty,\seq{\map{i^I_{\alpha,\infty}}{M^I_\alpha}{M^I_\infty}}{\alpha<\length{I}}\rangle$ denote  the direct limit of the above system and, if the model $M^I_\infty$ is well-founded, then we identify it with its transitive collapse. 
 Finally, the pair $\langle M,\calE\rangle$ is called \emph{linearly iterable} if the model $M^I_\infty$ is well-founded for every linear iteration $I$ of $\langle M,\calE\rangle$. Note that {\cite[Theorem 3.3]{MR3411035}} shows that, if every element of $\calE$ is $\sigma$-complete in $\VV$, then the pair is  $\langle M,\calE\rangle$ is linearly iterable. In particular, the pair $\langle\VV,\calE\rangle$ is linearly iterable for every set $\calE$ of normal ultrafilters.

 Given a transitive $ZFC^-$-model $M$ and $U\in M$ with $$M\models\anf{\textit{$U$ is a normal ultrafilters on a measurable cardinal}},$$ the set $U$ is a weakly amenable $M$-ultrafilter and the pair $\langle M,U\rangle$ is iterable (in the sense of {\cite[Section 19]{MR1994835}}) if and only if the pair $\langle M,\{U\}\rangle$ is linearly iterable (in the above sense). Moreover, if $\langle M,U\rangle$ is iterable, $$\langle\seq{M_\alpha}{\alpha\in\On},\seq{\map{j_{\alpha,\beta}}{M_\alpha}{M_\beta}}{\alpha\leq\beta\in\On}\rangle$$ denotes the  iteration of $\langle M,U\rangle$ (as defined in {\cite[Section 19]{MR1994835}}) and $\lambda>0$ is an ordinal, then $\seq{j_{0,\alpha}(U)}{\alpha<\lambda}$ is the unique linear iteration $I(U,\lambda)$ of $\langle M,\{U\}\rangle$ of length $\lambda$ and we have $M^{I(U,\lambda)}_\alpha=M_\alpha$ and $i^{I(U,\lambda)}_{\alpha,\beta}=j_{\alpha,\beta}$ for all $\alpha\leq\beta<\lambda$.

 The following technical lemma about the existence of certain systems of linear iterations is the starting point of the proofs of most of the results about limits of measurable cardinals stated in the introduction:

\begin{lemma}\label{lemma:TechnicalLemmaIterationsTree}
 Let $\mu$ be an infinite regular cardinal, let $\kappa$ be a limit of measurable cardinals with $\cof{\kappa}=\mu$ and let $\calE$ denote the collection of all normal ultrafilters on cardinals smaller than $\kappa$. 
 Given an element    $z$ of $\HH{\kappa}$ and  a subset $D$ of $\POT{\kappa}$ of cardinality $\kappa^+$, there exists  
  \begin{itemize}
  
   \item an element $x$ of $D$, 
   
   \item a system $\seq{\nu_s}{s\in{}^{{<}\mu}\kappa}$ of inaccessible cardinals smaller than $\kappa$, 
       
   \item a system $\seq{\kappa_s}{s\in{}^{{<}\mu}\kappa}$ of measurable cardinals smaller than $\kappa$, 
      
   \item a system $\seq{U_s}{s\in{}^{{<}\mu}\kappa}$ of elements of $\calE$,  and 
   
   \item a system $\seq{I_s}{s\in{}^{{<}\mu}\kappa}$ of linear iterations of $\langle\VV,\calE\rangle$ of length less than $\kappa$
  \end{itemize}
 such that the following statements hold for all $s,t\in{}^{{<}\mu}\kappa$: 
  \begin{enumerate}   
   \item $z\in\HH{\nu_\emptyset}$ and $\mu<\kappa$ implies that $\mu<\nu_\emptyset$. 
     
   \item $U_s$ is an ultrafilter on $\kappa_s$.
   
   \item  $I_s$ is a linear iteration of $\langle\VV,\Set{U_{s\restriction\xi}}{\xi\in\dom{s}}\rangle$. 
   
   \item\label{item:Cofinal} The sequence $\seq{\min\Set{\kappa_s}{s\in{}^\xi\kappa}}{\xi<\mu}$ is cofinal in $\kappa$. 

  \item\label{item:LimitLength} If $I_s$ is non-trivial, then $\length{I_s}\in\Lim$. 

   \item\label{item:Incvreasing} If $s\subsetneq t$, then $\length{I_s}{}<\nu_s<\kappa_s<\nu_t$. 
   
   \item\label{item:CardinalsFixed} $i^{I_s}_{0,\infty}(\nu_s)=\nu_s$ and $i^{I_s}_{0,\infty}(\kappa_s)=\kappa_s$. 

  \item\label{item:ParametersFixed} $i^{I_s}_{0,\infty}(\mu)=\mu$, $i^{I_s}_{0,\infty}(\kappa)=\kappa$ and $i^{I_s}_{0,\infty}(z)=z$. 
         
   \item\label{item:Extensions} If $s\subseteq t$, then $\length{I_s}{}\leq\length{I_t}{}$ and $U^{I_s}_\alpha=U^{I_t}_\alpha$ for all $\alpha<\length{I_s}{}$.\footnote{Note that this directly implies that $M^{I_s}_\alpha=M^{I_t}_\alpha$ and $i^{I_s}_{\alpha,\beta}=i^{I_t}_{\alpha,\beta}$ holds for all $\alpha\leq\beta<\length{I_s}{}$. Moreover, if $1<\length{I_s}{}<\length{I_t}{}$, then  \eqref{item:LimitLength}  implies that $M^{I_s}_\infty=M^{I_t}_{\length{I_s}{}}$ and $i^{I_s}_{0,\infty}=i^{I_t}_{0,\length{I_s}{}}$. Finally, if $\length{I_s}{}=1<\length{I_t}{}$, then   $M^{I_s}_\infty=M^{I_t}_0$ and $i^{I_s}_{0,\infty}=\id_{M^{I_t}_0}$.} 
   
   \item\label{item:ProperInitialSegments} If $s\subseteq t$ with $\length{I_s}{}<\length{I_t}{}$, then  $\HH{\kappa_s}^{M^{I_s}_\infty}=\HH{\kappa_s}^{M^{I_t}_\infty}$ and $$i^{I_t}_{\length{I_s}{},\infty}\restriction\HH{\kappa_s}^{M^{I_s}_\infty} ~ = ~ \id_{\HH{\kappa_s}^{M^{I_s}_\infty}}.$$ 
   
   \item\label{item:Split} If $\xi\in\dom{s}\cap\dom{t}$ satisfying $s\restriction\xi=t\restriction\xi$ and $s(\xi)< t(\xi)$, 
    then $\nu_{s\restriction(\xi+1)}\leq\nu_{t\restriction(\xi+1)}$ and $$i^{I_s}_{0,\infty}(x)\cap \nu_{s\restriction(\xi+1)} ~ \neq ~ i^{I_t}_{0,\infty}(x)\cap\nu_{s\restriction(\xi+1)}.$$ 
  \end{enumerate}
\end{lemma}

\begin{proof}
 Pick a strictly increasing, cofinal sequence $\seq{\kappa_\xi}{\xi<\mu}$ of measurable cardinals in $\kappa$ with the property that $\mu<\kappa$ implies that $\mu<\kappa_0$. 
  Given $\xi<\mu$, fix a normal ultrafilter $U_\xi$ on $\kappa_\xi$ and let $$\langle\seq{N^\xi_\alpha}{\alpha\in\On},\seq{\map{j^\xi_{\alpha,\beta}}{N^\xi_\alpha}{N^\xi_\beta}}{\alpha\leq\beta\in\On}\rangle$$ denote the  iteration of $\langle\VV,U_\xi\rangle$. 
 Given $\xi<\mu$, we then have $j^\xi_{0,\kappa}(\kappa_\xi)=\kappa$ and  $j^\xi_{0,\alpha}(\kappa)=\kappa$ for all $\alpha<\kappa$.  
  In particular, we know that $\betrag{\POT{\kappa}^{N^\xi_\kappa}}=\kappa$ holds for all $\xi<\mu$. 
  Therefore, we can find $x\in D$ with $x\notin N^\xi_\kappa$ for all $\xi<\mu$. 
  Given $\xi<\mu$, we then have $x\neq j^\xi_{0,\kappa}(x)\cap\kappa$ and hence we know that 
  \begin{equation}\label{equation:PreSplit}
   x\cap j^\xi_{0,\lambda}(\kappa_\xi) ~ \neq ~ j^\xi_{0,\lambda}(x\cap\kappa_\xi)
  \end{equation} 
  holds for all sufficiently large $\lambda<\kappa$.

  By earlier remarks, the pair $\langle\VV,\calE\rangle$ is linearly iterable. In the following, we inductively construct   systems with the properties listed above while also ensuring that for every $s\in{}^{{<}\mu}\kappa$,  there exists $\dom{s}\leq\xi<\mu$ with $\kappa_s=\kappa_\xi$ and $U_s=U_\xi$. 
  Note that this additional property will directly ensure that \eqref{item:Cofinal} holds in the end.

  First, we define  $I_\emptyset$ to be the trivial linear iteration of $\langle\VV,\calE\rangle$. Moreover,  we pick some inaccessible cardinal $\nu_\emptyset<\kappa$ such that  $z\in\HH{\nu_\emptyset}$ and $\mu<\kappa$ implies $\mu<\nu_\emptyset$.

 Next, assume that $\zeta\in\Lim\cap\mu$ and the objects $\nu_t$, $\kappa_t$, $U_t$ and $I_t$ are defined for all $t\in{}^{{<}\zeta}\kappa$. 
  Fix $s\in{}^\zeta\kappa$ and define $I_s$ to be the unique linear iteration of $\langle\VV,\Set{U_{s\restriction\eta}}{\eta<\zeta}\rangle$ of length $\sup_{\eta<\zeta}\length{I_{s\restriction\eta}}{}<\kappa$ with the property that $U^{I_s}_\alpha=U^{I_{s\restriction\eta}}_\alpha$ holds for all $\eta<\zeta$ and $\alpha<\length{I_{s\restriction\eta}}{}$. 
 In addition, define $\nu_s$ to be an inaccessible cardinal smaller than $\kappa$ and bigger than both $\sup_{\eta<\zeta}\kappa_{s\restriction\eta}$ and $\length{I_s}{}$. 
 This setup ensures that $\length{I_s}>1$ implies that $\length{I_s}\in\Lim$, and therefore we know that  \eqref{item:LimitLength} holds. 
  Moreover, these definitions directly ensure that the relevant parts of \eqref{item:Incvreasing} and \eqref{item:CardinalsFixed} hold in this case. 
  In addition, since $\length{I_s}<\nu_s$ and $I_s$ only makes use of ultrafilter on cardinals contained in the interval $(\nu_\emptyset,\nu_s)$, 
  the fact that the cofinality of $\kappa$ is not contained in this interval allows us to conclude that \eqref{item:ParametersFixed} holds in this case. 
  Next, notice that our construction directly ensures that \eqref{item:Extensions} holds in this case. 
  Moreover, if $\eta<\zeta$ with $\length{I_{s\restriction\eta}}{}<\length{I_s}{}$, then the fact that \eqref{item:LimitLength} and \eqref{item:ProperInitialSegments} hold for all $\eta<\rho<\zeta$ ensures that $$\HH{\kappa_{s\restriction\eta}}^{M^{I_{s\restriction\eta}}_\infty} ~ = ~ \HH{\kappa_{s\restriction\eta}}^{M^{I_s}_\infty}$$ and $$i^{I_s}_{\length{I_{s\restriction\eta}}{},\infty}\restriction\HH{\kappa_{s\restriction\eta}}^{M^{I_s}_\infty} ~ = ~ \id_{\HH{\kappa_{s\restriction\eta}}^{M^{I_s}_\infty}}.$$  
  By the definition of $I_s$, this shows that \eqref{item:ProperInitialSegments} also holds in this case. 
  Finally, pick $t\in{}^{{<}\mu}\kappa$ with $\dom{t}\leq\zeta$ and $\xi\in\dom{t}$ with $s\restriction\xi=t\restriction\xi$ and $s(\xi)\neq t(\xi)$. Set $\rho=\min(\nu_{s\restriction(\xi+1)},\nu_{t\restriction(\xi+1)})$. 
  Since we know that $\rho<\min(\kappa_{s\restriction(\xi+1)},\kappa_{t\restriction(\xi+1)})$, we can use \eqref{item:ProperInitialSegments} and \eqref{item:Split} to show that $$i^{I_s}_{0,\infty}(x)\cap\rho ~ = ~ i^{I_{s\restriction(\xi+1)}}_{0,\infty}(x)\cap\rho ~ \neq ~ i^{I_{t\restriction(\xi+1)}}_{0,\infty}(x)\cap\rho ~ = ~ i^{I_t}_{0,\infty}(x)\cap\rho.$$
  By the properties of $\nu_{s\restriction(\xi+1)}$ ensured by our induction hypothesis, these computations show that \eqref{item:Split} also holds in this case. 
   
  \medskip
  
  Now,  assume that $\zeta<\mu$ and the objects $\nu_t$, $\kappa_u$, $U_u$ and $I_t$ are defined for all $t,u\in{}^{{<}\mu}\kappa$ with $\dom{t}\leq\zeta$ and $\dom{u}<\zeta$. 
  Fix $s\in{}^\zeta\kappa$ and pick $\zeta\leq\xi<\mu$ with $\kappa_\xi>\nu_s$. 
   Set $\kappa_s=\kappa_\xi$ and $U_s=U_\xi$. 
   By \eqref{equation:PreSplit}, there exists a limit ordinal  $\kappa_s<\lambda<\kappa$ with the property that 
   \begin{equation}\label{equation:Splitting}
    x\cap j^\xi_{0,\lambda}(\kappa_s) ~ \neq ~ j^\xi_{0,\lambda}(x\cap\kappa_s).
   \end{equation} 
   Let $\seq{\lambda_\beta}{\beta<\kappa}$ denote the unique continuous sequence of ordinals with $\lambda_0=0$ and $\lambda_{\beta+1}=\lambda_\beta+j^\xi_{0,\lambda_\beta}(\lambda)$ for all $\beta<\lambda$. 
   Since $\kappa$ is a limit of inaccessible cardinals, we know that $\lambda_\beta<\kappa$ holds for all $\beta<\kappa$. 
   Given $\beta<\kappa$, define $I_{s^\frown\langle\beta\rangle}$ to be the unique linear iteration of $\langle\VV,\calE\rangle$ of length $\length{I_s}+i^{I_s}_{0,\infty}(\lambda_\beta)$ 
   with $U^{I_{s^\frown\langle\beta\rangle}}_\alpha=U^{I_s}_\alpha$ for all $\alpha<\length{I_s}{}$ and $U^{I_{s^\frown\langle\beta\rangle}}_\alpha=i^{I_{s^\frown\langle\beta\rangle}}_{0,\alpha}(U_s)$ for all $\length{I_s}{}\leq\alpha<\length{I_{s^\frown\langle\beta\rangle}}{}$. That means we linearly iterate $U_s$ on top of what we already have to obtain $I_{s^\frown\langle\beta\rangle}$.
  Moreover, for every $\beta<\kappa$, we define $\nu_{s^\frown\langle\beta\rangle}$ to be the least inaccessible cardinal greater than $\length{I_{s^\frown\langle\beta+1\rangle}}{}$. 
 These definitions then directly ensure that \eqref{item:LimitLength} and \eqref{item:Extensions} hold. 
 In addition, for all $\beta<\kappa$, we have $$\length{I_s} ~ < ~ \nu_s ~ < ~ \kappa_s ~ < ~ \lambda ~ \leq ~ \length{I_{s^\frown\langle\beta+1\rangle}} ~ < ~ \nu_{s^\frown\langle\beta\rangle}$$ and this can be used to conclude that $i^{I_s}_{0,\infty}(\kappa_s)=\kappa_s$ and $i^{I_{s^\frown\langle\beta\rangle}}_{0,\infty}(\nu_{s^\frown\langle\beta\rangle})=\nu_{s^\frown\langle\beta\rangle}$. 
  This shows that the relevant instances of \eqref{item:Incvreasing} and \eqref{item:CardinalsFixed}  hold in this case. 
 Moreover, the fact that all linear iterations of the form $I_{s^\frown\langle\beta\rangle}$ with $\beta<\kappa$ have length less than $\kappa$ and only make use of ultrafilters on cardinals contained in the interval $[\kappa_0,\kappa)$ directly implies that \eqref{item:ParametersFixed} holds in this case as $\mu < \kappa_0$ in case $\mu < \kappa$. 
  Next, notice that, if $0<\beta<\kappa$ and $\length{I_s}{}\leq\alpha<\length{I_{s^\frown\langle\beta\rangle}}{}$, then our construction ensures that  $$\HH{\kappa_s}^{M^{I_s}_\infty} ~ = ~ \HH{\kappa_s}^{M^{I_{s^\frown\langle\beta\rangle}}_{\length{I_s}{}}} ~ = ~ \HH{\kappa_s}^{M^{I_{s^\frown\langle\beta\rangle}}_\alpha}$$ and  $$i^{I_{s^\frown\langle\beta\rangle}}_{\length{I_s}{},\alpha}\restriction\HH{\kappa_s}^{M^{I_s}_\infty} ~ = ~ \id_{\HH{\kappa_s}^{M^{I_s}_\infty}}.$$ 
  This directly implies that \eqref{item:ProperInitialSegments} holds in this case. 
  Finally, fix $\beta<\gamma<\kappa$. 
  Then $\length{I_{s^\frown\langle\beta+1\rangle}}{}<\length{I_{s^\frown\langle\gamma+1\rangle}}{}$ and hence we know that $\nu_{s^\frown\langle\beta\rangle}\leq\nu_{s^\frown\langle\gamma\rangle}$.

  \begin{claim*}
   $i^{I_{s^\frown\langle\beta\rangle}}_{0,\infty}(x)\cap\nu_{s^\frown\langle\beta\rangle} ~ \neq ~ i^{I_{s^\frown\langle\gamma\rangle}}_{0,\infty}(x)\cap\nu_{s\frown\langle\beta\rangle}$.
  \end{claim*}
  
  \begin{proof}[Proof of the Claim]
  Let $$\langle\seq{N_\alpha}{\alpha\in\On},\seq{\map{j_{\alpha_0,\alpha_1}}{N_{\alpha_0}}{N_{\alpha_1}}}{\alpha_0\leq\alpha_1\in\On}\rangle$$ denote the  iteration of $\langle M^{I_s}_\infty,i^{I_s}_{0,\infty}(U_s)\rangle$. 
  Given $\delta<\kappa$ and $\alpha<i^{I_s}_{0,\infty}(\lambda_\delta)$, the definition of $I_{s^\frown\langle\delta\rangle}$  ensures that  the following statements hold: 
  \begin{itemize}
   \item $M^{I_{s^\frown\langle\delta\rangle}}_{\length{I_s}+\alpha}=N_\alpha$ and $M^{I_{s^\frown\langle\delta\rangle}}_\infty=N_{i^{I_s}_\infty(\lambda_\delta)}$. 
   
   \item $i^{I_{s^\frown\langle\delta\rangle}}_{0,\length{I_s}+\alpha}=j_{0,\alpha}\circ i^{I_s}_{0,\infty}$ and $i^{I_{s^\frown\langle\delta\rangle}}_{0,\infty}=j_{0,i^{I_s}_{0,\infty}(\lambda_\delta)}\circ i^{I_s}_{0,\infty}$. 
  \end{itemize}
  
 Now, set $M_*=M^{I_{s^\frown\langle\beta\rangle}}_\infty$,  $x_*=i^{I_{s^\frown\langle\beta\rangle}}_{0,\infty}(x)$, $\kappa_*=i^{I_{s^\frown\langle\beta\rangle}}_{0,\infty}(\kappa_s)$, $\lambda_*=i^{I_{s^\frown\langle\beta\rangle}}_{0,\infty}(\lambda)$ and $U_*=i^{I_{s^\frown\langle\beta\rangle}}_{0,\infty}(U_s)$. 
  Note that elementarity ensures that $$i_{0,\infty}^{I_s}(j^\xi_{0,\lambda_\beta}(\lambda)) ~ = ~ j_{0,i^{I_s}_{0,\infty}(\lambda_\beta)}(i^{I_s}_{0,\infty}(\lambda)) ~ = ~ i^{I_{s^\frown\langle\beta\rangle}}_{0,\infty}(\lambda) ~ = ~ \lambda_*$$ and this allows us to conclude that $$i^{I_s}_{0,\infty}(\lambda_\gamma) ~ \geq ~ i^{I_s}_{0,\infty}(\lambda_{\beta+1}) ~ = ~ i^{I_s}_{0,\infty}(\lambda_\beta+j^\xi_{0,\lambda_\beta}(\lambda)) ~ = ~ i^{I_s}_{0,\infty}(\lambda_\beta)+\lambda_*.$$
  In particular, we know that 
  \begin{equation}\label{equation:LongerThanLambda}
   \length{I_{s^\frown\langle\beta\rangle}}{}+\lambda_*\leq\length{I_{s^\frown\langle\gamma\rangle}}{}.
  \end{equation} 
  
 Now, define $$\langle\seq{N^*_\alpha}{\alpha\in\On},\seq{\map{j^*_{\alpha_0,\alpha_1}}{N^*_{\alpha_0}}{N^*_{\alpha_1}}}{\alpha_0\leq\alpha_1\in\On}\rangle$$ to be the  iteration of $\langle M_*,U_*\rangle$. 
  Given ordinals $\alpha_0\leq\alpha_1$, we then have $N^*_{\alpha_0}=N_{i^{I_s}_\infty(\lambda_\beta)+\alpha_0}$ and $j^*_{\alpha_0,\alpha_1}=j_{i^{I_s}_\infty(\lambda_\beta)+\alpha_0,i^{I_s}_\infty(\lambda_\beta)+\alpha_1}$. 
  In particular, we can use \eqref{equation:LongerThanLambda} to find an ordinal $\alpha\geq\lambda_*$ with $M^{I_{s^\frown\langle\gamma\rangle}}_\infty=N^*_\alpha$ and $i^{I_{s^\frown\langle\gamma\rangle}}_{0,\infty}=j^*_{0,\alpha}\circ i^{I_{s^\frown\langle\beta\rangle}}_{0,\infty}$.

  By elementarity, the inequality \eqref{equation:Splitting}  implies that $$x_*\cap j^*_{0,\lambda_*}(\kappa_*) ~ \neq ~ j^*_{0,\lambda_*}(x_*\cap\kappa_*).$$
  Moreover, our setup ensures that $$\nu_{s^\frown\langle\beta\rangle} ~ > ~ \length{I_{s^\frown\langle\beta+1\rangle}}{} ~ \geq ~ i^{I_s}_{0,\infty}(\lambda_{\beta+1}) ~ \geq ~ i^{I_s}_{0,\infty}(j^\xi_{0,\lambda_\beta}(\lambda)) ~ = ~ \lambda_* ~ > ~ \kappa_*$$ and $$\nu_{s^\frown\langle\beta\rangle} ~ = ~ j^*_{0,\lambda_*}(\nu_{s^\frown\langle\beta\rangle}) ~ > ~ j^*_{0,\lambda_*}(\kappa_*).$$  
  Since for $\alpha\geq\lambda_*$ as above $$i^{I_{s^\frown\langle\gamma\rangle}}_{0,\infty}(x) ~ = ~ j^*_{0,\alpha}(i^{I_{s^\frown\langle\beta\rangle}}_{0,\infty}(x)) ~ = ~ (j^*_{\lambda_*,\alpha}\circ j^*_{0,\lambda_*})(x_*)$$ and  $$j^*_{\lambda_*,\alpha}\restriction(j^*_{0,\lambda_*}(\kappa_*)) ~ = ~ \id_{j^*_{0,\lambda_*}(\kappa_*)},$$ 
  we know that 
  \begin{equation*}
    \begin{split}
      i^{I_{s^\frown\langle\gamma\rangle}}_{0,\infty}(x)\cap j^*_{0,\lambda_*}(\kappa_*) ~ & = ~ j^*_{0,\lambda_*}(x_*)\cap j^*_{0,\lambda_*}(\kappa_*) ~ = ~ j^*_{0,\lambda_*}(x_*\cap\kappa_*) \\
      & \neq ~ x_*\cap j^*_{0,\lambda_*}(\kappa_*) ~ = ~ i^{I_{s^\frown\langle\beta\rangle}}_{0,\infty}(x)\cap j^*_{0,\lambda_*}(\kappa_*)
    \end{split}
    \end{equation*}
    and the fact that $\nu_{s^\frown\langle\beta\rangle}> j^*_{0,\lambda_*}(\kappa_*)$ then yields the statement of the claim.  
  \end{proof}

  The above claim now shows that \eqref{item:Split} also holds in this case. This completes the proof of the lemma.  
\end{proof}

 We now extend the above construction to obtain linear iterations indexed by sequences of length equal to the cofinality of the given limit of measurable cardinals. In addition, we also allow these sequences to exist in small forcing extensions of the ground model.

\begin{lemma}\label{lemma:LongIterationsGenericExtensions}
 In the situation of Lemma \ref{lemma:TechnicalLemmaIterationsTree}, let $\lambda\leq\mu$ be a limit ordinal, let $\PPP$ be a partial order\footnote{Note that $\PPP$ is allowed  be the trivial partial order.} and let $G$ be $\PPP$-generic over $\VV$. 
 Given a function $c\in({}^\lambda\kappa)^{\VV[G]}$ with the property that all of its proper initial segments are contained in $\VV$, 
  we let $I_c$ denote the unique linear iteration of $\langle\VV,\Set{U_{c\restriction\xi}}{\xi<\lambda}\rangle$ of length $\sup_{\xi<\lambda}\length{I_{c\restriction\xi}}{}$ in $\VV[G]$ with $U^{I_c}_\alpha=U^{I_{c\restriction\xi}}_\alpha$ for all $\xi<\lambda$ and $\alpha<\length{I_{c\restriction\xi}}{}$. 
   
  If either $\PPP$ is an element of $\HH{\kappa_\emptyset}$ or forcing with $\PPP$ does not add bounded subsets of $\kappa$, 
  then the following statements hold in $\VV[G]$ for all functions $c,d\in{}^\lambda\kappa$ with the property that all of their proper initial segments are contained in $\VV$:  
 \begin{enumerate}
  \item $M^{I_c}_\infty$ is well-founded. 
  
  \item $i^{I_c}_{0,\infty}(\mu)=\mu$, $i^{I_c}_{0,\infty}(\kappa)=\kappa$ and $i^{I_c}_{0,\infty}(z)=z$. 
  
  \item\label{item:SplitLong} If $\xi<\lambda$ with $c\restriction\xi=d\restriction\xi$ and $c(\xi)\neq d(\xi)$, then $$i^{I_c}_{0,\infty}(x)\cap\kappa_{c\restriction\xi} ~ = ~ i^{I_d}_{0,\infty}(x)\cap\kappa_{c\restriction\xi}$$ and there is  $\rho<\min(\kappa_{c\restriction(\xi+1)},\kappa_{d\restriction(\xi+1)})$ with 
  \begin{equation}\label{equation:SplitLong}
   i^{I_c}_{0,\infty}(x)\cap\rho ~ \neq ~ i^{I_d}_{0,\infty}(x)\cap\rho. 
  \end{equation}
 \end{enumerate}
\end{lemma}

\begin{proof}
  Work in $\VV[G]$, pick a function $c\in{}^\lambda\kappa$ with the desired properties and define $I_c$ as above. 
  If $\PPP$ is contained in $\HH{\kappa_\emptyset}$, then we can use the  \emph{L{\'e}vy--Solovay Theorem} to show that for all $\xi<\lambda$, the set $\Set{B\in\POT{\kappa_\xi}}{\exists A\in U_{c\restriction\xi} ~ A\subseteq B}$ is a normal ultrafilter on $\kappa_\xi$ and therefore  we know that  $U_{c\restriction\xi}$ itself is a $\sigma$-complete $\VV$-ultrafilter. 
  Since the same conclusion obviously holds true  if forcing with $\PPP$ does not add bounded subsets of $\kappa$, we can apply  {\cite[Theorem 3.3]{MR3411035}}  to conclude that the pair $\langle\VV,\Set{U_{c\restriction\xi}}{\xi<\lambda}\rangle$ is linearly iterable and therefore we know that $M^{I_c}_\infty$ is well-founded. 
  
 Next, since \eqref{item:ProperInitialSegments} of Lemma \ref{lemma:TechnicalLemmaIterationsTree} ensures that $$\HH{\kappa_{c\restriction\xi}}^{M^{I_c}_{\length{I_{c\restriction\xi}}{}}} ~ = ~ \HH{\kappa_{c\restriction\xi}}^{M^{I_{c\restriction\xi}}_\infty} ~ = ~ \HH{\kappa_{c\restriction\xi}}^{M^{I_{c\restriction\zeta}}_\infty} ~ = ~ \HH{\kappa_{c\restriction\xi}}^{M^{I_c}_{\length{I_{c\restriction\zeta}}{}}}$$ 
 and 
 $$i^{I_c}_{\length{I_{c\restriction\xi}}{},\length{I_{c\restriction\zeta}}{}}\restriction\HH{\kappa_{c\restriction\xi}}^{M^{I_{c\restriction\xi}}_\infty} ~ = ~ i^{I_{c\restriction\zeta}}_{\length{I_{c\restriction\xi}}{},\infty}\restriction\HH{\kappa_{c\restriction\xi}}^{M^{I_{c\restriction\xi}}_\infty} ~ = ~ \id_{\HH{\kappa_{c\restriction\xi}}^{M^{I_{c\restriction\xi}}_\infty}}$$ hold for all $\xi<\zeta<\lambda$ with $\length{I_{c\restriction\xi}}{}<\length{I_{c\restriction\zeta}}{}<\length{I_c}{}$, we know that $$\HH{\kappa_{c\restriction\xi}}^{M^{I_{c\restriction\xi}}_\infty} ~ = ~ \HH{\kappa_{c\restriction\xi}}^{M^{I_c}_\infty}$$ and 
 \begin{equation}\label{equation:TailsIdentityInLongIterations}
  i^{I_c}_{\length{I_{c\restriction\xi}}{},\infty}\restriction\HH{\kappa_{c\restriction\xi}}^{M^{I_{c\restriction\xi}}_\infty} ~ = ~ \id_{\HH{\kappa_{c\restriction\xi}}^{M^{I_{c\restriction\xi}}_\infty}}
 \end{equation}
 hold for all $\xi<\lambda$ with $\length{I_{c\restriction\xi}}{}<\length{I_c}{}$. 
 In particular, it follows that $i^{I_c}_{0,\infty}(z)=z$ and, if  $\mu<\kappa$, then $i^{I_c}_{0,\infty}(\mu)=\mu$. 
 In addition,  for all $\xi<\lambda$ with the property that $\length{I_{c\restriction\xi}}{}<\length{I_c}{}$,  we have $i^{I_c}_{0,\length{I_{c\restriction\xi}}{}}=i^{I_{c\restriction\xi}}_{0,\infty}$ and therefore   \begin{equation}\label{equation:LongIterationTailDontMove}
   i^{I_c}_{0,\length{I_{c\restriction\xi}}{}}(\alpha) ~ < ~ i^{I_c}_{0,\length{I_{c\restriction\xi}}{}}(\kappa_{c\restriction\xi}) ~ = ~ i^{I_{c\restriction\xi}}_{0,\infty}(\kappa_{c\restriction\xi}) ~ = ~ \kappa_{c\restriction\xi}.
  \end{equation}
 for all $\alpha<\kappa_{c\restriction\xi}$. 
  In particular, a combination of \eqref{equation:TailsIdentityInLongIterations} and \eqref{equation:LongIterationTailDontMove} allows us to conclude that $i^{I_c}_{0,\infty}[\kappa_{c\restriction\xi}]\subseteq\kappa_{c\restriction\xi}$ holds for all $\xi<\lambda$. 
 If the sequence $\seq{\kappa_{c\restriction\xi}}{\xi<\lambda}$ is cofinal in $\kappa$, then this observation directly implies that $i^{I_c}_{0,\infty}(\kappa)=\kappa$. 
 In the other case, if the above sequence is bounded by $\rho<\kappa$, then $I_c$ is a linear iteration of length less than $\kappa$ that only uses ultrafilters on measurable cardinals in the interval $[\kappa_\emptyset,\rho]$ and, since $\kappa$ is a limit of inaccessible cardinals whose cofinality is not contained in this interval, we also know that  $i^{I_c}_{0,\infty}(\kappa)=\kappa$ holds in this case. 
 
   Finally, pick functions $c,d\in{}^\lambda\kappa$ whose proper initial segments are all contained in $\VV$ and $\xi<\lambda$ with $c\restriction\xi=d\restriction\xi$ and $c(\xi)\neq d(\xi)$. 
   Then \eqref{equation:TailsIdentityInLongIterations} implies that $$i^{I_c}_{0,\infty}(x)\cap\kappa_{c\restriction\xi} ~ = ~ i^{I_{c\restriction\xi}}_{0,\infty}(x)\cap\kappa_{c\restriction\xi} ~ = ~ i^{I_d}_{0,\infty}(x)\cap\kappa_{c\restriction\xi}.$$  
 If we now define $$\rho ~ = ~ \min(\nu_{c\restriction(\xi+1)},\nu_{d\restriction(\xi+1)}) ~ < ~ \min(\kappa_{c\restriction(\xi+1)},\kappa_{d\restriction(\xi+1)}),$$ then statement \eqref{item:Split} of Lemma \ref{lemma:TechnicalLemmaIterationsTree} directly implies that  \eqref{equation:SplitLong} holds.      
  \end{proof}

 We now use the above constructions to derive the desired perfect subset result for $\Sigma_1$-definable subsets of power sets of limits of measurable cardinals.

\begin{proof}[Proof of Theorem \ref{theorem:PerfectSubset}]
 Let $\mu$ be an infinite regular cardinal, let $\kappa$ be a limit of measurable cardinals with $\cof{\kappa}=\mu$, let $z$ be an element of $\HH{\kappa}$ and let $D$ be a subset of $\POT{\kappa}$ of cardinality greater than $\kappa$ that is definable by a $\Sigma_1$-formula with parameters $\kappa$ and $z$. 
 An application of Lemma \ref{lemma:LongIterationsGenericExtensions} with the trivial partial order now yields $x\in D$ and systems $\seq{\kappa_s}{s\in{}^{{<}\mu}\kappa}$,  $\seq{U_s}{s\in{}^{{<}\mu}\kappa}$ and $\seq{I_c}{c\in{}^\mu\kappa}$ such that the following statements hold for all $s,t\in{}^{{<}\mu}\kappa$ and all $c,d\in{}^\mu\kappa$: 
 \begin{itemize}
   \item $\kappa_s$ is a measurable cardinal smaller than $\kappa$.
   
   \item $U_s$ is a normal ultrafilter on $\kappa_s$. 
   
   \item $I_c$ is a linear iteration of $\langle\VV,\Set{U_{c\restriction\xi}}{\xi<\mu}\rangle$ with $M^{I_c}_\infty$  well-founded.  
   
   
   \item The sequence $\seq{\kappa_{c\restriction\xi}}{\xi<\mu}$ is cofinal in $\kappa$. 
   
     \item $i^{I_c}_{0,\infty}(\mu)=\mu$, $i^{I_c}_{0,\infty}(\kappa)=\kappa$ and $i^{I_c}_{0,\infty}(z)=z$.  
   
   \item If $\xi<\mu$ with $c\restriction\xi=d\restriction\xi$ and $c(\xi)\neq d(\xi)$, then $$i^{I_c}_{0,\infty}(x)\cap\kappa_{c\restriction\xi} ~ = ~ i^{I_d}_{0,\infty}(x)\cap\kappa_{c\restriction\xi}$$ and $$i^{I_c}_{0,\infty}(x)\cap\rho ~ \neq ~ i^{I_d}_{0,\infty}(x)\cap\rho,$$ where $\rho=\min(\kappa_{c\restriction(\xi+1)},\kappa_{d\restriction(\xi+1)})$.  
  \end{itemize} 
   
   We now define $$\Map{\iota}{{}^\mu\kappa}{\POT{\kappa}}{c}{i^{I_c}_{0,\infty}(x)}.$$ 
   Then $\iota$ is an injection.

   \begin{claim*}
    The map $\iota$ is a perfect embedding. 
   \end{claim*}

   \begin{proof}[Proof of the Claim]
   Fix $c\in{}^\mu\kappa$.
   Given $\alpha<\kappa$,  there is $\xi<\mu$ with $\kappa_{c\restriction\xi}\geq\alpha$ and, if $d\in{}^\mu\kappa$ with $c\restriction\xi=d\restriction\xi$, then $\iota(c)\cap\kappa_{c\restriction\xi}=\iota(d)\cap\kappa_{c\restriction\xi}$. 
   In the other direction, fix $\xi<\nu$ and $d\in{}^\mu\kappa$ with $\iota(c)\cap\kappa_{c\restriction\xi}=\iota(d)\cap\kappa_{c\restriction\xi}$. Assume, towards a contradiction, that $c\restriction\xi\neq d\restriction\xi$. Then there is  $\eta<\xi$ with $c\restriction\eta=d\restriction\eta$ and $c(\eta)\neq d(\eta)$. Our construction then ensures that  $\iota(c)\cap\kappa_{c\restriction(\eta+1)}\neq\iota(d)\cap\kappa_{c\restriction(\eta+1)}$ and therefore  $\iota(c)\cap\kappa_{c\restriction\xi}\neq\iota(c)\cap\kappa_{c\restriction\xi}$, a contradiction. 
   This proves the statement of the claim. 
   \end{proof}

  \begin{claim*}
   $\ran{\iota}\subseteq D$. 
  \end{claim*} 
  
  \begin{proof}[Proof of the Claim]
   Fix $c\in{}^\mu\kappa$. Pick a $\Sigma_1$-formula $\varphi(v_0,v_1,v_2)$ such that $$D ~ = ~ \Set{y\subseteq\kappa}{\varphi(\kappa,y,z)}.$$ 
   As $x \in D$, $\varphi(\kappa,x,z)$ holds in $V$ and hence the properties listed above ensure that $\varphi(\kappa,\iota(c),z)$ holds in $M^{I_c}_\infty$. 
  Since the upwards absoluteness of $\Sigma_1$-statements directly implies that $\varphi(\kappa,\iota(c),z)$ holds in $\VV$, we can conclude that $\iota(c)$ is an element of $D$.  
  \end{proof}
   
  This completes the proof of the theorem.  
\end{proof}

Note that, in general, the conclusion of Theorem \ref{theorem:PerfectSubset} cannot be extended to $\Sigma_1$-formulas using arbitrary subsets of the cardinal $\kappa$ as parameters. 
For example, if $\kappa$ is a regular limit of measurable cardinals, then, in a generic extension by some ${<}\kappa$-closed forcing, there exists a subset of $\POT{\kappa}$ that does not contain the range of a perfect  embedding and is definable by a $\Sigma_1$-formula with parameters in $\POT{\kappa}$ (see {\cite[Corollary 7.9]{MR2987148}}).


\section{Local complexity of canonical inner models}

As a first application of the results of Section \ref{section:PSPlimitsMeasurables}, we prove a result that shows that canonical inner models with infinitely many measurable cardinals are not locally $\Sigma_1$-definable.  
 Note that G\"odel's constructible universe $\LL$, the \emph{Dodd--Jensen core model}  and Kunen's model $\LL[U]$ all possess the property that for every uncountable cardinal $\kappa$, the $\HH{\kappa^+}^M$ of the corresponding inner model $M$ is definable by a $\Sigma_1$-formula with parameters in $\kappa+1$ (see, for example, the proof of {\cite[Lemma 4.13]{Sigma1Partitions}}). 
 In particular, these models satisfy the assumptions of the next theorem.

 \begin{theorem}\label{theorem:Sigma1InnerModelMeasurables}
  Assume that $M$ is a class term with the property that $\ZFC$ proves the following statements:\footnote{{i.e.} there is a formula $\varphi(v)$ in the language of set theory with the property that $\ZFC$ proves the listed statements about the class $M=\Set{x}{\varphi(x)}$.} \begin{enumerate}
      \item The class $M$ is a transitive model of $\ZFC+{\VV=M}$ that contains all ordinals. 
      
      \item\label{item:AbsoAdd} $M$ is forcing invariant under Cohen forcing $\Add{\omega}{1}$. 
      
      \item\label{item:MlocalSigma1} If $\kappa$ is a  limit of measurable cardinals with $\cof{\kappa}=\omega$ and $(\kappa^+)^M=\kappa^+$, then there is a subset of $M\cap\POT{\kappa}$ of cardinality greater than $\kappa$ that is definable by a $\Sigma_1$-formula with parameters in $\HH{\kappa}^M\cup\{\kappa\}$. 
   \end{enumerate}
  Then $\ZFC$ proves that $M$ contains only finitely many measurable cardinals. 
 \end{theorem}
 
 \begin{proof}
  Assume, towards a contradiction, that the above conclusion fails. Then we may work in a model $\VV$ of $\ZFC+{\VV=M}$ that contains infinitely many measurable cardinals. Let $\kappa$ be the least limit of measurable cardinals and let $G$ be $\Add{\omega}{1}$-generic over $\VV$. 
  An application of \eqref{item:AbsoAdd} and \eqref{item:MlocalSigma1} in $\VV[G]$ now yields $z\in\HH{\kappa}^\VV$ and a $\Sigma_1$-formula $\varphi(v_0,v_1,v_2)$ with the property that the set $\Set{y\in\VV[G]}{\VV[G]\models\varphi(\kappa,y,z)}$ is a subset of $\POT{\kappa}^\VV$ that has cardinality greater than $\kappa$ in $\VV[G]$. 
  In this situation, the homogeneity of $\Add{\omega}{1}$ in $\VV$ ensures that these statements about the class defined by the formula $\varphi$ and the parameters $\kappa$ and $z$ hold in every $\Add{\omega}{1}$-generic extension of $\VV$.

  Now, let $G_0\times G_1$ be $(\Add{\omega}{1}\times\Add{\omega}{1})$-generic over $\VV$. 
  Then the above observations show that $$D ~ = ~ \Set{y\in\VV[G_0]}{\VV[G_0]\models\varphi(\kappa,y,z)}$$ is a subset of $\POT{\kappa}^\VV$ that has cardinality greater than $\kappa$ in $\VV[G_0]$. 
  An application of Lemma \ref{lemma:TechnicalLemmaIterationsTree} in $\VV[G_0]$ then yields 
  $x\in D$ and systems  $\seq{\nu_s}{s\in{}^{{<}\omega}\kappa}$, $\seq{\kappa_s}{s\in{}^{{<}\omega}\kappa}$, $\seq{U_s}{s\in{}^{{<}\omega}\kappa}$ and $\seq{I_s}{s\in{}^{{<}\omega}\kappa}$ satisfying the statements listed  in the lemma with respect to $z$ and some subset of $D$ of cardinality $\kappa^+$.
  Define $$c ~ = ~ \bigcup G_1 ~ \in ~ ({}^\omega 2)^{V[G_0,G_1]}\setminus\VV[G_0]$$ and let $I$ denote the unique linear iteration of $\langle\VV[G_0],\Set{U_{c\restriction n}}{n<\omega}\rangle$ of length $\sup_{n<\omega}\length{I_{c\restriction n}}{}$ in $\VV[G_0,G_1]$ with $U^I_\alpha=U^{I_{c\restriction n}}_\alpha$ for all $n<\omega$ and $\alpha<\length{I_{c\restriction n}}{}$.
  Then Lemma \ref{lemma:LongIterationsGenericExtensions} shows that $M^I_\infty$ is well-founded. Set $x_*=i^I_{0,\infty}(x)$. 
  Since Lemma \ref{lemma:LongIterationsGenericExtensions} ensures that $i^I_{0,\infty}(\kappa)=\kappa$ and $i^I_{0,\infty}(z)=z$, we can use the elementarity of $i^I_{0,\infty}$ and $\Sigma_1$-upwards absoluteness to conclude that $\varphi(\kappa,x_*,z)$ holds in $\VV[G_0,G_1]$. 
  Since $\VV[G_0,G_1]$ is an $\Add{\omega}{1}$-generic extension of $\VV$, our earlier observations allow us to  conclude that $x_*$ is an element of $\POT{\kappa}^\VV$. 
  If we now pick $n<\omega$ and set  $\rho=\min(\kappa_{(c\restriction n)^\frown\langle 0\rangle},\kappa_{(c\restriction n)^\frown\langle 1\rangle})$, then Clause \eqref{item:SplitLong} of Lemma \ref{lemma:LongIterationsGenericExtensions} shows that $c(n)$ is the unique $i<2$ with $$x_*\cap\rho ~ = ~ i^{I_{(c\restriction n)^\frown\langle i\rangle}}_{0,\infty}(x)\cap\rho.$$ 
  This conclusion implies that $c$ is an element of $\VV[G_0]$, a contradiction. 
 \end{proof}

 The above result can easily be shown to be optimal, in the sense that there exists a class term $M$ satisfying the above three properties that can consistently contain any finite number of measurable cardinals. 
 Ideas from the proof of Theorem \ref{theorem:PerfectSubsetOptimal} for singular cardinals of countable cofinality will allow us to prove the following result in Section \ref{section:OptimalCountable}.

\begin{proposition}\label{proposition:InnerModelNmeasurables}
  There exists a class term $M$ with the following properties: 
  \begin{enumerate}
   \item There is a $\Sigma_1$-formula $\varphi(v_0,v_1,v_2)$ with the property that $\ZFC$ proves the following statements:
    \begin{enumerate}
     \item The class $M$ is a transitive model of $\ZFC+{\VV=M}$ that contains all ordinals. 
      
      \item $M$ is forcing invariant. 
      
      \item  If $\kappa$ is an uncountable  cardinal, then there exists $z\in\HH{\kappa}^M$ with $\HH{\kappa^+}^M=\Set{y}{\varphi(\kappa,y,z)}$. 
    \end{enumerate}
   
   \item Given a natural number $n$, if the theory $$\ZFC ~ + ~ \anf{\textit{There exist $n$ measurable cardinals}}$$ is consistent, then so is the theory $$\ZFC ~ +  ~ \VV=M ~ + ~  \anf{\textit{There exist $n$ measurable cardinals}.}$$ 
  \end{enumerate}
\end{proposition}


\section{The lower bound for  singular cardinals of countable cofinality}\label{section:OptimalCountable}

 In this section, we will prove the following result that covers the case of singular strong limit cardinals of countable cofinality in the statement of Theorem \ref{theorem:PerfectSubsetOptimal}:

\begin{theorem}\label{theorem:StrengthCountableCof}
 Assume that there is no inner model with infinitely many measurable cardinals and let $\kappa$ be a singular strong limit cardinal of countable cofinality. Then there is a subset $D$ of $\POT{\kappa}$ of cardinality greater than $\kappa$ that is definable by a $\Sigma_1$-formula with parameters in $\HH{\kappa}\cup\{\kappa\}$ such that there is no continuous injection  $\map{\iota}{{}^\omega\kappa}{\POT{\kappa}}$ with $\ran{\iota}\subseteq D$. 
\end{theorem}

The proof of the above theorem relies on the theory of  \emph{short} core models developed by Koepke in \cite{MR926749} and generalizations of basic concepts from classical descriptive set theory to simply definable collections of subsets of singular cardinals of countable cofinality. In the following, we will briefly introduce these generalized notions.

\begin{definition}
 Let $\kappa$ be a  limit cardinal of countable cofinality and let $0<n<\omega$ be a natural number. 
 \begin{enumerate}
     \item A subset $T$ of $({}^{{<}\omega}\kappa)^n$ is a \emph{subtree of $({}^{{<}\omega}\kappa)^n$} if the following statements hold for all $\langle t_0,\ldots,t_{n-1}\rangle\in T$: 
     \begin{enumerate}
         \item $\length{t_0}   =  \ldots  =  \length{t_{n-1}}$. 
         
         \item If  $m<\length{t_0}$, then  $\langle t_0\restriction m,\ldots,t_{n-1}\restriction m\rangle ~ \in ~ T$.
     \end{enumerate}
     
  \item If $T$ is a subtree of $({}^{{<}\omega}\kappa)^n$, then we define $[T]$ to be the set of all elements $\langle x_0,\ldots,x_{n-1}\rangle$ of $({}^\omega\kappa)^n$ with the property that $\langle x_0\restriction m,\ldots,x_{n-1}\restriction m\rangle\in T$ holds for all $m<\omega$. 
  
  \item A subset $X$ of $({}^\omega\kappa)^n$ is a \emph{$\mathbf{\Sigma}^1_1$-subset} if there exists a subtree $T$ of $({}^{{<}\omega}\kappa)^{n+1}$ with $$X ~ = ~ p[T] ~ = ~ \Set{\langle x_0,\ldots,x_{n-1}\rangle\in({}^\omega\kappa)^n}{\exists y ~ \langle x_0,\ldots,x_{n-1},y\rangle\in[T]}.$$ 
  
  \item A subset of $({}^\omega\kappa)^n$ is a \emph{$\mathbf{\Pi}^1_1$-subset} if its complement in $({}^\omega\kappa)^n$ is a $\mathbf{\Sigma}^1_1$-subset. 
 \end{enumerate}
\end{definition}

As in the classical case, we can use universal sets to show that the classes of $\mathbf{\Sigma}^1_1$- and $\mathbf{\Pi}^1_1$-subsets do not coincide at singular strong limits of countable cofinality.

\begin{proposition}\label{proposition:DiagSing}
 If $\kappa$ is a singular strong limit cardinal of countable cofinality, then there exists a $\mathbf{\Sigma}^1_1$-subset of ${}^\omega\kappa$ that is not a $\mathbf{\Pi}^1_1$-subset. 
\end{proposition}

\begin{proof}
 Pick a strictly increasing sequence $\seq{\kappa_m}{m<\omega}$ of infinite cardinals that is cofinal in  $\kappa$. In addition, fix an enumeration $\seq{a_\alpha}{\alpha<\kappa}$ of $\HH{\kappa}$. 
 Define $U$ to be the set of all pairs $\langle s,t\rangle$ in ${}^{{<}\omega}\kappa\times{}^{{<}\omega}\kappa$ with the property that $\length{s}=\length{t}$ and $\langle s\restriction l,t\restriction l\rangle\in a_{s(m)}$ for all $l\leq m<\length{s}$ with $s[l]\cup t[l]\subseteq\kappa_m$. 
 Then it is easy to see that $U$ is a subtree of ${}^{{<}\omega}\kappa\times{}^{{<}\omega}\kappa$. 
 Assume, towards a contradiction, that there exists a subtree $T$ of ${}^{{<}\omega}\kappa\times{}^{{<}\omega}\kappa$ with $p[T]={}^\omega\kappa\setminus p[U]$. 
 Pick a function $x\in{}^\omega\kappa$ with the property that $a_{x(m)}=\HH{\kappa_m}\cap T$ holds for all $m<\omega$. 
 
 Now, assume that there is $y\in{}^\omega\kappa$ with $\langle x,y\rangle\in[T]$. Then $\langle x,y\rangle\notin[U]$ and there exists $l<\omega$ with $\langle x\restriction l,y\restriction l\rangle\notin U$. 
 Then there exists $l\leq m<\omega$ 
 with $x[l]\cup y[l]\subseteq\kappa_m$ and $\langle x\restriction l,y\restriction l\rangle\notin a_{x(m)}=\HH{\kappa_m}\cap T$. 
 But, this yields a contradiction, because $\langle x\restriction l,y\restriction l\rangle$ is an element of $T$. 
 This shows that there is $y\in{}^\omega\kappa$ with $\langle x,y\rangle\in[U]$. Then $\langle x,y\rangle\notin[T]$ and there is $l<\omega$ with $\langle x\restriction l,y\restriction l\rangle\notin T$. 
 Pick $l\leq m<\omega$ with $x[l]\cup y[l]\subseteq\kappa_m$. Then the fact that $\langle x,y\rangle\in[U]$ implies that $\langle x\restriction l,y\restriction l\rangle\in a_{x(m)}\subseteq T$, a contradiction. 
\end{proof}

The proof of Theorem \ref{theorem:StrengthCountableCof} relies on a generalization of the \emph{Boundedness Lemma}  to singular cardinals of countable cofinality. Below, we introduce the definitions needed in the formulation of this result.

\begin{definition}
 Let $\kappa$ be an infinite cardinal, let $\vec{\kappa}=\seq{\kappa_\xi}{\xi<\cof{\kappa}}$ be a strictly increasing sequence of ordinals that is cofinal in $\kappa$ and let $\vec{a}=\seq{a_\alpha}{\alpha<\kappa}$ be a  sequence of elements of $\HH{\kappa}$. 
 \begin{enumerate}
     \item Given $z\subseteq\kappa$, we define $\lhd_z$ to be the unique binary relation on $\kappa$ with the property that $$\alpha\lhd_z\beta ~ \Longleftrightarrow ~ \goedel{\alpha}{\beta}\in z$$ holds for all $\alpha,\beta<\kappa$.\footnote{Here, we let $\map{\goedel{\cdot}{\cdot}}{\On\times\On}{\On}$ denote the \emph{G\"odel pairing function}.} 
     
     \item We define $\calW\calO_\kappa$ to be the set of all $z\in\POT{\kappa}$ with the property that $\lhd_z$ is a well-ordering of $\kappa$. 
     
    \item We let $\mathsf{WO}(\vec{\kappa},\vec{a})$ denote the set of all $x\in{}^{\cof{\kappa}}\kappa$ with the property that there exists $y\in\calW\calO_\kappa$ such that  $y\cap\kappa_\xi=a_{x(\xi)}$ holds for all $\xi<\cof{\kappa}$. 
     
    \item Given an element  $x$ of $\mathsf{WO}(\vec{\kappa},\vec{a})$, we let $\|x\|_{\vec{a}}$ denote the order-type of the resulting well-order  $\langle\kappa,\lhd_{\hspace{0.9pt}\bigcup \Set{a_{x(\xi)}}{\xi<\cof{\kappa}}}\rangle$. 
 \end{enumerate}
\end{definition}

\begin{lemma}\label{lemma:BoundednessSingular}
 Let $\kappa$ be a singular strong limit cardinal of countable cofinality, let $\vec{\kappa}=\seq{\kappa_m}{m<\omega}$ be a strictly increasing sequence of cardinals that is cofinal in $\kappa$ and let $\vec{a}=\seq{a_\alpha}{\alpha<\kappa}$ be an enumeration of $\HH{\kappa}$. 
 If $B$ is a $\mathbf{\Sigma}^1_1$-subset of ${}^\omega\kappa$ with  $B\subseteq\mathsf{WO}(\vec{\kappa},\vec{a})$, then there exists an ordinal $\gamma<\kappa^+$ with $\|y\|_{\vec{a}}<\gamma$ for all $y\in B$. 
\end{lemma}

\begin{proof}
 Assume, towards a contradiction, that the set $\Set{\|y\|_{\vec{a}}}{y\in B}$ is unbounded in $\kappa^+$. 
  Pick a subtree $S$ of ${}^{{<}\omega}\kappa\times{}^{{<}\omega}\kappa$ with $p[S]=B$. 
 By Proposition \ref{proposition:DiagSing}, there exists a subtree $T$ of ${}^{{<}\omega}\kappa\times{}^{{<}\omega}\kappa$ with the property that the set  $A={}^\omega\kappa\setminus p[T]$ is not a $\mathbf{\Sigma}^1_1$-subset of ${}^\omega\kappa$. 
 Given $x\in{}^\omega\kappa$, set $$T_x ~ = ~ \Set{t\in{}^{{<}\omega}\kappa}{\langle x\restriction\length{t},t\rangle\in T}.$$ Then $T_x$ is a subtree of ${}^{{<}\omega}\kappa$ for all  $x\in{}^\omega\kappa$ and $A=\Set{x\in{}^\omega\kappa}{[T_x]=\emptyset}$. 
 By standard arguments (see {\cite[Section 2.E]{MR1321597}}), we now know that a function $x\in{}^\omega\kappa$ is contained in $A$ if and only if there exists an ordinal $\gamma<\kappa^+$ and a function $\map{r}{T_x}{\gamma}$ with $r(s)>r(t)$ for all $s,t\in T_x$ with $s\subsetneq t$. 
 Since our assumption implies that for every $\gamma<\kappa^+$, there is $y\in B$ with the property that there exists an order-preserving embedding of $\langle\gamma,<\rangle$ into $\langle\kappa,\lhd_{\hspace{0.9pt}\bigcup \Set{a_{y(\xi)}}{\xi<\cof{\kappa}}}\rangle$, 
 we know that $A$ consists of all $x\in{}^\omega\kappa$ with the property that there exists $y\in B$ and a function $\map{f}{T_x}{\kappa}$ such that for all $s,t\in T_x$ with $s\subsetneq t$ and all $m<\omega$ with $f(s),f(t)<\kappa_m$, we have  $\goedel{f(t)}{f(s)}\in a_{y(m)}$.

 Below, we  aim to derive a contradiction from the above assumption by constructing a subtree $U$ of ${}^{{<}\omega}\kappa\times{}^{{<}\omega}\kappa$ with the property that a pair $\langle x,y\rangle$ in ${}^\omega\kappa\times{}^\omega\kappa$ is an element of $[U]$ if and only if $y$ codes (in some fixed canonical way) functions $\map{c}{T_x}{\omega}$, $\map{f}{T_x}{\kappa}$ and $\map{u,v}{\omega}{\kappa}$ such that $f(p)<\kappa_{c(p)}$ for all $p\in T_x$, the pair $\langle u,v\rangle$ is an element of $[S]$ and $\goedel{f(q)}{f(p)}\in a_{u(\max(c(p),c(q)))}$ holds for all $p,q\in T_x$ with $p\subsetneq q$.  
 Note that if $U$ is a tree with these properties, $\langle x,y\rangle$ is an element of $U$ and $c$, $f$, $u$ and $v$ are the functions coded by $y$, then $u$ is an element of $B$ and the functions $f$ and $u$ witness that $x$ is an element of $A$. But this means that, if we succeed in constructing such a tree $U$, then we derive a contradiction, because the  properties of $U$ ensure that $p[U]=A$ and hence such a tree $U$ witnesses that $A$ is a $\mathbf{\Sigma}^1_1$-subset of ${}^\omega\kappa$.

 We now show that our assumptions allow us to construct a tree $U$ with the properties described above. 
 For every $s\in{}^{{<}\omega}\kappa$, we set $$T_s ~ = ~ \Set{t\in{}^{{<}\omega}\kappa}{\length{t}\leq\length{s}, ~ \langle s\restriction\length{t},t\rangle\in T}.$$
 Define $U$ to be the subset  of ${}^{{<}\omega}\kappa\times{}^{{<}\omega}\kappa$ consisting of pairs $\langle s,t\rangle$ with $\length{s}=\length{t}$ and the property that for all $m<\length{s}$, there exist $c_m,f_m,u_m,v_m\in\HH{\kappa}$ such that $a_{t(m)}=\langle c_m,f_m,u_m,v_m\rangle$ and the following statements hold for all $l\leq m$: 
 \begin{itemize}
  \item $\langle u_l,v_l\rangle,\langle u_m,v_m\rangle\in S$, $\length{u_m}=m+1$, $u_l=u_m\restriction(l+1)$ and $v_l=v_m\restriction(l+1)$. 
  
  \item $\map{c_m}{\HH{\kappa_m}\cap T_{s\restriction m}}{\omega}$ with $c_m\restriction\dom{c_l}=c_l$. 
  
  \item $\map{f_m}{\Set{w\in\HH{\kappa_m}\cap T_{s\restriction m}}{c_m(w)\leq m}}{\kappa_m}$ with $f_m\restriction\dom{f_l}=f_l$ and  $\goedel{f_m(q)}{f_m(p)}\in a_{u_m(m)}$ for all $p,q\in\dom{f_m}$ with $p\subsetneq q$. 
 \end{itemize}
 Then $U$ is a subtree of ${}^{{<}\omega}\kappa\times{}^{{<}\omega}\kappa$. 
 
 \begin{claim*}
  $p[U] = A$.
 \end{claim*}

 \begin{proof}[Proof of the Claim]
  First, fix $\langle x,y\rangle\in[U]$. Then there are $\map{c}{T_x}{\omega}$, $\map{f}{T_x}{\kappa}$ and $\langle u,v\rangle\in[S]$ with the property that for all $m<\omega$, the set $a_{y(m)}$ is equal to the quadruple $$\langle c\restriction(\HH{\kappa_m}\cap T_{x\restriction m}),f\restriction\Set{w\in\HH{\kappa_m}\cap T_{x\restriction m}}{c(w)\leq m},u\restriction(m+1),v\restriction(m+1)\rangle.$$ 
 Then $u\in B$ and $\goedel{f(q)}{f(p)}\in a_{u(m)}$ holds for all $p,q\in T_x$ with $p\subsetneq q$ and all $m<\omega$ with $f(p),f(q)<\kappa_m$. 
 By earlier observations, this shows that $x\in A$. 
 
 Now, pick $x\in A$. 
 Then we can find $\langle u,v\rangle\in S$ and a function $\map{f}{T_x}{\kappa}$ such that for all $p,q\in T_x$ with $p\subsetneq q$ and all $m<\omega$ with $f(p),f(q)<\kappa_m$, we have  $\goedel{f(q)}{f(p)}\in a_{u(m)}$. 
 Let $\map{c}{T_x}{\omega}$ denote the unique function with $c(p)=\min\Set{m<\omega}{f(p)<\kappa_m}$. If we then pick $y\in{}^\omega\kappa$ such that the set $a_{y(m)}$ is equal to the quadruple $$\langle c\restriction(\HH{\kappa_m}\cap T_{x\restriction m}),f\restriction\Set{w\in\HH{\kappa_m}\cap T_{x\restriction m}}{c(w)\leq m},u\restriction(m+1),v\restriction(m+1)\rangle$$ for all $m<\omega$, then we can conclude that $\langle x,y\rangle\in[U]$. 
 \end{proof}
 
 The above computations allow us to conclude that $A=p[U]$, contradicting the fact that $A$ is not a $\mathbf{\Sigma}^1_1$-subset of ${}^\omega\kappa$. 
\end{proof}

 We are now ready to prove the main result of this section.

\begin{proof}[Proof of Theorem \ref{theorem:StrengthCountableCof}]
 Assume that there is no inner model with infinitely many measurable cardinals. Then {\cite[Theorem 2.14]{MR926749}} implies that $0^{\text{long}}$ (as defined in {\cite[Definition 2.13]{MR926749}}) does not exist. 
 Let $U_{\text{can}}$ denote the \emph{canonical sequence of measures} and let $\KK[U_{\can}]$ denote the \emph{canonical core model} (as defined in {\cite[Definition 3.15]{MR926749}}). Then our assumption implies that $\dom{U_{\can}}$ is finite and {\cite[Theorem 3.23]{MR926749}} shows that there is a generic extension $\KK[U_{\can},G]$ of $\KK[U_{\can}]$ by finitely many Prikry forcings with the property that for every ordinal $\tau\geq\omega_2$ and every $X\subseteq\tau$ such that  $\betrag{X}$ is a regular cardinal smaller than $\betrag{\tau}$, there exists $Z\in\POT{\tau}^{\KK[U_{\can},G]}$ with $X\subseteq Z$ and $\betrag{Z}^{\KK[U_{\can},G]}<\tau$.

 Now, let $\kappa$ be a singular strong limit cardinal of countable cofinality. 
 Then $\kappa$ is singular in $\KK[U_{\can},G]$ and $\kappa^+=(\kappa^+)^{\KK[U_{\can},G]}$. Moreover, since forcing with a finite iteration of Prikry forcings preserves all cardinals, we also know that $\kappa^+=(\kappa^+)^{\KK[U_{\can}]}$. 
 Set $U={U_{\can}}\restriction{\kappa}$ and $\KK=\KK[U]$ (see {\cite[Definition 3.1]{MR926749}}). Then {\cite[Theorem 3.2]{MR926749}} shows that $\KK$ is an inner model of $\ZFC$. 
 Moreover, we can use {\cite[Theorem 3.9.(iii)]{MR926749}} to conclude that $\POT{\kappa}^{\KK[U_{\can}]}\subseteq\KK$ and therefore we know that $\kappa^+=(\kappa^+)^\KK$. 
 
 Next, let $<_\KK$ denote the canonical well-ordering of $\KK$ given by {\cite[Theorem 3.4]{MR926749}}). 
 For every $\kappa\leq\gamma<\kappa^+$, let $\map{b_\gamma}{\kappa}{\gamma}$ denote the $<_\KK$-least bijection between $\kappa$ and $\gamma$, and set $y_\gamma=\Set{\goedel{\alpha}{\beta}}{\alpha,\beta<\kappa, ~ b_\gamma(\alpha)<b_\gamma(\beta)}$. 
 Finally, we define $D=\Set{y_\gamma}{\kappa\leq\gamma<\kappa^+}$. Then $D$ is a subset of $\calW\calO_\kappa$ of cardinality $\kappa^+$.

 \begin{claim*}
  The set $D$ is definable by a $\Sigma_1$-formula with parameters in $\HH{\kappa}\cup\{\kappa\}$. 
 \end{claim*}
 
 \begin{proof}[Proof of the Claim]
  First, note that our assumption implies that $U$ is an element of $\HH{\kappa}^\KK$. 
  By arguing as in the proof of {\cite[Lemma 2.3]{MR3845129}}, we can combine {\cite[Theorem 2.7]{MR926749}} with {\cite[Theorem 2.10]{MR926749}}  to conclude that the collection of all initial segments of the restriction of $<_\KK$ to $\HH{\kappa^+}^\KK$ is definable by a $\Sigma_1$-formula with parameters $\kappa$ and $U$. 
  This conclusion directly implies the statement of the claim. 
 \end{proof}

 In the following, assume, towards a contradiction, that there is a continuous injection $\map{\iota}{{}^\omega\kappa}{\POT{\kappa}}$ with $\ran{\iota}\subseteq D$. 
 Fix a strictly increasing sequence $\vec{\kappa}=\seq{\kappa_m}{m<\omega}$ of cardinals that is cofinal in $\kappa$ and an enumeration $\vec{a}=\seq{a_\alpha}{\alpha<\kappa}$ of $\HH{\kappa}$. 
 Define $T$ to be the set of all pairs $\langle s,t\rangle$ in ${}^{{<}\omega}\kappa\times{}^{{<}\omega}\kappa$ such that  $\length{s}=\length{t}$ and the following statements hold for all $l\leq m<\length{s}$: 
 \begin{itemize}
     \item $a_{s(m)}\subseteq\kappa_m$ and $a_{s(l)}=a_{s(m)}\cap\kappa_l$. 
     
     \item $a_{t(l)},a_{t(m)}\in{}^{{<}\omega}\kappa$ with $l\leq\length{a_{t(l)}}\leq\length{a_{t(m)}}$, $a_{t(l)}=a_{t(m)}\restriction\length{a_{t(l)}}$ and $\iota(u)\cap\kappa_m=a_{s(m)}$ for all $u\in{}^\omega\kappa$ with $a_{t(m)}\subseteq u$. 
 \end{itemize}
 This definition directly ensures that $T$ is a subtree of ${}^{{<}\omega}\kappa\times{}^{{<}\omega}\kappa$. 
 Pick $\langle x,y\rangle\in [T]$. Set $u=\bigcup\Set{a_{y(m)}}{m<\omega}\in{}^\omega\kappa$ and $v=\bigcup\Set{a_{x(m)}}{m<\omega}\subseteq\kappa$. 
 By the definition of $T$, we then have $\iota(u)=v\in D\subseteq\calW\calO_\kappa$ and this shows that $x$ is an element of $\mathsf{WO}(\vec{\kappa},\vec{a})$. 
 This shows that $p[T]\subseteq\mathsf{WO}(\vec{\kappa},\vec{a})$ and therefore Lemma \ref{lemma:BoundednessSingular} yields an ordinal $\gamma<\kappa^+$ with $\|x\|_{\vec{a}} < \gamma$ for all $x\in p[T]$.

 Since for every ordinal $\kappa\leq\delta<\kappa^+$, there is a unique element $y$ of $D$ with $\otp{\kappa,\lhd_y}=\delta$, we know that the map $$\Map{i}{{}^\omega\kappa}{\kappa^+}{u}{\otp{\kappa,\lhd_{\iota(u)}}}$$ is an injection and  we can find $u\in{}^\omega\kappa$ with $\otp{\kappa,\lhd_{\iota(u)}}>\gamma$. 
 Pick $x\in{}^\omega\kappa$ with $a_{x(m)}=\iota(u)\cap\kappa_m$ for all $m<\omega$. 
 In addition, pick $y\in{}^\omega\kappa$ with the property that for all $l\leq m<\omega$, we have  $a_{y(l)},a_{y(m)}\in{}^{{<}\omega}\kappa$, $l\leq\length{a_{y(l)}}\leq\length{a_{y(m)}}$,  $a_{y(l)}\subseteq a_{y(m)}=u\restriction\length{a_{y(m)}}$ and $\iota(w)\cap\kappa_m=\iota(u)\cap\kappa_m$ for all $w\in{}^\omega\kappa$ with $u\restriction\length{a_{y(m)}}\subseteq w$. Note that this is possible as $\iota$ is a continuous injection.
 Then $\langle x,y\rangle\in[T]$ and $x\in p[T]\subseteq\mathsf{WO}(\vec{\kappa},\vec{a})$ with $\|x\|_{\vec{a}}=\otp{\kappa,\lhd_{\iota(u)}}>\gamma$, a contradiction. 
\end{proof}

We close this section by using ideas from the above proof to show that the assumptions of Theorem \ref{theorem:Sigma1InnerModelMeasurables} are optimal. 
These arguments make use of the following observation that can be seen as a special case of the generic absoluteness of the core model as, for example, in \cite[Theorem 3.4]{Schimmerling10}, \cite{JS13}, or \cite{St96}.

\begin{lemma}\label{lemma:KoepeCareAbsolute}
 Assume that $0^{\text{long}}$ does not exist. If $\VV[G]$ is a generic extension of the ground model $\VV$, then $\KK[U_{can}]^\VV=\KK[U_{can}]^{\VV[G]}$. 
\end{lemma}

\begin{proof}
 The statement of the lemma will be a direct consequence of the following two claims: 
 
 \begin{claim*}
  If $\VV[G]$ is a generic extension of the ground model $\VV$, then $\KK[U_{can}]^\VV=\KK[U_{can}^\VV]^{\VV[G]}$. 
 \end{claim*}
 
 \begin{proof}[Proof of the Claim]
  Since the property of being a $U$-mouse is upwards absolute between transitive models of $\ZFC$ with the same ordinals, we know that $\KK[U_{can}]^\VV\subseteq\KK[U_{can}^\VV]^{\VV[G]}$. 
  %
  Next,  observe that the fact that $\VV[G]$ is a set forcing extension of $\VV$ implies that all sufficiently large singular cardinals in $\VV[G]$ are singular in $\VV$. Moreover, 
  an application of {\cite[Theorem 3.23]{MR926749}} shows that all sufficiently large singular cardinals in $\VV$ are singular  in  $\KK[U_{can}]^\VV$. 
  In combination, this shows that for all sufficiently large singular cardinals $\lambda$ of uncountable cofinality in $\VV[G]$, every closed unbounded subset of $\lambda$ in $\VV[G]$ contains an element that is singular in $\KK[U_{can}]^\VV$. 
  This observation allows us to use  {\cite[Theorem 3.24(ii)]{MR926749}} to conclude that $\KK[U_{can}^\VV]^{\VV[G]}\subseteq\KK[U_{can}]^\VV$. 
 \end{proof}

 \begin{claim*}
  Let $\PPP$ be a weakly homogeneous partial order. If $G$ is $\PPP$-generic over $\VV$, then $\KK[U_{can}]^\VV=\KK[U_{can}]^{\VV[G]}$. 
 \end{claim*}
 
  \begin{proof}[Proof of the Claim]
  First, the weak homogeneity of $\PPP$ in $\VV$ ensures that $$U_{can}^{\VV[G]} ~ \subseteq ~ \KK[U_{can}]^{\VV[G]} ~ \subseteq ~ \HOD^{\VV[G]} ~ \subseteq ~ \VV.$$ In particular, we know that the set  $U_{can}^{\VV[G]}(\kappa)\cap\POT{\kappa}^\VV$ is an element of $\VV$ for every $\kappa\in\dom{U_{can}^{\VV[G]}}$. 
  In this situation, we can now use the first claim to inductively show that the definition of the canonical measure sequence ensures that  $U_{\can}^\VV\restriction\xi=U_{can}^{\VV[G]}\restriction\xi$ holds for all $\xi\in\On$. 
 \end{proof}
 
 Now, let $\PPP$ be a partial order and let $G$ be $\PPP$-generic over $\VV$. 
 Pick a sufficiently large cardinal $\delta$ such that $\PPP\times\Col{\omega}{\delta}$ densely embeds into $\Col{\omega}{\delta}$ and let $H$ be $\Col{\omega}{\delta}$-generic over $\VV[G]$.  
 Since $\Col{\omega}{\delta}$ is weakly homogeneous in both $\VV$ and $\VV[G]$, we can now use the above claim twice to conclude that $\KK[U_{can}]^\VV=\KK[U_{can}]^{\VV[G,H]}=\KK[U_{can}]^{\VV[G]}$.  
\end{proof}

\begin{proof}[Proof of Proposition \ref{proposition:InnerModelNmeasurables}] 
 Let $M$ denote the  class term with the property that $\ZFC$ proves the following statements: 
 \begin{itemize}
     \item If either $0^{\text{long}}$ exists, or $0^{\text{long}}$ does not exist and the model $\KK[U_{can}]$ contains infinitely many measurable cardinals, then $M$ is equal to the constructible universe $\LL$. 
     
     \item Otherwise, $M$ is equal to $\KK[U_{can}]$.  
 \end{itemize}

 Then the standard results about $\LL$ together with {\cite[Theorem 3.2]{MR926749}}  show that $\ZFC$ proves that $M$ is a transitive model of $\ZFC+{\VV=M}$ that contains all ordinals.  
 Moreover,  Lemma \ref{lemma:KoepeCareAbsolute} together with the fact that $0^{\text{long}}$ cannot be added by forcing  show that $M$ is forcing invariant. 
 
 \begin{claim*}
  Assume that $0^{\text{long}}$ does not exist and  $\KK[U_{can}]$ contains only  finitely many measurable cardinals. If $\kappa$ is an uncountable cardinal, then $\HH{\kappa^+}^{\KK[U_{can}]}$ is definable by a $\Sigma_1$-formula with parameters in $\HH{\kappa}\cup\{\kappa\}$. 
 \end{claim*}
 
 \begin{proof}[Proof of the Claim]
  Set $U=U_{can}\restriction\kappa$ and $\KK=\KK[U]$. Then {\cite[Theorem 3.9]{MR926749}} shows that $\HH{\kappa^+}^{\KK[U_{can}]}=\HH{\kappa^+}^\KK$. 
  Moreover, if $\kappa$ is not the successor of an element of $\dom{U}$ in $\KK$, then $U$ is an element of $\HH{\kappa}$ and we can repeat arguments from the proof of Theorem \ref{theorem:StrengthCountableCof} to show that the class of all $U$-mice $M$ (see {\cite[Definition 2.9]{MR926749}}) that contain $\kappa$ in their \emph{lower part $lp(M)$} (see {\cite[Definition 2.1]{MR926749}}) is definable by a $\Sigma_1$-formula with parameters $\kappa$ and $U$. 
  Since every element of $\HH{\kappa^+}^\KK$ is contained in such a lower part, the statement of the claim follows in this case. 
  
  In the following, assume that there is $\delta\in\dom{U}$ with $\kappa=(\delta^+)^\KK$. 
  Let $F$ be a simple predicate with $\dom{F}=\dom{U}$ and let $M$ be an $F$-mouse such that  $\kappa,F\in lp(M)$, $\kappa=(\delta^+)^M$ and $F(\mu)$ is an ultrafilter in $M$ for every $\mu\in\dom{F}$.  %
  Since every subset of $\delta$ in $\KK[F]$ is contained in an $F$-mouse of cardinality less than $\kappa$, we can now apply {\cite[Theorem 2.10]{MR926749}} to conclude that $F(\mu)$ is an ultrafilter in $\KK[F]$ for every $\mu\in\dom{F}$. 
  This shows that $\KK[F]$ is a \emph{core model} (in the sense of {\cite[Definition 3.6]{MR926749}}) and therefore {\cite[Theorem 3.14]{MR926749}} shows that $\KK=\KK[F]$ holds. 
  Since every element of $\HH{\kappa^+}^\KK$ is contained in the lower part of a $U$-mouse $M$ with $\kappa,U\in lp(M)$ and  $\kappa=(\delta^+)^M$, 
  we now know that $\HH{\kappa^+}^\KK$ consists of all sets $x$ with the property that there exists a simple predicate $F$ and an $F$-mouse $M$ such that    $\dom{F}=\dom{U}$,   $F(\mu)$ is an ultrafilter in $M$ for every $\mu\in\dom{F}$,   $\kappa=(\delta^+)^M$  and $\kappa,F,x\in lp(M)$.
  This allows us to conclude that the set $\HH{\kappa^+}^\KK$ is definable by a $\Sigma_1$-formula with parameters $\delta$, $\kappa$ and $\dom{U}$ in this case. 
 \end{proof}
 
 The above claim now allows us to find a $\Sigma_1$-formula $\varphi(v_0,v_1,v_2)$ with the property that for every uncountable cardinal $\kappa$, we have $\LL_{\kappa^+}=\Set{x}{\varphi(\kappa,\kappa,x)}$ and, if $0^{\text{long}}$ does not exist and $\KK[U_{can}]$ contains only  finitely many measurable cardinals, then there exists $z\in\HH{\kappa}$ with $\HH{\kappa^+}^{\KK[U_{can}]}=\Set{x}{\varphi(\kappa,x,z)}$.  
 Finally, if the existence of $n$ measurable cardinals is consistent with the axioms of $\ZFC$ for some natural number $n$, then the existence of exactly $n$ measurable cardinals in $\KK[U_{can}]$ is consistent with $\ZFC$. 
\end{proof}


\section{The lower bound for singular cardinals of uncountable cofinality}\label{section:LowerBoundUncountableCof}

We now use ideas from \cite{HM19} to complete the proof of Theorem \ref{theorem:PerfectSubsetOptimal}.

\begin{proof}[Proof of Theorem \ref{theorem:PerfectSubsetOptimal}]
 Let $\kappa$ be a singular strong limit cardinal with the property that for every subset $D$ of $\POT{\kappa}$ of cardinality greater than $\kappa$ that is definable by a $\Sigma_1$-formula with parameters in $\HH{\kappa}\cup\{\kappa\}$, there exists a perfect embedding $\map{\iota}{{}^{\cof{\kappa}}\kappa}{\POT{\kappa}}$ with $\ran{\iota}\subseteq D$. 
 Assume, towards a contradiction, that there is no inner model with a sequence of measurable cardinals of length $\cof{\kappa}$.
 Then Theorem \ref{theorem:StrengthCountableCof} implies that the cofinality of $\kappa$ is uncountable. 
 Moreover, we know that there is no inner model with a measurable cardinal of Mitchell order $1$ and therefore we can construct the canonical core model $\KK$ as in \cite{Ze02} (which is Steel's core model \cite{St96} in this easier setting). Note that our hypothesis implies that, in $\KK$,  the sequence of measurable cardinals below $\kappa$ is bounded below $\kappa$.
 In addition, as $\kappa$ is singular in $\VV$, \emph{weak covering} (see \cite[Theorem 7.5.1]{Ze02}) holds for $\KK$ at $\kappa$, i.e., we have $ (\kappa^+)^\KK = \kappa^+$.
 %
 Finally, we know that $\kappa$ is singular in $\KK$, because otherwise the fact that $\kappa$ is a singular cardinal of uncountable cofinality would allow us to apply the second part of {\cite[Theorem 1]{MR2567927}} to find an inner model in which $\kappa$ has Mitchell order greater than $0$.

 We will now construct a tree of height $\cof{\kappa}^\KK$ that is an element of $\KK$ and then argue that this tree does not have a perfect subtree in $\VV$. 
 These arguments use ideas from \cite{HM19} that ultimately go back to Solovay's argument for the consistency strength of the \emph{Kurepa Hypothesis} (see  \cite[Section 4]{Je71}). 
 Our tree consists of hulls of initial segments of $\KK$ of size $\kappa$ and we will argue that we can obtain such initial segments in a $\Sigma_1$-definable way with parameters in $\HH{\kappa} \cup \{ \kappa \}$.

 In the following, we say a premouse $N$ (in the sense of \cite[Section 4.1]{Ze02}) is \emph{good} if the following statements hold: 
 \begin{itemize}
     \item $N$ is \emph{iterable} (in the sense of \cite[Section 4.2]{Ze02}). 
     
     \item $\kappa+1 \subseteq N$ and $\betrag{N} = \kappa$. 
     
     \item $\cof{\kappa}^N=\cof{\kappa}^\KK$. 
     
     \item $\kappa$ is the largest cardinal in $N$.
      
     \item If $\gamma < \kappa$ is the supremum of the measurable cardinals below $\kappa$ in $\KK$, then \[ N \vert \gamma^{++} ~ = ~  \KK \vert \gamma^{++}. \] In particular, $K$ and $N$ have the same measurable cardinals and the same total measures below $\kappa$. 
 \end{itemize}

 \begin{claim*}
Let $N$ be a good premouse. Then $N \lhd \KK$.
\end{claim*}

\begin{proof}[Proof of the Claim]
 Compare $N$ and $\KK$ and suppose, towards a contradiction,  that the comparison is not trivial.  Consider the first measure that is used. As $N \vert \gamma^{++} = \KK \vert \gamma^{++},$ where $\gamma < \kappa$ is the supremum of the measurable cardinals below $\kappa$ in $N$ and $\KK$, the first measure that is used in the comparison has to be a partial measure above $\gamma$. Say this is a partial measure $\mu$ with critical point $\nu$ on the $\KK$-side of the comparison. 
 Then, in order to use this partial measure, we need to truncate $\KK$ as $\mu$ does not measure all subsets of $\nu$ in $\KK$. By the \emph{Comparison Lemma} (see, for example,  \cite[Lemma 4.4.2]{Ze02} or \cite[Theorem 3.11]{St10}), we obtain iterates $N^*$ of $N$ and $\KK^*$ of $\KK$ (or, in fact, of a truncation $\KK\vert\xi$ of $\KK$) such that $N^* \unlhd {\KK^*}$. Note that truncations can only appear on one side of the comparison and this side has to come out longer in the end.  In particular, the iteration from $N$ to $N^*$ can only use total measures with critical point above $\nu > \gamma$ and is therefore trivial, {i.e.} we have $N = N^*$. 
 
 Suppose that $\nu > \kappa$.  Note that $\nu$ is a cardinal in $\KK^*$. As $\nu < N \cap \On$ and $N \unlhd \KK^*$, this implies that there are cardinals above $\kappa$ in $N$, contradicting the assumption that $\kappa$ is the largest cardinal in $N$.

 Now suppose that $\nu < \kappa$. The iteration from $\KK|\xi$ to $\KK^*$ cannot leave any total measures below $\kappa$ behind as $N \unlhd \KK^*$ does not have any total measures between $\gamma$ and $\kappa$. As we suppose that there is no inner model with a sequence of measurable cardinals of length $\cof{\kappa}$, this implies that the iteration from $\KK|\xi$ to $\KK^*$ is a linear iteration of $\mu$ and its images. Again, as $N \unlhd \KK^*$ does not have any total measures between $\gamma$ and $\kappa$ and $N \cap \On \geq \kappa$, this iteration needs to last at least $\kappa$-many steps by {\cite[Corollary 19.7.(b)]{MR1994835}} since $\kappa$ is a cardinal in $\VV$. Moreover,  {\cite[Corollary 19.7.(b)]{MR1994835}} shows that  $\kappa$ is inaccessible in $\KK^*$. As $N \unlhd \KK^*$, this contradicts the fact that $\kappa$ is singular in $N$. Therefore, $\mu$ is not used on the $\KK$-side of the comparison.
 
 Similarly, we can argue that no partial measure on $N$ gets used in the comparison and hence we can conclude that $N \lhd \KK$.
\end{proof}

 \begin{claim*}
  For every $x \in \left({}^{\cof{\kappa}} \kappa\right)^\KK$, there is a good premouse $N$ with $x \in N$.
 \end{claim*}

 \begin{proof}[Proof of the Claim]
  As $\KK$ satisfies the $\GCH$, there is some $\xi < (\kappa^+)^\KK=\kappa^+$ such that $x \in \KK|\xi$ and $\KK|\xi$ is a good premouse.
 \end{proof}

 Following \cite{HM19}, we say a pair $\langle M, \bar{x}\rangle$ is an \emph{active node at $\rho$} for some $\rho < \cof{\kappa}^\KK$ if there is a good premouse $N$ and some $x \in \left({}^{\cof{\kappa}} \kappa\right)^N$ with $\ran{x} \subseteq \Reg^N$, the regular cardinals in $N$, such that the following statements hold: 
 \begin{itemize}
  \item $x$ is strictly increasing and cofinal in $\kappa$. 
  
  \item $M$ is equal to the transitive collapse of $\Hull^{N}(x(\rho) \cup \{x\})$ and $\bar{x} \in M$ is the image of $x$ under  the transitive collapse.  

  \item If $\map{\pi}{M}{N}$ is the corresponding uncollapsing  map, then $\crit{\pi} = \bar{x}(\rho)$. 
  
 \end{itemize}
 In addition, we say a pair $\langle M, \bar{x}\rangle$ is an \emph{active node} if there is some ordinal $\rho<\cof{\kappa}^\KK$ such that $\langle M, \bar{x}\rangle$ is an active node at $\rho$.

  We now let $T$ denote the unique partial order defined by the following clauses: 
  \begin{enumerate}
      \item The elements of $T$ are triples of the form $\langle M, \bar{x}, s\rangle$ satisfying the following properties: 
      
      \begin{enumerate}
       \item The pair $\langle M, \bar{x}\rangle$ is either an active node or equal to the pair $\langle \emptyset, \emptyset\rangle$. 
       
       \item $s$ is an element of $({}^{{<}\cof{\kappa}}\kappa)^\KK$ with the property that the set $$\bigcup_{0 < \alpha < \kappa} s^{{-}1}(\{\alpha\})$$ is finite. 
       
       \item If $\langle M, \bar{x}\rangle$ is an active node at $\rho$, then $\dom{s} \geq \rho$.
      \end{enumerate}
      
      \item The order of $T$ is defined by $$\langle M_0, x_0, s_0\rangle ~ \leq_{T} ~  \langle M_1, x_1, s_1\rangle$$ if and only if the following statements hold: 
      \begin{enumerate}
       \item $M_0$ is the transitive collapse of $\Hull^{M_1}(x_1(\rho) \cup \{x_1\})$ for some ordinal $\rho$ and $x_0$ is the image of $x_1$ under the transitive collapse, or $M_0 = x_0 = \emptyset$. In the following, write $\rho$ for the minimal such ordinal and $\rho = -1$ if $M_0 = x_0 = \emptyset$. 
       
       \item $s_0$ is an initial segment of $s_1$. 
       
       \item There is no ordinal $\rho'$ between $\rho$ and $\dom{s_0}$ with the property that $\langle \Hull^{M_1}(x_1(\rho') \cup \{x_1\}), x_1\rangle$ transitively collapses to an active node which, in case $\rho \neq -1$, is not $\langle M_0, x_0\rangle$.  
    \end{enumerate} 
  \end{enumerate}
  
 It is now easy to see that $T$ is a tree of height $\cof{\kappa}^\KK$ that is contained in $\KK$ and has  the property that each node is splitting into $\kappa$-many  successors. 
 Moreover, each $x \in \left({}^{\cof{\kappa}} \kappa\right)^\KK$ that is strictly  increasing and cofinal in $\kappa$ with range contained in $\Reg^N$ naturally gives rise to a cofinal branch $b_x$ through $T$ and two different such elements $x,y \in \left({}^{\cof{\kappa}} \kappa\right)^\KK$ give rise to different branches $b_x$ and $b_y$. 
 Hence, the fact that the $\GCH$ holds in $\KK$ implies that the set of cofinal branches through $T$ has cardinality at least $$\left(\kappa^{\cof{\kappa}}\right)^\KK ~ = ~  (\kappa^+)^\KK ~ = ~  \kappa^+.$$ 

 \begin{claim*}
  Let $b$ be a cofinal branch through $T$ and let $\langle \calR_b, x_b \rangle$ denote the direct limit of models along $b$. Then $\calR_b$ is well-founded and we can identify it with its transitive collapse. Moreover, $\calR_b \lhd K$.
 \end{claim*}

 \begin{proof}[Proof of the Claim]
  As the proof of the well-foundedness of $\calR_b$ is easier, we focus on the argument that $\calR_b \lhd K$.
  By our first claim, it suffices to show that $\calR_b$ is a good premouse. 
  We obtain iterability for $\calR_b$ by reflecting countable elementary substructures of $\calR_b$ into models in the tree $T$, as in \cite{HM19}, using the fact that $$\cof{\cof{\kappa}^\KK} ~ = ~ \cof{\kappa} ~ > ~  \omega$$ (see {\cite[Lemma 3.7(ii)]{MR1940513}}). 
  In the following, write $\seq{M_\rho}{\rho < \cof{\kappa}^\KK}$ for the sequence of models appearing in active nodes at $\rho$ along the branch $b$. Then the definition of $T$ ensures that for every $\xi < \kappa$, there is some $\rho < \cof{\kappa}^\KK$ such that if $M_\rho$ is the transitive collapse of $\Hull^{N_\rho}(x_\rho(\rho) \cup \{x_\rho\})$ for some good premouse $N_\rho$ and some $x_\rho \in \left({}^{\cof{\kappa}} \kappa\right)^{N_\rho}$, then $x_\rho(\rho) > \xi$. Therefore, we know that $\kappa \subseteq \calR_b$ and elementarity implies that $\kappa+1 \subseteq \calR_b$.  
  Our setup also directly ensures that   $\betrag{\calR_b}= \kappa$. As $N_\rho \lhd K$ and the critical point of the inverse of the collapse embedding $\pi_\rho \colon M_\rho \rightarrow N_\rho$ is at least $x_\rho(\rho)$, this also shows that $N_\rho | \gamma^{++} = K | \gamma^{++}$,  where $\gamma < \kappa$ is the supremum of the measurable cardinals below $\kappa$ in $\KK$ and $N_\rho$. Moreover, we know that  $\kappa$ has cofinality $\cof{\kappa}^\KK$ in $\calR_b$, as witnessed by $x_b$. Finally, $\kappa$ is the largest cardinal in $\calR_b$ by elementarity as $\kappa = \sup(\ran{x_\rho})$ is the largest cardinal in $N_\rho$ for all $\rho < \cof{\kappa}$.
 \end{proof}

 \begin{claim*}
  The set $T$ is $\Sigma_1$-definable with parameters in $\HH{\kappa} \cup \{\kappa\}$.
 \end{claim*}
 
 \begin{proof}[Proof of the Claim]
  It clearly suffices to show that the set of all good mice $N$ is definable in the above way. 
  %
  As there is no inner model with $\cof{\kappa}$-many measurable cardinals, the mice we consider are simple and therefore iterability for $N$ is $\Sigma_1$-definable from the parameter $\omega_1$, using {\cite[Theorem 4.5.5]{Ze02}}. 
  All other conditions can obviously be stated by $\Sigma_1$-formulas using the parameters $\kappa$ and $K\vert\gamma^{++} \in \HH{\kappa}$. 
 \end{proof}

 \begin{claim*}
  There is an injection $\map{i}{T}{\HH{\kappa}\cap\POT{\kappa}}$ that is definable by a $\Sigma_1$-formula with parameters in $\HH{\kappa}\cup\{\kappa\}$. 
 \end{claim*}
 
 \begin{proof}[Proof of the Claim]
  It clearly suffices to construct an injection from the set of all active nodes to $\HH{\kappa}\cap\POT{\kappa}$. Let $\langle M,\bar{x}\rangle$ be an active node at some $\rho<\cof{\kappa}^\KK$. 
  Since $M$ is the transitive collapse of an elementary submodel of some good premouse $N=\langle\mathrm{J}^{\vec E}_{\alpha},\vec E\rangle$, we know that $M$ is of the form $\langle\mathrm{J}^A_\epsilon,A\rangle$ and there is a well-ordering $\lhd$ of $M$ that is definable in $M$. Let $\map{\tau}{\langle  M,\lhd\rangle}{\langle\lambda,<\rangle}$ denote the corresponding transitive collapse and associate $\langle M,\bar{x}\rangle$ with the element  
  \begin{equation*}
   \begin{split}
       & \Set{\goedel{0}{\goedel{\alpha}{\beta}}}{\alpha,\beta<\lambda, ~ \tau^{{-}1}(\alpha)\in\tau^{{-}1}(\beta)} \\
       & \cup ~ \Set{\goedel{1}{\alpha}}{\alpha<\lambda, ~ \tau^{{-}1}(\alpha)\in A} \\ 
       & ~ \cup ~ \Set{\goedel{2}{\alpha}}{\alpha<\lambda, ~ \tau^{{-}1}(\alpha)\in\bar{x}} 
   \end{split}
  \end{equation*}
  of $\HH{\kappa}\cap\POT{\kappa}$. It is now easy to see that the resulting injection is definable in the desired way. 
 \end{proof}

 \begin{claim*}
 No 
 countably closed forcing adds a cofinal branch through $T$. 
 \end{claim*}
 
 \begin{proof}[Proof of the Claim]
 Let $\PPP$ be a countably closed forcing notion and let $G$ be $\PPP$-generic over $\VV$.  Suppose, towards a contradiction,  that there is a cofinal branch $b$ through $T$ in $\VV[G]$ that is not contained in $\VV$.
 By considering the direct limit of the active nodes along $b$ and using the fact that $\cof{\kappa}^\VV$ has uncountable cofinality in $\VV[G]$, we obtain a pair $\langle\calR_b, x_b\rangle$ such that $b$ (modulo some choice of an almost zero sequence $s$) can be recovered from $\calR_b$ and $x_b$ via the transitive collapses of models of the form \[ \Hull^{\calR_b}(x_b(\rho) \cup \{x_b\}), \] for $\rho < \cof{\kappa}^\KK$. 
 As $\cof{\kappa}^\VV$ has uncountable cofinality in $\VV[G]$, the argument in one of our previous claims then shows that  $\calR_b \lhd \KK^{V[G]}$ holds in $V[G]$.
 By \cite[Theorem 7.4.11]{Ze02}, the core model $\KK$ is forcing absolute, {i.e.} we have $\KK^\VV = \KK^{\VV[G]}$.  Therefore, we know that $\calR_b$ and hence $b$ is already an element of $\VV$, a contradiction. 
 \end{proof}

 Fix a strictly increasing,  cofinal function $\map{c}{\cof{\kappa}}{\cof{\kappa}^\KK}$ and let $T_*$ denote the unique partial order defined by the following clauses: 
 \begin{enumerate}
     \item The elements of $T_*$ are functions $t$ such that $\dom{t}\in\cof{\kappa}$ and the following statements hold:
     \begin{enumerate}
         \item If $\alpha\in\dom{t}$, then $t(\alpha)$ is a branch through $T$ of order-type $c(\alpha)$. 
         
         \item If $\alpha<\beta\in\dom{t}$, then $t(\alpha)$ is an initial segment of $t(\beta)$. 
     \end{enumerate}
     
     \item The ordering of $T_*$ is given by inclusion. 
 \end{enumerate}
 
 It then follows that $T_*$ is a tree of height $\cof{\kappa}$ with the property that every node has $\kappa$-many successors. 
 Since the tree $T$ has at least $\kappa^+$-many branches, it follows that $T_*$ also has at least $\kappa^+$-many branches. 
 Moreover, by using the injection $i$, it is possible to construct an injection $\map{i_*}{T_*}{\HH{\kappa}\cap\POT{\kappa}}$ with $\emptyset\notin\ran{i_*}$ that is definable by a $\Sigma_1$-formula with parameters in $\HH{\kappa}\cup\{\kappa\}$. 
 Finally, the above computations also imply that forcing with a countably closed partial order does not add a new cofinal branch to $T_*$.

 Define $D$ to be the set of all subsets of $\kappa$ of the form $$y_b ~ = ~ \Set{\goedel{\alpha}{\goedel{\beta}{\sup(i_*(b\restriction\alpha))}}}{\alpha<\cof{\kappa}, ~ \beta\in i_*(b\restriction\alpha)}$$ for some function $b$ with domain $\cof{\kappa}$ and the property that $b\restriction\alpha\in T_*$ for all $\alpha<\cof{\kappa}$. 
 Since the fact that the tree $T_*$ has at least $\kappa^+$-many cofinal branches implies that $D$ has cardinality greater than $\kappa$ and the above computations show that $D$ is definable by a $\Sigma_1$-formula with parameters in $\HH{\kappa}\cup\{\kappa\}$, our assumption yields a perfect  embedding $\map{\iota}{{}^{\cof{\kappa}}\kappa}{\POT{\kappa}}$ with $\ran{\iota}\subseteq D$. 
 Using the fact that $\cof{\kappa}$ is uncountable, a routine construction now allows us to find 
 \begin{itemize}
     \item a system $\seq{u_s}{s\in{}^{{<}\cof{\kappa}}2}$ of elements of ${}^{{<}\cof{\kappa}}2$, 
     
     \item a strictly increasing sequence $\seq{\kappa_\alpha}{\alpha<\cof{\kappa}}$ that is cofinal in $\kappa$, and 
     
     \item a system $\seq{a_s}{s\in{}^{{<}\cof{\kappa}}2}$ of bounded subsets of $\kappa$ 
 \end{itemize}
 such that the following statements hold for all $s,t\in{}^{{<}\cof{\kappa}}2$: 
 \begin{enumerate}
     \item If $\length{s}{}=\length{t}{}$, then $\length{u_s}{}=\length{u_t}{}$. 
     
     \item $a_s$ is a subset of $\kappa_{\length{s}{}}$. 
     
     \item If $s\subseteq t$, then $a_s=a_t\cap\kappa_{\length{s}{}}$ and $u_s\subseteq u_t$. 
     
     \item $a_{s^\frown\langle 0\rangle}\neq a_{s^\frown\langle 1\rangle}$ and $u_{s^\frown\langle 0\rangle}\neq u_{s^\frown\langle 1\rangle}$. 
     
     \item $\iota[\Set{x\in{}^\cof{\kappa}2}{x\restriction\length{s}{}=s}]=\Set{y\in\ran{\iota}}{y\cap\kappa_{\length{s}{}}=a_s}$. 
     
     \item If $\alpha<\length{s}{}$, then there are $\gamma\leq\delta<\kappa_{\length{s}{}}$ with $\goedel{\alpha}{\goedel{\gamma}{\delta}}\in a_s$. 
 \end{enumerate}

  Now, let $G$ be $\Add{\cof{\kappa}}{1}$-generic over $\VV$. Set $x_G=\bigcup G\in({}^\cof{\kappa}2)^{\VV[G]}$ and $$y_G ~ = ~ \bigcup\Set{a_{x_G\restriction\alpha}}{\alpha<\cof{\kappa}} ~ \in ~ \POT{\kappa}^{\VV[G]}.$$ 
  In this situation, our construction ensures that there is a function $b_G$ in $\VV[G]$ such that $\dom{b_G}=\cof{\kappa}$, $b_G\restriction\alpha\in T_*$ for all $\alpha<\cof{\kappa}$ and $y_G=y_{b_G}$. 
  By our earlier observations, the cofinal branch through $T_*$ induced by $b_G$ is contained in $\VV$ and hence $b_G$ is an element of $\VV$. 
  But this implies that $y_G$ is also contained in the ground model $\VV$. 
  Since $x_G$ is the unique element $x$ of $({}^{\cof{\kappa}}2)^{\VV[G]}$ with the property that $y_G\cap\kappa_\alpha=a_{x\restriction\alpha}$ holds for all $\alpha<\cof{\kappa}$, we can now conclude that $x_G$ is contained in $\VV$, a contradiction. 
\end{proof}


\section{Almost disjoint families at limits of measurable cardinals}

 We now proceed by using the techniques developed in Section \ref{section:PSPlimitsMeasurables} to show that large almost disjoint families at cardinals with sufficiently strong large cardinal properties are not simply definable.

\begin{proof}[Proof of Theorem \ref{theorem:AlmostDisjoint}]
 Let $\kappa$ be an iterable cardinal that is a limit of measurable cardinals, let $z$ be an element of $\HH{\kappa}$ and let $A$ be a subset of $\POT{\kappa}$ of cardinality greater than $\kappa$ with the property that there exists a $\Sigma_1$-formula $\varphi(v_0,v_1,v_2)$ with $A=\Set{y\subseteq\kappa}{\varphi(\kappa,y,z)}$. 
 Assume, towards a contradiction, that $A$ is an almost disjoint family in $\POT{\kappa}$. 
 Since $\kappa$ is an inaccessible cardinal and the collection of all bounded subsets of $\kappa$ is definable by a $\Sigma_0$-formula with parameter $\kappa$, we may then also assume that $A$ consists of unbounded subsets of $\kappa$. 
 Fix an inaccessible cardinal $\lambda<\kappa$ with $z\in\HH{\lambda}$ and use Lemma \ref{lemma:TechnicalLemmaIterationsTree} to obtain $x\in A$ and systems  $\seq{\nu_s}{s\in{}^{{<}\kappa}\kappa}$, $\seq{\kappa_s}{s\in{}^{{<}\kappa}\kappa}$, $\seq{U_s}{s\in{}^{{<}\kappa}\kappa}$ and $\seq{I_s}{s\in{}^{{<}\kappa}\kappa}$ with $\lambda<\kappa_\emptyset$ and the properties listed in the lemma. 
 %
 Then there exists an $\Add{\lambda}{1}$-nice name $\dot{x}$ for an unbounded subset of $\kappa$ with the property that $\dot{x}^G=i^{I_{c_G}}_{0,\infty}(x)$ holds whenever $G$ is $\Add{\lambda}{1}$-generic over $\VV$, $c_G=\bigcup G\in({}^\lambda 2)^{\VV[G]}$ and $I_{c_G}$ is the unique linear iteration of $\langle\VV,\Set{U_{c_G\restriction\xi}}{\xi<\lambda}\rangle$ of length $\sup_{\xi<\lambda}\length{I_{c_G\restriction\xi}}{}$ in $\VV[G]$ with $U^{I_G}_\alpha=U^{I_{c_G\restriction\xi}}_\alpha$ for all $\xi<\lambda$ and $\alpha<\length{I_{c_G\restriction\xi}}{}$. 
 Note that, by Lemma \ref{lemma:LongIterationsGenericExtensions}, the elementarity of $i^{I_{c_G}}_{0,\infty}$ and the upwards absoluteness of $\Sigma_1$-statements between $M^{I_{c_G}}_\infty$ and $\VV[G]$ ensures that 
 \begin{equation}\label{equation:FormulaForced}
  \mathbbm{1}_{\Add{\lambda}{1}}\Vdash\varphi(\check{\kappa},\dot{x},\check{z})
 \end{equation}
  holds in $\VV$.

 \begin{claim*}
  If $G_0\times G_1$ is $(\Add{\lambda}{1}\times\Add{\lambda}{1})$-generic over $\VV$, then $\dot{x}^{G_0}\neq\dot{x}^{G_1}$. 
 \end{claim*}
 
 \begin{proof}[Proof of the Claim]
  Given $i<2$, the absoluteness of the iterated ultrapower construction ensures that $(I_{c_{G_i}})^{\VV[G_i]}=(I_{c_{G_i}})^{\VV[G_0,G_1]}$ holds and this implies that  $$\dot{x}^{G_i} ~ = ~ (i^{I_{c_{G_i}}}_{0,\infty}(x))^{\VV[G_0,G_1]}.$$ 
  Since mutual genericity implies that $c_{G_0}\neq c_{G_1}$, the desired inequality now directly follows from an application of statement \eqref{item:SplitLong} of Lemma \ref{lemma:LongIterationsGenericExtensions} in $\VV[G_0,G_1]$. 
 \end{proof}

 Pick an elementary submodel $M_0$ of $\HH{\kappa^+}$ of cardinality $\kappa$ with ${}^{{<}\kappa}M_0\subseteq M_0$ that contains $\HH{\kappa}$ and all objects listed above. 
 Since iterable cardinals are weakly compact, we can find a transitive set $M_1$ of cardinality $\kappa$ and an elementary embedding $\map{j}{M_0}{M_1}$ with $\crit{j}=\kappa$ (see {\cite[Theorem 1.3]{MR1133077}}). 
 Then $j(\dot{x})$ is an $\Add{\lambda}{1}$-name for an unbounded subset of $j(\kappa)$ and there is a canonical $\Add{\lambda}{1}$-name $\dot{\gamma}$ for an ordinal in the interval $[\kappa,j(\kappa))$ with the property that $$\dot{\gamma}^G ~ = ~ \min(j(\dot{x})^G\setminus\kappa)$$ holds whenever $G$ is $\Add{\lambda}{1}$-generic over $\VV$.

 \begin{claim*}
  If $G_0\times G_1$ is $(\Add{\lambda}{1}\times\Add{\lambda}{1})$-generic over $\VV$, then $\dot{\gamma}^{G_0}\neq\dot{\gamma}^{G_1}$. 
 \end{claim*}
 
 \begin{proof}[Proof of the Claim]
  Given an $\Add{\lambda}{1}$-name $\dot{a}$, let $\dot{a}_l$ and $\dot{a}_r$ denote the canonical $(\Add{\lambda}{1}\times\Add{\lambda}{1})$-names such that $\dot{a}_l^{G_0\times G_1}=\dot{a}^{G_0}$ and $\dot{a}_r^{G_0\times G_1}=\dot{a}^{G_1}$ holds whenever $G_0\times G_1$ is $(\Add{\lambda}{1}\times\Add{\lambda}{1})$-generic over $\VV$. 
  Given an $\Add{\lambda}{1}$-name $\dot{a}$ in $M_0$, we  then have $j(\dot{a}_l)=j(\dot{a})_l$ and $j(\dot{a}_r)=j(\dot{a})_r$. 

  Assume, towards a contradiction, that $$\langle p,q\rangle\Vdash_{\Add{\lambda}{1}\times\Add{\lambda}{1}}\anf{\dot{\gamma}_l=\dot{\gamma}_r}$$ holds for some condition $\langle p,q\rangle$ in $(\Add{\lambda}{1}\times\Add{\lambda}{1})$.

  \begin{subclaim*}
   $\langle p,q\rangle\Vdash_{\Add{\lambda}{1}\times\Add{\lambda}{1}}\anf{\textit{$\dot{x}_l\cap\dot{x}_r$ is unbounded in $\check{\kappa}$}}.$
  \end{subclaim*}
  
  \begin{proof}[Proof of the Subclaim]
   Let $G_0\times G_1$ be $(\Add{\lambda}{1}\times\Add{\lambda}{1})$-generic over $\VV$ with $\langle p,q\rangle\in G_0\times G_1$. 
   By standard arguments, there exists an elementary embedding $$\map{j_*}{M_0[G_0,G_1]}{M_1[G_0,G_1]}$$ with $j_*(\dot{b}^{G_0\times G_1})=j(\dot{b})^{G_0\times G_1}$ for every $(\Add{\lambda}{1}\times\Add{\lambda}{1})$-name $\dot{b}$  in $M_0$. 
   Then our assumptions ensure that 
   \begin{equation*}
    \begin{split}
    \dot{\gamma}^{G_0} ~ = ~ \dot{\gamma}^{G_1}  ~  \in ~ & ~ j(\dot{x})^{G_0} ~ \cap ~ j(\dot{x})^{G_1} ~ \cap ~ [\kappa,j(\kappa)) \\ 
     & = ~ j_*(\dot{x}_l^{G_0\times G_1}) ~ \cap ~ j_*(\dot{x}_r^{G_0\times G_1}) ~ \cap ~ [\kappa,j_*(\kappa)) ~ \neq ~ \emptyset. 
    \end{split}
   \end{equation*}
   In particular, if $\alpha<\kappa$, then the elementarity of $j_*$ and the fact that $j_*(\alpha)=\alpha$ directly imply that $$\dot{x}_l^{G_0\times G_1} ~ \cap ~ \dot{x}_r^{G_0\times G_1} ~ \cap ~ (\alpha,\kappa) ~ \neq ~ \emptyset.$$
   This proves the statement of the subclaim. 
  \end{proof}

 We now use the fact that  $\kappa$ is an iterable cardinal to find a transitive model $M$ of $\ZFC^-$ of cardinality $\kappa$ with $M_0\in M$  and a weakly amenable $M$-ultrafilter $F$ on $\kappa$ such that $\langle M,F\rangle$ is iterable.
 Pick an elementary submodel $\langle X,\in,\bar{F}\rangle$  of $\langle M,\in,F\rangle$ of cardinality $\lambda$ with ${}^{{<}\lambda}X\subseteq X$ that contains $\HH{\lambda}$, $M_0$ and all other relevant objects. 
  Let $\map{\pi}{X}{N_0}$ denote the corresponding transitive collapse and set $F_0=\pi[\bar{F}]$. 
  By {\cite[Theorem 19.15]{MR1994835}}, we know that $\langle N_0,F_0\rangle$ is iterable. 
  Let $\langle N_1,F_1\rangle$ denote the $\kappa$-th iterate of $\langle N_0,F_0\rangle$ and let $\map{i}{N_0}{N_1}$ denote the corresponding elementary embedding. 
  Then $(i\circ\pi)(\kappa)=\kappa$, $(i\circ\pi)(z)=z$, $(i\circ\pi)(\langle p,q\rangle)=\langle p,q\rangle$ and $\HH{\pi(\kappa)}^{N_0}=\HH{\pi(\kappa)}^{N_1}$. 
  Since $\Add{\lambda}{1}\times\Add{\lambda}{1}$ is ${<}\lambda$-closed and a subset of $N_0$, 
  we know that $N_0$ contains all sequences of conditions in $\Add{\lambda}{1}\times\Add{\lambda}{1}$ of length less than $\lambda$ and therefore 
  the fact that $\betrag{N_0}=\lambda$ allows us to find a filter $H_0\times H_1$ on $\Add{\lambda}{1}\times\Add{\lambda}{1}$ that contains $\langle p,q\rangle$ and is generic over $N_0$ by constructing a descending sequence of conditions in $\Add{\lambda}{1}\times\Add{\lambda}{1}$ below $\langle p,q\rangle$  that has  length $\lambda$ and  intersects all dense subsets of $\Add{\lambda}{1}\times\Add{\lambda}{1}$ contained in $N_0$. 
  Moreover, since $\HH{\lambda^+}^{N_1}\subseteq N_0$, we know that the filter $H_0\times H_1$ is also generic over $N_1$. 
   
   Given $i<2$, we now define $x_i=(i\circ\pi)(\dot{x})^{H_i}$. 
  Set $N=(i\circ\pi)(M_0)$. Then $\Add{\lambda}{1}\subseteq N$, $(i\circ\pi)(\dot{x})\in N$ and $H_0\times H_1$ is  $(\Add{\lambda}{1}\times\Add{\lambda}{1})$-generic over $N$. 
   Since our first claim and the above subclaim show that $$\langle p,q\rangle\Vdash_{\Add{\lambda}{1}\times\Add{\lambda}{1}}\anf{\textit{$\dot{x}_l\neq\dot{x}_r$ and $\dot{x}_l\cap\dot{x}_r$ is unbounded in $\check{\kappa}$}}$$ holds in $M_0$, elementarity implies that $x_0$ and $x_1$ are distinct subsets of $\kappa$ and    $x_0\cap x_1$ is unbounded in $\kappa$. 
   Moreover, using \eqref{equation:FormulaForced}, $\Sigma_1$-upwards absoluteness and the fact that $\Sigma_1$-statements in the forcing language can be expressed by $\Sigma_1$-formulas, we know that $$\langle p,q\rangle\Vdash_{\Add{\lambda}{1}\times\Add{\lambda}{1}}\anf{\varphi(\check{\kappa},\dot{x}_l,\check{z})\wedge\varphi(\check{\kappa},\dot{x}_r,\check{z})}$$ holds in $M_0$ and therefore elementarity allows us to conclude that $\varphi(\kappa,x_i,z)$ holds in  $N[H_0,H_1]$ for all $i<2$. 
   By $\Sigma_1$-upwards absoluteness, this implies that $x_0$ and $x_1$ are distinct elements of $A$, contradicting the fact that $A$ is an almost disjoint family.  
 \end{proof}

 Now, let $G$ be $\Add{\lambda}{\kappa^+}$-generic over $\VV$. 
 Since $\lambda$ is inaccessible, the model $\VV[G]$ has the same cardinals as $\VV$. 
 Let $\seq{G_\delta}{\delta<\kappa^+}$ denote the induced sequence of filters on $\Add{\lambda}{1}$. 
 Given $\delta<\varepsilon<\kappa^+$, the filter $G_\delta\times G_\varepsilon$ on $\Add{\lambda}{1}\times\Add{\lambda}{1}$ is generic over $\VV$ and therefore the previous claim implies that $\dot{\gamma}^{G_\delta}\neq\dot{\gamma}^{G_\varepsilon}$. 
 In particular, the map $$\Map{\iota}{\kappa^+}{j(\kappa)}{\delta}{\dot{\gamma}^{G_\delta}}$$ is an injection. 
 Since $j(\kappa)<\kappa^+$, this yields a contradiction. 
\end{proof}

 The conclusion of Theorem \ref{theorem:AlmostDisjoint} provably does not generalize to $\Sigma_1$-definitions using arbitrary subsets of $\kappa$ as parameters. If $\kappa$ is an infinite cardinal and $z\subseteq\kappa$ codes an injective sequence $\seq{s_\beta}{\beta<\kappa}$ of elements of ${}^{{<}\kappa}2$ with the property that the set $$I ~ = ~ \Set{x\in{}^\kappa 2}{\forall\alpha<\kappa ~ \exists\beta<\kappa ~ x\restriction\alpha=s_\beta}$$ has cardinality greater than $\kappa$, then the collection $\Set{\Set{\beta<\kappa}{s_\beta\subseteq x}}{x\in I}$ is an almost disjoint family of cardinality greater than $\kappa$ that is definable by a $\Sigma_1$-formula with parameter $z$. 
  Note that such sequences exist for every strong limit cardinal $\kappa$, or, more generally, for every cardinal $\kappa$ that is a strong limit cardinal in an inner model $M$ satisfying $(2^\kappa)^M\geq\kappa^+$.


\section{Long well-orders at limits of measurable cardinals}\label{section:LongWOSingular}

In order to motivate the statement of Theorem \ref{MainTheorem:LimitMeasurables}, we first show how classical results of Dehornoy can easily be used to show that, if $\kappa$ is a limit of measurable cardinals, then no well-ordering of $\POT{\kappa}$ is  definable by a $\Sigma_1$-formula with parameters in $\HH{\kappa}\cup\{\kappa\}$. Moreover, if $\kappa$  has uncountable cofinality, then we can also easily show that no injection from $\kappa^+$ into $\POT{\kappa}$ is definable in this way. 
 This non-definability result will be a direct consequence of the following theorem.

\begin{theorem}\label{theorem:Dehorney}
  If $\delta$ is a measurable cardinal,  $z\in\HH{\delta}$ and  $\nu$ is a cardinal with $\cof{\nu}\neq\delta$ and $\mu^\delta<\nu$ for all $\mu<\nu$, then the following statements hold for $\kappa\in\{\nu,\nu^+\}$: 
  \begin{enumerate}
      \item No well-ordering of $\POT{\kappa}$ is definable by a $\Sigma_1$-formula with parameters $\nu$, $\nu^+$ and $z$. 
      
     \item 
     If $\cof{\kappa}>\omega$, then no injection from $\kappa^+$ into $\POT{\kappa}$ is definable by a $\Sigma_1$-formula with parameters $\kappa$ and $z$. 
  \end{enumerate}
\end{theorem}

The proof of the above theorem is based on two standard results about measurable cardinals. 
 A proof of the first of these lemmas is contained in the proof of {\cite[Lemma 1.3]{MR3845129}}:

\begin{lemma}\label{lemma:UltrapowerEmbeddingFixedPoints}
 Let $U$ be a normal ultrafilter on a measurable cardinal $\delta$ and let $\nu>\delta$ be a cardinal with $\cof{\nu}\neq\delta$ and $\mu^\delta<\nu$ for all $\mu<\nu$. If $\map{j}{\VV}{\Ult{\VV}{U}}$ is the induced ultrapower embedding, than $j(\nu)=\nu$ and $j(\nu^+)=\nu^+$. \qed 
\end{lemma}

 \begin{lemma}\label{lemma:FixedPointsIteration}
  Let $U$ be a normal ultrafilter on a measurable cardinal $\delta$ and let $$\langle\seq{N_\alpha}{\alpha\in\On}, \seq{\map{j_{\alpha,\beta}}{N_\alpha}{N_\beta}}{\alpha\leq\beta\in\On}\rangle$$ denote the system of iterated ultrapowers of $\langle\VV,\in,U\rangle$. If $\nu$ is a cardinal with $\cof{\nu}\neq\delta$ and  $\mu^\delta<\nu$ for all $\mu<\nu$, then $j_{0,\alpha}(\kappa)=\kappa$ holds for   $\kappa\in\{\nu,\nu^+\}$ and all $\alpha<\kappa$.  
 \end{lemma}
 
 \begin{proof}
  We start by using induction to show that $j_{0,\alpha}(\nu)=\nu$ holds for all $\alpha<\nu$. 
   In the successor case, the desired conclusion follows directly from the induction hypothesis and an application of Lemma \ref{lemma:UltrapowerEmbeddingFixedPoints} in $N_\alpha$. Hence, we may assume that $\alpha$ is a limit ordinal. Pick $\bar{\alpha}<\alpha$ and $\xi<\nu$. Then elementarity allows us to apply {\cite[Corollary 19.7.(a)]{MR1994835}} in $N_{\bar{\alpha}}$ to conclude that $j_{\bar{\alpha},\alpha}(\xi)<\nu$. Since every element of $j_{0,\alpha}(\nu)\geq\nu$ is of the form $j_{\bar{\alpha},\alpha}(\xi)$ for some $\bar{\alpha}<\alpha$ and $\xi<j_{0,\bar{\alpha}}(\nu)=\nu$, these computations show that $j_{0,\alpha}(\nu)=\nu$ holds. 
  
  Next, we inductively show that $j_{0,\alpha}(\nu)<\nu^+$ holds for all $\alpha<\nu^+$, where the successor step is again a direct consequence of the induction hypothesis and Lemma \ref{lemma:UltrapowerEmbeddingFixedPoints}. In the other case, if $\alpha\in\nu^+\cap\Lim$ and $j_{0,\bar{\alpha}}(\nu)<\nu^+$ holds for all $\bar{\alpha}<\alpha$, then every element of $j_{0,\alpha}(\nu)$ is of the form $j_{\bar{\alpha},\alpha}(\xi)$ with $\bar{\alpha}<\alpha$ and $\xi<j_{0,\bar{\alpha}}(\nu)$ and this shows that $\betrag{j_{0,\alpha}(\nu)}\leq\nu\cdot\betrag{\alpha}<\nu^+$. 
  
  Finally, we have $\nu^+\leq j_{0,\alpha}(\nu^+)\leq\betrag{j_{0,\alpha}(\nu)}^+$ for all $\alpha<\nu^+$. Since the above computations show that $\betrag{j_{0,\alpha}(\nu)}=\nu$ holds for all $\alpha<\nu^+$, this shows that $j_{0,\alpha}(\nu^+)=\nu^+$ holds for all $\alpha<\nu^+$. 
 \end{proof}

\begin{proof}[Proof of Theorem \ref{theorem:Dehorney}]
  Let $U$ be a normal ultrafilter on a measurable cardinal $\delta$ and let $$\langle\seq{N_\alpha}{\alpha\in\On}, ~ \seq{\map{j_{\alpha,\beta}}{N_\alpha}{N_\beta}}{\alpha\leq\beta\in\On}\rangle$$ denote the system of iterated ultrapowers of $\langle\VV,\in,U\rangle$.  
  Moreover, for every $\alpha\in\Lim$, we define  $M_\alpha=\bigcap\Set{N_\xi}{\xi<\alpha}$. Then {\cite[Proposition 1.6.1]{DEHORNOY1978109}} shows that each $M_\alpha$ is an inner  model of $\ZF$.

 (i) Assume, towards a contradiction, that there is a $\Sigma_1$-formula $\varphi(v_0,\ldots,v_4)$ with the property that $$\lhd ~ = ~ \Set{\langle x,y\rangle}{\varphi(x,y,z,\nu,\nu^+)}$$ is a well-ordering of $\POT{\kappa}$. 
 For all $\alpha\in\On$, we define $\lhd_\alpha=j_{0,\alpha}({\lhd})$. 
 Given $\alpha<\omega^2$, Lemma \ref{lemma:FixedPointsIteration} implies that $j_{0,\alpha}(\nu)=\nu$ and $j_{0,\alpha}(\nu^+)=\nu^+$. 
 In particular, elementarity implies that $\lhd_\alpha$ is a well-ordering of $\POT{\kappa}^{N_\alpha}$ and the sequence $\seq{\lhd_{\alpha+\beta}}{\beta<\omega^2}$ is an element of $N_\alpha$.
  By our assumptions, elementarity and $\Sigma_1$-upwards absoluteness imply that ${\lhd_\beta}\subseteq{\lhd_\alpha}\subseteq{\lhd}$ for all $\alpha\leq\beta<\omega^2$. 
 Define ${\blacktriangleleft}=\bigcap\Set{\lhd_\alpha}{\alpha<\omega^2}$. If $\alpha<\omega^2$, then ${\blacktriangleleft}=\bigcap\Set{\lhd_{\alpha+\beta}}{\beta<\omega^2}$ and therefore ${\blacktriangleleft}\in N_\alpha$. This shows that ${\blacktriangleleft}$ is an element of $M_{\omega^2}$ and it follows that ${\blacktriangleleft}$ is a well-ordering of $\POT{\kappa}^{M_{\omega^2}}$. 
  But this yields a  contradiction, because {\cite[Theorem 5.3.4]{DEHORNOY1978109}} shows that $M_{\omega^2}$ contains a subset $\calG_{\omega^2}$ of $\POT{j_{0,\omega^2}(\delta)}$ with the property that $M_{\omega^2}$ does not contain a well-ordering of the set  $\calG_{\omega^2}$.

 (ii)  Assume, towards a contradiction, that $\cof{\kappa}>\omega$ and  there is an injection $\map{\iota}{\kappa^+}{\POT{\kappa}}$ that is definable by a $\Sigma_1$-formula $\varphi(v_0,\ldots,v_3)$ and the  parameters $\kappa$ and $z$. 
  
  \begin{claim*}
   If $\alpha<\kappa$, then $j_{0,\alpha}(\iota)=\iota$. 
  \end{claim*}

  \begin{proof}[Proof of the Claim]
   Since Lemma \ref{lemma:FixedPointsIteration} shows that $j_{0,\alpha}(\kappa)=\kappa$,  we also know that $j_{0,\alpha}(\kappa^+)=\kappa^+$ and therefore elementarity implies that $j_{0,\alpha}(\iota)$ is an injection from $\kappa^+$ into $\POT{\kappa}$ that is definable in $N_\alpha$ by the formula $\varphi$ and the parameters $\kappa$ and $z$. 
   But then $\Sigma_1$-upwards absoluteness implies that $j_{0,\alpha}(\iota)\subseteq\iota$ and this allows us to conclude that  $j_{0,\alpha}(\iota)=\iota$.  
  \end{proof}   
  
  The above claim directly implies that the injection $\iota$ is an element of $ M_\kappa$. 
  By {\cite[Theorem B.(i)]{DEHORNOY1978109}}, the fact that $\cof{\kappa}>\omega$ implies that $N_\kappa = M_\kappa=\bigcap_{\alpha<\kappa}N_\alpha$ and hence $\betrag{\POT{\kappa}^{N_\kappa}}\geq\kappa^+$. 
  Since $N_\kappa$ is a direct limit and $j_{0,\kappa}(\delta)=\kappa$, we also know that $$\POT{\kappa}^{N_\kappa} ~ = ~ \Set{j_{\alpha,\kappa}(x)}{\alpha<\kappa, ~ x\in\POT{j_{0,\alpha}(\delta)}^{N_\alpha}}.$$ 
  But our assumptions imply that $2^\delta<\kappa$ and therefore $$\betrag{\POT{j_{0,\alpha}(\delta)}^{N_\alpha}} ~ \leq ~  j_{0,\alpha}(2^\delta) ~ < ~ j_{0,\alpha}(\kappa) ~ = ~ \kappa$$ holds for all $\alpha<\kappa$.
  We can now conclude that $\betrag{\POT{\kappa}^{N_\kappa}}=\kappa$, a contradiction. 
\end{proof}

\begin{corollary}\label{corollary:NoSigma1WOorGoodLong}
 Let $\kappa$ be a limit of measurable cardinals. 
 \begin{enumerate}
     \item No well-ordering of $\POT{\kappa}$ is definable by a $\Sigma_1$-formula with parameters in $\HH{\kappa}\cup\{\kappa,\kappa^+\}$.
     
     \item If $\cof{\kappa}>\omega$, then no injection from $\kappa^+$ into $\POT{\kappa}$ is definable by a $\Sigma_1$-formula with parameters in $\HH{\kappa}\cup\{\kappa\}$. \qed 
 \end{enumerate}
\end{corollary}

 We now proceed by proving our result on the non-existence of long $\Sigma_1$-well-orders.

\begin{proof}[Proof of Theorem \ref{MainTheorem:LimitMeasurables}]
 Let $\kappa$ be a limit of measurable cardinals with $\cof{\kappa}=\omega$, let $D$ be a subset of $\POT{\kappa}$ of cardinality greater than $\kappa$ and let $\lhd$ be a well-ordering of $D$ that is definable by a $\Sigma_1$-formula with parameter $\kappa$. Then $D$ is definable in the same way and we can pick $\Sigma_1$-formulas $\varphi(v_0,v_1)$ and $\psi(v_0,v_1,v_2)$ with $D=\Set{x}{\varphi(x,\kappa)}$ and $\lhd = \Set{\langle x,y\rangle}{\psi(x,y,\kappa)}$.  
  Now, use Lemma \ref{lemma:TechnicalLemmaIterationsTree} to find $x\in D$ and systems  $\seq{\nu_s}{s\in{}^{{<}\omega}\kappa}$, $\seq{\kappa_s}{s\in{}^{{<}\omega}\kappa}$, $\seq{U_s}{s\in{}^{{<}\omega}\kappa}$ and $\seq{I_s}{s\in{}^{{<}\omega}\kappa}$ with the listed properties. 
  Pick an $\Add{\omega}{1}$-nice name $\dot{x}$ for a subset of $\kappa$ such that $\dot{x}^G=i^{I_{c_G}}_{0,\infty}(x)$ holds whenever $G$ is $\Add{\omega}{1}$-generic over $\VV$, $c_G=\bigcup G\in({}^\omega 2)^{\VV[G]}$ and $I_{c_G}$ is the unique linear iteration of $\langle\VV,\Set{U_{c_G\restriction n}}{n<\omega}\rangle$ of length $\sup_{n<\omega}\length{I_{c_G\restriction n}}{}$ in $\VV[G]$ with $U^{I_G}_\alpha=U^{I_{c_G\restriction n}}_\alpha$ for all $n<\omega$ and $\alpha<\length{I_{c_G\restriction n}}{}$. 
  The elementarity of $i^{I_{c_G}}_{0,\infty}$ and $\Sigma_1$-upwards absoluteness between $M^{I_{c_G}}_\infty$ and $\VV[G]$ then imply that 
  \begin{equation}\label{equation:ForcePhiAddOmega}
   \mathbbm{1}_{\Add{\omega}{1}}\Vdash\varphi(\dot{x},\check{\kappa})
  \end{equation}
  holds in $\VV$. 
  Finally, let $\map{z}{\omega}{2}$ denote the constant function with value $0$ and for each $n<\omega$, set $\kappa_n=\kappa_{z\restriction n}$ and $U_n=U_{z\restriction n}$. Then the sequence $\seq{\kappa_n}{n<\omega}$ is strictly increasing and cofinal in $\kappa$.

  Pick a sufficiently large regular cardinal $\theta$ and a countable elementary submodel $X$ of $\HH{\theta}$ containing $\kappa$, $\dot{x}$, $\seq{\kappa_s}{s\in{}^{{<}\omega}\kappa}$, $\seq{U_s}{s\in{}^{{<}\omega}\kappa}$ and $\seq{I_s}{s\in{}^{{<}\omega}\kappa}$.  
  Let $\map{\pi}{X}{M}$ denote the corresponding transitive collapse. 
  Define $\bar{\kappa}=\pi(\kappa)$ and, given $n<\omega$, set $\bar{\kappa}_n=\pi(\kappa_n)$ and $\bar{U}_n=\pi(U_n)$. 
   Then {\cite[Lemma 3.5]{MR3411035}} shows that the pair $\langle M,\Set{\bar{U}_n}{n<\omega}\rangle$ is linearly iterable. 
  Let $\bar{I}$ denote the unique linear iteration of $\langle M,\Set{\bar{U}_n}{n<\omega}\rangle$ of length $\kappa$ with the property that $$U^{\bar{I}}_\alpha ~ = ~ i^{\bar{I}}_{0,\alpha}(\bar{U}_{\min\{n<\omega\vert\alpha<\kappa_n\}})$$ holds for all $\alpha<\kappa$. 
  Set $N=M^{\bar{I}}_{0,\infty}$ and $\map{j=i^{\bar{I}}_{0,\infty}}{M}{N}$. 
  Then it is easy to see that $j(\bar{\kappa}_n)=\kappa_n$ for all $n<\omega$ and this implies that $j(\bar{\kappa})=\kappa$.

   Now, pick $c\in{}^\omega 2$ with the property that $G_c=\Set{c\restriction n}{n<\omega}$ is  $\Add{\omega}{1}$-generic  over $M$.  
  Then $G_c$ is also $\Add{\omega}{1}$-generic over $N$ and 
  %
  we define $$x_c ~ = ~  j(\pi(\dot{x}))^{G_c} ~ \in ~ \POT{\kappa}^{N[G_c]}.$$

 \begin{claim*}
  If $c\in{}^\omega 2$ has the property that $G_c$ is $\Add{\omega}{1}$-generic  over $M$, then $x_c\in D$. 
 \end{claim*}
   
 \begin{proof}[Proof of the Claim]
  By $\Sigma_1$-absoluteness, we know that   \eqref{equation:ForcePhiAddOmega} implies that the given forcing statement also holds in $\HH{\theta}$. This shows that  $$\mathbbm{1}_{\Add{\omega}{1}}\Vdash\varphi(j(\pi(\dot{x})),\check{\kappa})$$ holds in $N$. But this allows us to conclude that $\varphi(x_c,\kappa)$ holds in $N[G_c]$ and $\Sigma_1$-upwards absoluteness implies that this statement also holds in $\VV$.  
 \end{proof}

 Let $E$ denote the set of all pairs $\langle c,d\rangle$  in ${}^\omega 2\times{}^\omega 2$ with the property that $G_c\times G_d$ is $(\Add{\omega}{1}\times\Add{\omega}{1})$-generic over $M$. 
  Then $E$ is a comeager subset of ${}^\omega 2\times{}^\omega 2$ and a classical result of Mycielski (see {\cite[Theorem 19.1]{MR1321597}}) yields a continuous injection $\map{p}{{}^\omega 2}{{}^\omega 2}$ with $\langle p(a),p(b)\rangle\in E$ for all distinct $a,b\in{}^\omega 2$.

  \begin{claim*}
   The map $$\Map{\iota}{{}^\omega 2}{D}{a}{x_{p(a)}}$$ is an injection. 
  \end{claim*}
   
   \begin{proof}[Proof of the Claim]
    Given an $\Add{\omega}{1}$-name $\dot{y}$, let $\dot{y}_l$ and $\dot{y}_r$ denote the canonical $(\Add{\omega}{1}\times\Add{\omega}{1})$-names such that $\dot{y}_l^{G_0\times G_1}=\dot{y}^{G_0}$ and $\dot{y}_r^{G_0\times G_1}=\dot{y}^{G_1}$ holds whenever $G_0\times G_1$ is $(\Add{\omega}{1}\times\Add{\omega}{1})$-generic over $\VV$. 
    If $G_0\times G_1$ is $(\Add{\omega}{1}\times\Add{\omega}{1})$-generic over $\VV$ and $i<2$, then $(I_{c_{G_i}})^{\VV[G_i]}=(I_{c_{G_i}})^{\VV[G_0,G_1]}$ and this shows that  $$\dot{x}^{G_i} ~ = ~ (i^{I_{c_{G_i}}}_{0,\infty}(x))^{\VV[G_0,G_1]}$$ holds for the $\Add{\omega}{1}$-name $\dot{x}$ fixed at the beginning of the proof of Theorem \ref{theorem:AlmostDisjoint}. 
    Therefore,  we can apply Lemma \ref{lemma:LongIterationsGenericExtensions}  to see that $$\mathbbm{1}_{\Add{\omega}{1}\times\Add{\omega}{1}}\Vdash\anf{\dot{x}_l\neq\dot{x}_r}$$ holds in $\VV$ and, by $\Sigma_1$-absoluteness, this statement also holds in $\HH{\theta}$.

    Now, given $a,b\in{}^\omega 2$ with $a\neq b$, we have  
     \begin{equation*}
     \begin{split}
      \iota(a) ~  = & ~ x_{p(a)} ~ = ~ j(\pi(\dot{x}))^{G_{p(a)}} ~ = ~ j(\pi(\dot{x}_l))^{G_{p(a)}\times G_{p(b)}} \\
       & \neq ~ j(\pi(\dot{x}_r))^{G_{p(a)}\times G_{p(b)}} ~ = ~  j(\pi(\dot{x}))^{G_{p(b)}} ~ = ~ x_{p(b)} ~ = ~ \iota(b). \qedhere
      \end{split}
     \end{equation*}
   \end{proof}

    In the following, let $\blacktriangleleft$ denote the unique binary relation on ${}^\omega 2$ with $$a\blacktriangleleft b ~ \Longleftrightarrow ~ x_{p(a)}\lhd x_{p(b)}$$  for all $a,b\in{}^\omega 2$. 
   Then the above claim  implies that $\blacktriangleleft$ is a well-ordering of ${}^\omega 2$.

   \begin{claim*}
    The following statements are equivalent for all $a,b\in{}^\omega 2$: 
    \begin{enumerate}
     \item\label{item:aSMALLERb} $a\blacktriangleleft b$. 
     
     \item\label{item:EquivCountableModel} There exists a countable transitive model $W$ of $\ZFC^-$ and elements $\delta$, $\vec{\delta}$, $\vec{F}$ and $I$ of $W$ such that the following statements hold: 
      \begin{itemize}
       \item $W$ contains $M$, $p(a)$, $p(b)$ and a surjection from $\omega$ onto $M$.  
       
       \item $\vec{\delta}=\seq{\delta_n}{n<\omega}$ is a strictly increasing sequence of cardinals in $W$ with $\delta=\sup_{n<\omega}\delta_n$. 
       
       \item $\vec{F}=\seq{F_n}{n<\omega}$ is a sequence with the property that $F_n$ is a normal ultrafilter on $\delta_n$ in $W$ for all $n<\omega$. 
       
       \item If $\map{k}{\bar{W}}{W}$ is an elementary embedding of a transitive model $\bar{W}$ into $W$ and $\mathcal{E}\in\bar{W}$ satisfies $k(\mathcal{E})=\Set{F_n}{n<\omega}$, then the pair $\langle\bar{W},\mathcal{E}\rangle$ is \emph{$\alpha$-iterable} (see {\cite[p. 131]{MR3411035}}) for all $\alpha<\omega_1$. 
       
       \item $I$ is the unique linear iteration of $\langle M,\Set{\bar{U}_n}{n<\omega}\rangle$ of length $\delta$ with the property that $$U^I_\alpha ~ = ~ i^I_{0,\alpha}(\bar{U}_{\min\{n<\omega\vert\alpha<\delta_n\}})$$ holds for all $\alpha<\delta$. 
       
       \item The statement $$\psi(i^I_{0,\infty}(\pi(\dot{x}))^{G_{p(a)}},i^I_{0,\infty}(\pi(\dot{x}))^{G_{p(b)}},\delta)$$ holds in $W$. 
      \end{itemize}
    \end{enumerate}
   \end{claim*}
   
   \begin{proof}[Proof of the Claim]
    First, assume that \eqref{item:aSMALLERb} holds. 
    Pick a sufficiently large regular cardinal $\vartheta>\theta$ and a countable elementary submodel $Y$ of $\HH{\vartheta}$ containing $\theta$, $p(a)$, $p(b)$, $\seq{U_n}{n<\omega}$, $X$ and $\bar{I}$. 
    Let $\map{\tau}{Y}{W}$ denote the the corresponding transitive collapse. 
    Given $n<\omega$, set $\delta_n=\tau(\kappa_n)$ and $F_n=\tau(U_n)$. Moreover, define $\delta=\tau(\kappa)$ and $I=\tau(\bar{I})$. 
    In this situation, {\cite[Lemma 3.5]{MR3411035}} shows that the pair $\langle W,\Set{F_n}{n<\omega}\rangle$ is linearly iterable. 
    Another application of  {\cite[Lemma 3.5]{MR3411035}} allows us to also conclude that $\langle\bar{W},\mathcal{E}\rangle$ is $\alpha$-iterable, whenever $\alpha$ is a countable ordinal, $\bar{W}$ is a transitive set, $\map{k}{\bar{W}}{W}$ is an elementary embedding and $\mathcal{E}\in\bar{W}$ with $k(\mathcal{E})=\Set{F_n}{n<\omega}$. 
    Next, since we have $\delta=\sup_{n<\omega}\delta_n$ and $\tau\restriction(M\cup\{M\})=\id_{M\cup\{M\}}$, elementarity directly implies that $I$ is the unique linear iteration of $\langle M,\Set{\bar{U}_n}{n<\omega}\rangle$ of length $\delta$ with the property that $$U^I_\alpha ~ = ~ i^I_{0,\alpha}(\bar{U}_{\min\{n<\omega\vert\alpha<\delta_n\}})$$ holds for all $\alpha<\delta$. 
     Finally, since \eqref{item:aSMALLERb} implies that $$\psi(i^{\bar{I}}_{0,\infty}(\pi(\dot{x}))^{G_{p(a)}},i^{\bar{I}}_{0,\infty}(\pi(\dot{x}))^{G_{p(b)}},\kappa)$$ holds in $\HH{\vartheta}$, 
     elementarity directly implies that $$\psi(i^I_{0,\infty}(\pi(\dot{x}))^{G_{p(a)}},i^I_{0,\infty}(\pi(\dot{x}))^{G_{p(b)}},\delta)$$ holds in $W$. 
    In combination, these observations show that $W$, $\delta$, $\seq{\delta_n}{n<\omega}$, $\seq{F_n}{n<\omega}$ and $I$ witness that \eqref{item:EquivCountableModel} holds.

     Now, assume that $W$, $\delta$, $\seq{\delta_n}{n<\omega}$, $\seq{F_n}{n<\omega}$ and $I$ witness that \eqref{item:EquivCountableModel} holds. 
   By {\cite[Lemma 3.6]{MR3411035}}, our assumptions ensure that the pair $\langle W,\Set{F_n}{n<\omega}\rangle$ is linearly iterable. 
     Let $I_*$ denote the unique linear iteration of $\langle W,\Set{F_n}{n<\omega}\rangle$ of length $\kappa$ with the property that $$U^{I_*}_\alpha ~ = ~ i^{I_*}_{0,\alpha}(F_{\min\{n<\omega\vert\alpha<\kappa_n\}})$$ holds for all $\alpha<\kappa$. 
     Then we have $i^{I_*}_{0,\infty}(\delta_n)=\kappa_n$ for all $n<\omega$ and  $i^{I_*}_{0,\infty}(\delta)=\kappa$. 
     Moreover, we know that $$i^{I_*}_{0,\infty}\restriction(M[G_{p(a)},G_{p(b)}]\cup\{M\}) ~ = ~ \id_{M[G_{p(a)},G_{p(b)}]\cup\{M\}}.$$
     This shows that $i^{I_*}_{0,\infty}(I)$ is a linear iteration of $\langle M,\Set{\bar{U}_n}{n<\omega}\rangle$ of length $\kappa$ with the property that $$U^{i^{I_*}_{0,\infty}(I)}_\alpha ~ = ~ i^{i^{I_*}_{0,\infty}(I)}_{0,\alpha}(\bar{U}_{\min\{n<\omega\vert\alpha<\kappa_n\}})$$ holds for all $\alpha<\kappa$, and this implies that $i^{I_*}_{0,\infty}(I)=\bar{I}$ holds. 
     In particular, it follows that  $$i^{I_*}_{0,\infty}(i^I_{0,\infty}(y)) ~ = ~ i^{\bar{I}}_{0,\infty}(y)$$ holds for all $y\in M$. 
     By our assumptions and the above observations, this shows that $$\psi(i^{\bar{I}}_{0,\infty}(\pi(\dot{x}))^{G_{p(a)}},i^{\bar{I}}_{0,\infty}(\pi(\dot{x}))^{G_{p(b)}},\kappa)$$ holds in $M^{I_*}_{0,\infty}$. 
     Using $\Sigma_1$-upwards absoluteness, we know that $\psi(x_{p(a)},x_{p(b)},\kappa)$ holds in $\VV$ and this shows that \eqref{item:aSMALLERb} holds in this case. 
   \end{proof}

    Since the above claim shows that the relation $\blacktriangleleft$ is definable over $\HH{\aleph_1}$ by a $\Sigma_2$-formula with parameters, we can conclude that $\blacktriangleleft$ is a $\mathbf{\Sigma}^1_3$-subset of ${}^\omega 2\times{}^\omega 2$ (see {\cite[Lemma 25.25]{MR1940513}}). 
  This completes the proof of the theorem.  
\end{proof}

We end this section by proving the equiconsistency stated in Theorem \ref{theorem:EquiLongWO}.  %
One direction is given by the following lemma that follows from arguments presented in the proof of Theorem \ref{theorem:StrengthCountableCof}.\footnote{The construction of simply definable long well-orderings in the power sets  of uncountable cardinals was the original motivation for the work presented in \cite{ClosureUlt}. 
In combination with ideas contained in the proof of Lemma \ref{lemma:Sigma1Injection}, the results of \cite{ClosureUlt} can be used to show that, if $0^\dagger$  does not exist and the cardinal $\kappa$ is either  singular or  weakly compact, then there exists a well-ordering of a subset of $\POT{\kappa}$ of order-type ${\kappa^+}\cdot\kappa$ that is definable by a $\Sigma_1$-formula with parameter $\kappa$.}

\begin{lemma}\label{lemma:Sigma1Injection}
 Assume that there is no inner model with infinitely many measurable cardinals. If $\kappa$ is a singular cardinal, then there exists an injection from $\kappa^+$ into $\POT{\kappa}$ that is definable by a $\Sigma_1$-formula with parameters in $\HH{\kappa}\cup\{\kappa\}$. 
\end{lemma}

\begin{proof}
 As in the proof of Theorem \ref{theorem:StrengthCountableCof} in Section \ref{section:OptimalCountable}, we know that $0^{\text{long}}$ does not exist and we let $U_{can}$ denote the \emph{canonical sequence of measures} as in \cite{MR926749}. We again set $U=U_{can}\restriction\kappa$ and $\KK=\KK[U]$. 
 Then our assumptions imply that $U\in\HH{\kappa}^\KK$ and the results of \cite{MR926749} show that $\KK$ is an inner model of $\ZFC$ with a canonical well-ordering $<_\KK$. 
 Since the domain of $U_{can}$ is finite, we can again combine {\cite[Theorem 3.9]{MR926749}},  {\cite[Theorem 3.19]{MR926749}} and {\cite[Theorem 3.23]{MR926749}} to show that $\kappa^+=(\kappa^+)^\KK$. 

 Given $\kappa\leq\gamma<\kappa^+$, we let $y_\gamma$ denote the subset of $\kappa$ that canonically codes the $<_\KK$-least bijection between $\kappa$ and $\gamma$. As in the proof of Theorem \ref{theorem:StrengthCountableCof}, we can now conclude that the unique  injection $\map{\iota}{\kappa^+}{\POT{\kappa}}$ with $\iota\restriction\kappa=\id_\kappa$ and $\iota(\gamma)=y_\gamma$ for all $\kappa\leq\gamma<\kappa^+$ can be defined by a $\Sigma_1$-formula and the parameters $\kappa$ and $U$.   
\end{proof}

The next lemma is needed in the converse direction of our  equiconsistency proof:

\begin{lemma}\label{lemma:MoreMeasurables}
 Let $U$ be a normal ultrafilter on a measurable cardinal $\delta$, let $\alpha<\delta$, let $\calE$ be a set of normal ultrafilters on cardinals smaller than $\alpha$ and let $I$ be a linear iteration of $\langle\VV,\calE\rangle$ of length less than $\alpha$. 
 If $B\in i^I_{0,\infty}(U)$, then there is $A\in U$ with 
  $i^I_{0,\infty}(A)\subseteq B$. 
\end{lemma}

\begin{proof}
 %
 Using  {\cite[Exercise 12]{MR1994835}}, we find a function $\map{f}{[\alpha]^{{<}\omega}}{U}$ with the property that $B\in\ran{i^I_{0,\infty}(f)}$. If we now define $$A ~ = ~  \bigcap\Set{f(a)}{a\in[\alpha]^{{<}\omega}},$$ then $A$ is an element of $U$ with the desired properties. 
\end{proof}

In order to complete the proof of Theorem \ref{theorem:EquiLongWO}, we will now use \emph{diagonal Prikry forcing}  and a characterisation of generic sequences for this forcing due to Fuchs  \cite{MR2193185} to construct a model without $\Sigma_1$-definable long well-orderings from an infinite sequence of measurable cardinals.

\begin{proof}[Proof of Theorem \ref{theorem:EquiLongWO}]
 Assume that $\vec{\kappa}=\seq{\kappa(n)}{n<\omega}$ is a strictly increasing sequence of measurable cardinals with limit $\kappa$. Pick a sequence $\vec{U}=\seq{U(n)}{n<\omega}$ with the property that $U(n)$ is a normal ultrafilter on $\kappa(n)$ for each $n<\omega$. 
 Let $\PPP_{\vec{U}}$ denote the \emph{diagonal Prikry forcing} associated to the sequence $\vec{U}$ (see {\cite[Section 1.3]{MR2768695}}), {i.e.} $\PPP_{\vec{U}}$ is the partial order defined by the following clauses: 
 \begin{itemize}
     \item Conditions in $\PPP_{\vec{U}}$ are sequences $p=\seq{p_n}{n<\omega}$ with the property that there exists a natural number $l_p$ such that $p_n\in\kappa(n)$ for all $n<l_p$ and $p_n\in U(n)$ for all $l_p\leq n<\omega$. 
     
     \item Given conditions $p$ and $q$ in $\PPP_{\vec{U}}$, we have $p\leq_{\PPP_{\vec{U}}}q$ if and only if $l_q\leq l_p$, $p_n=q_n$ for all $n<l_q$, $q_n\in p_n$ for all $l_q\leq n<l_p$ and $p_n\subseteq q_n$ for all $l_p\leq n<\omega$. 
 \end{itemize}
 By {\cite[Lemma 1.35]{MR2768695}}, forcing  with $\PPP_{\vec{U}}$ does not add bounded subsets of $\kappa$.

 Given a filter $G$ on $\PPP_{\vec{U}}$, we let  $c_G$ denote the unique function with domain $\sup_{p\in G}l_p\leq\omega$ and $c_G(n)=p_n$ for all $p\in G$ and $n<l_p$. 
 In the other direction, given a  sequence $c$ contained in the set $\prod\vec{\kappa}$ of all functions $d\in {}^\omega\kappa$ with $d(n)<\kappa(n)$ for all $n<\omega$, we let $G_c$ denote the set of all conditions $p$ in $\PPP_{\vec{U}}$ with $p_n=c(n)$ for all $n<l_p$ and $c(n)\in p_n$ for all $l_p\leq n<\omega$. %
 It is easy to see that $G_c$ is a filter on $\PPP_{\vec{U}}$ with $c_{G_c}=c$ in this situation. 
 Given an inner model $M$ that contains $\vec{U}$ and $c\in\prod\vec{\kappa}$, we say that $c$ is    \emph{$\vec{U}$-generic} over $M$ if $G_c$ is $\PPP_{\vec{U}}$-generic over $M$. 
  The results of \cite{MR2193185} then show that a  sequence $c\in\prod\vec{\kappa}$ is $\vec{U}$-generic over an inner model $M$ if and only if $\Set{n<\omega}{c(n)\in A_n}$ is a cofinite subset of $\omega$ for every sequence $\seq{A_n\in U(n)}{n<\omega}$ in $M$. 
  Using {\cite[Theorem 3.5.1]{MR0373889}}, this characterization can be used to show that the Boolean completion of $\PPP_{\vec{U}}$ is weakly homogeneous and therefore every statement in the forcing language of $\PPP_{\vec{U}}$ that uses only ground model elements as parameters is decided by $\mathbbm{1}_{\PPP_{\vec{U}}}$. 
  
  Now, let $G$ be $\PPP_{\vec{U}}$-generic over $\VV$ and assume that, in $\VV[G]$, there exists a well-ordering $\lhd$ of a subset $D$ of $\POT{\kappa}$ of cardinality greater than $\kappa$ that can be defined by a $\Sigma_1$-formula $\varphi(v_0,\ldots,v_3)$, a parameter $z\in\HH{\kappa}$ and the parameter $\kappa$. 
  Then we can find a $\Sigma_1$-formula $\psi(v_0,v_1,v_2)$ with the property that, in $\VV[G]$, the set $D$ can be defined by the formula $\psi$ and the parameters $\kappa$ and $z$. 
  In this situation, we know that   $z\in\VV$ and the homogeneity properties of $\PPP_{\vec{U}}$ imply that $D\subseteq \VV$, because, given $y\in D$, we know that $y$ is the $\alpha$-th element of the well-order $\langle D,\lhd\rangle$ for some ordinal $\alpha$ and therefore the set $\{y\}$ is definable in $\VV[G]$ by a formula with parameters in $\VV$. In addition, we know that 
  \begin{equation}\label{equation:DomainDefGroundModel}
   D ~ = ~ \Set{y\in\POT{\kappa}^\VV}{\mathbbm{1}_{\PPP_{\vec{U}}}\Vdash\psi(\check{y},\check{z},\check{\kappa})}.
  \end{equation}
  
  Let $\calE$ denote the set of all normal ultrafilters on cardinals smaller than $\kappa$ in $\VV$. 
  Apply Lemma \ref{lemma:TechnicalLemmaIterationsTree} to $\kappa$, $z$ and $D$ in $\VV$ to obtain 
  an element $x_*$ of $D$, 
   a system $\seq{\nu_s}{s\in{}^{{<}\omega}\kappa}$ of inaccessible cardinals smaller than $\kappa$, 
   a system $\seq{\kappa_s}{s\in{}^{{<}\omega}\kappa}$ of measurable cardinals smaller than $\kappa$, 
   a system $\seq{U_s}{s\in{}^{{<}\omega}\kappa}$ of elements of $\calE$,  and 
   a system $\seq{I_s}{s\in{}^{{<}\omega}\kappa}$ of linear iterations of $\langle\VV,\calE\rangle$ possessing the properties listed in the lemma. 
   Next, for each $c\in({}^\omega\kappa)^{\VV[G]}$, let $I_c$ denote the unique iteration of $\langle\VV,\Set{U_{c\restriction n}}{n<\omega}\rangle$ of length $\sup_{n<\omega}\length{I_{c\restriction n}}{}$ in $\VV[G]$ with $U^{I_c}_{\alpha}=U^{c\restriction n}_\alpha$ for all $n<\omega$ and $\alpha<\length{I_{c\restriction n}}{}$. 
   Then $M^{I_c}_{\length{I_{c\restriction n}}{}}=M^{I_{c\restriction n}}_\infty$ and $i^{I_c}_{0,\length{I_{c\restriction n}}{}}=i^{I_{c\restriction n}}_{0,\infty}$ for all $c\in({}^\omega\kappa)^{\VV[G]}$ and $n<\omega$ with $\length{I_{c\restriction n}}{}<\length{I_c}{}$.
   Moreover, we have $M^{I_c}_\infty=M^{I_{c\restriction n}}_\infty$ and $i^{I_c}_{0,\infty}=i^{I_{c\restriction n}}_{0,\infty}$ for all $c\in({}^\omega\kappa)^{\VV[G]}$ and $n<\omega$ with $\length{I_{c\restriction n}}{}=\length{I_c}{}$.
   Given $c\in({}^\omega\kappa)^{\VV[G]}$, we define $M_c=M^{I_c}_\infty$, $\bar{c}=i^{I_c}_{0,\infty}\circ c_G$ and $x_c=i^{I_c}_{0,\infty}(x_*)$. 
   In this situation, Lemma \ref{lemma:LongIterationsGenericExtensions} shows that $M_c$ is well-founded for all $c\in({}^\omega\kappa)^{\VV[G]}$.

   \begin{claim*}
    If $c\in({}^\omega\kappa)^{\VV[G]}$, then $\bar{c}$ is $i^{I_c}_{0,\infty}(\vec{U})$-generic over $M_c$. 
   \end{claim*}
   
   \begin{proof}[Proof of the Claim]
    Suppose that  $i^{I_c}_{0,\infty}(\vec{U})=\seq{U^{\prime\prime}(n)}{n<\omega}$ and 
    fix a sequence $\vec{C}=\seq{C_n\in U^{\prime\prime}(n)}{n<\omega}$ in $M_c$. 
    Since $I_c$ is a linear iteration of length at most $\kappa$, we can find $n_0<\omega$ and a sequence $\vec{B}=\seq{B_n}{n<\omega}$ in $M^{I_{c\restriction n_0}}_\infty$ such that either $\length{I_c}{}=\length{I_{c\restriction n_0}}{}$ and $\vec{B}=\vec{C}$, or $\length{I_c}{}>\length{I_{c\restriction n_0}}{}$ and $i^{I_c}_{\length{I_{c\restriction n_0}}{},\infty}(\vec{B})=\vec{C}$. 
    Now, pick $n_1<\omega$ with $\kappa(n)>\kappa_{c\restriction n_0}$ for all $n_1\leq n<\omega$. 
    In this situation, the conclusions of Lemma \ref{lemma:TechnicalLemmaIterationsTree} ensure that we can apply Lemma \ref{lemma:MoreMeasurables} to find a sequence $\seq{A_n\in U(n)}{n<\omega}$ with  $i^{I_c}_{0,\length{I_{c\restriction n_0}}{}}(A_n)\subseteq B_n$ for all $n_1\leq n<\omega$. 
    Since $c_G$ is $\vec{U}$-generic over $\VV$, we find $n_1\leq n_2<\omega$ with $c_G(n)\in A_n$ for all $n_2\leq n<\omega$. 
    But this shows that $\bar{c}(n)\in C_n$ holds for all $n_2\leq n<\omega$. 
    Using the characterization of generic sequences provided by \cite{MR2193185}, these computations prove the statement of the claim. 
   \end{proof}

   \begin{claim*}
    If $c\in({}^\omega\kappa)^{\VV[G]}$, then $x_c\in D$. 
   \end{claim*}
    
   \begin{proof}[Proof of the Claim]
    By the previous claim, there exists a filter $H$ on $i^{I_c}_{0,\infty}(\PPP_{\vec{U}})$ in $\VV[G]$ that is generic over $M_c$. 
    Since Lemma \ref{lemma:LongIterationsGenericExtensions} shows that $i^{I_c}_{0,\infty}(\kappa)=\kappa$ and $i^{I_c}_{0,\infty}(z)=z$, we can use \eqref{equation:DomainDefGroundModel} to show that $$\mathbbm{1}_{i^{I_c}_{0,\infty}}\Vdash\psi(\check{x}_c,\check{z},\check{\kappa})$$ holds in $M_c$. 
    This shows that $\psi(x_c,z,\kappa)$ holds in $M_c[H]$ and $\Sigma_1$-upwards absoluteness implies that this statement also holds in $\VV[G]$.  
   \end{proof}

   By Lemma \ref{lemma:LongIterationsGenericExtensions}, our definitions ensure that the map $$\Map{\iota}{({}^\omega\kappa)^{\VV[G]}}{D}{c}{x_c}$$ is an injection that is definable in $\VV[G]$ from parameters contained in the ground model $\VV$. 
   Since $\iota(c_G)\in D\subseteq\VV$, this shows that, in $\VV[G]$, the set $\{c_G\}$ is definable from parameters in $\VV$. 
   Using the homogeneity properties of $\PPP_{\vec{U}}$ in $\VV$, we can now conclude that $c_G$ is an element of $\VV$, a contradiction.  
\end{proof}


\section{Long well-orderings in $\POT{\omega_1}$}

We now show that both strong large cardinal assumptions and  strong forcing axioms cause analogues of the above results on the definability of long well-orders to hold for  $\omega_1$. 
In the following, we combine well-known consequences of the \emph{Axiom of Determinacy} $\AD$ with Woodin's analysis of $\PPP_{max}$-extensions of determinacy models (see \cite{MR2768703} and \cite{MR1713438}).

\begin{lemma}[\ZF]\label{lemma:NormalFilterLONG}
 Let $\kappa$ be an infinite cardinal. If there is an injection from $\kappa^+$ into $\POT{\kappa}$, then every ${<}\kappa^+$-complete ultrafilter on $\kappa^+$ is principal. 
\end{lemma}

\begin{proof}
 Let $U$ be a ${<}\kappa^+$-complete ultrafilter on $\kappa^+$ and let $\map{\iota}{\kappa^+}{\POT{\kappa}}$ be an injection. 
 Given $\alpha<\kappa$, set $$B_\alpha ~ = ~ \Set{\gamma<\kappa^+}{\alpha\in\iota(\gamma)}.$$ Since $U$ is an ultrafilter, there is  $A\subseteq\kappa$ such that $$\alpha\in A ~ \Longleftrightarrow ~ B_\alpha\in U$$ holds for all $\alpha<\kappa$. 
 The ${<}\kappa^+$-completeness of $U$ then ensures that the set $$B ~ = ~ \bigcap\Set{B_\alpha}{\alpha\in A} ~ \cap ~ \bigcap\Set{\kappa^+\setminus B_\alpha}{\alpha\in\kappa\setminus A}$$ is an element of $U$. We now know that $\iota[B]=\{A\}$ and hence the injectivity of $\iota$ implies that $B$ is a singleton.  
\end{proof}

\begin{corollary}[\ZF+\DC+\AD]\label{corollary:ADnoLONGWO}
 There is no injection from $\omega_2$ into $\POT{\omega_1}$. 
\end{corollary}

\begin{proof}
 By results of Kleinberg and Martin--Paris (see {\cite[Section 13]{MR526915}}), the restriction of the closed unbounded filter on $\omega_2$ to the set of all ordinals of countable cofinality is a ${<}\omega_2$-complete, non-principal ultrafilter on $\omega_2$. 
\end{proof}

The following lemma will allow us to use the theory developed in \cite{MR1713438} to prove Theorem \ref{theorem:ResultsOmega1}.\eqref{item:Omega1-1}.

\begin{lemma}\label{lemma:PmaxLongWO}
 Assume that  $\AD$  holds in $\LL(\RRR)$ and $\VV$ is a $\PPP_{max}$-generic extension of $\LL(\RRR)$. 
 Then no well-ordering of a subset of $\POT{\omega_1}$ of cardinality greater than $\aleph_1$ is contained in $\OD(\RRR)$. 
\end{lemma}

\begin{proof}
 Assume, towards a contradiction, that there exists a subset $D$ of $\POT{\omega_1}$ of cardinality greater than $\aleph_1$ and a well-orderding $\lhd$ of $D$ that is contained in $\OD(\RRR)$. 
 Then the fact that $\PPP_{max}$ is countably closed and homogeneous in $\LL(\RRR)$ (see {\cite[Lemma 4.38]{MR1713438}} and {\cite[Lemma 4.43]{MR1713438}}) directly implies that $D$ and $\lhd$ are both contained in $\LL(\RRR)$. 
 But this shows that  $\LL(\RRR)$ is a model of $\ZF+\DC+\AD$ that contains an injection from  $\omega_2$ into $\POT{\omega_1}$, contradicting Corollary \ref{corollary:ADnoLONGWO}. 
\end{proof}

\begin{proof}[Proof of Theorem \ref{theorem:ResultsOmega1}.\eqref{item:Omega1-1}]
 Let $\lhd$ be a well-ordering of a subset of $\POT{\omega_1}$ of cardinality greater than $\aleph_1$ that is definable by a $\Sigma_1$-formula $\varphi(v_0,\ldots,v_3)$ and parameters $\omega_1$ and $z\in\HH{\aleph_1}$. 
 
 First, assume that Woodin's Axiom $(*)$ holds, {i.e.} $\AD$ holds in $\LL(\RRR)$ and $\LL(\POT{\omega_1})$ is a $\PPP_{max}$-generic extension of $\LL(\RRR)$. 
 We now know that $\lhd$ and its domain are both elements of $\OD(\RRR)^{\LL(\POT{\omega_1})}$, because  $\Sigma_1$-statements with parameters in $\HH{\aleph_2}$ are absolute between $\LL(\POT{\omega_1})$ and $\VV$. 
 Since the domain of $\lhd$ has cardinality greater than $\aleph_1$ in $\LL(\POT{\omega_1})$, we can now use Lemma \ref{lemma:PmaxLongWO} to derive a contradiction.

 Now, assume that there is a measurable cardinal above  infinitely many Woodin cardinals. 
 Then $\AD$ holds in $\LL(\RRR)$. %
 Note that the formula $\varphi$ and the parameters $\omega_1$ and $z$ also define $\lhd$ in $\HH{\aleph_2}$, and this  statement can be formulated by a $\Pi_2$-formula with parameter $z$ in the structure  $\langle\HH{\aleph_2},\in\rangle$
  Let $G$ be $\PPP_{max}$-generic over $\LL(\RRR)$. Then the  \emph{$\Pi_2$-maximality} of $\LL(\RRR)[G]$ (see {\cite[Theorem 7.3]{MR2768703}}) implies that the formula $\varphi$ and the parameters $\omega_1$ and $z$ also  define a well-ordering of a subset of $\POT{\omega_1}$ of cardinality greater than $\aleph_1$ in the structure $\langle\HH{\aleph_2}^{\LL(\RRR)[G]},\in\rangle$. 
  In particular, such a well-ordering is contained in $\OD(\RRR)^{\LL(\RRR)[G]}$, again contradicting Lemma \ref{lemma:PmaxLongWO}.  
\end{proof}



\section[Almost disjoint families]{Almost disjoint families in $\POT{\omega_1}$}\label{section:AlmostDisjointOmega1}

Following the structure of the arguments in the previous section, we now show that both large cardinals and forcing axioms imply that large almost disjoint families of subsets of $\omega_1$ are not simply definable. 
The first step in these proofs is the following unpublished result of William Chan, Stephen Jackson and Nam Trang whose proof we include with their permission. This result is an application of their work on the validity of the \emph{Kurepa Hypothesis} in determinacy models and continues a line of groundbreaking results on definable combinatorics at  $\omega_1$ (see, for example, \cite{CJ19}, \cite{CJ21}, \cite{AlmostDisjointAD} and \cite{CJT2}).

\begin{theorem}[Chan--Jackson--Trang, ~ $\ZF+\DC_\RRR+\AD^+$]\label{theorem:AD-MAD-omega_1}
 Assume that $\VV=\LL(\POT{\RRR})$ holds. If $A$ is a set of cofinal subsets of $\omega_1$, then one of the following statements holds: 
 \begin{enumerate}
     \item $A$ can be well-ordered and its cardinality is less than or equal to $\aleph_1$. 
     
     \item There are distinct $x,y\in A$ such  that $x\cap y$ is unbounded in $\omega_1$. 
 \end{enumerate}
\end{theorem}

The proof of this result makes use of the following topological fact:

\begin{proposition}[$\ZF+\DC$]\label{proposition:LongSequenceDisjointBaire}
 If $X$ is a Polish space and $\seq{A_\alpha}{\alpha<\omega_1}$ is a sequence of pairwise disjoint non-meager subsets of $X$, then there is an $\alpha<\omega_1$ such that the subset $A_\alpha$ does not have the property of Baire. 
\end{proposition}

\begin{proof}
 Assume, towards a contradiction, that $A_\alpha$ has the property of Baire for all $\alpha<\omega_1$.   
 Given $\alpha<\omega_1$, our assumption implies that there is a non-empty open set $U$ with the property that $U\setminus A_\alpha$ is meager. Hence, there is a sequence $\seq{N_\alpha}{\alpha<\omega_1}$ of non-empty basic open subsets of $X$ such that $N_\alpha\setminus A_\alpha$ is meager. Pick $\alpha<\beta<\omega_1$ with $N_\alpha=N_\beta$. Then $N_\alpha\setminus(A_\alpha\cap A_\beta)=(N_\alpha\setminus A_\alpha)\cup(N_\alpha\setminus A_\beta)$ is meager and hence $A_\alpha\cap A_\beta\neq\emptyset$, a contradiction
\end{proof}

\begin{proof}[Proof of Theorem \ref{theorem:AD-MAD-omega_1}]
 Assume, towards a contradiction, that both conclusions fail. 
 
 \begin{claim*}
  The set $A$ cannot be  well-ordered. 
 \end{claim*}
 
 \begin{proof}[Proof of the Claim]
  Assume, towards a contradiction, that  $A$ can be  well-ordered. Then our assumptions imply that it has cardinality greater than $\aleph_1$ and hence we obtain an injection of $\omega_2$ into $\POT{\omega_1}$. 
  But this yields a contradiction, because the assumption of Corollary \ref{corollary:ADnoLONGWO} are satisfied in our setting. 
 \end{proof}

 By combining the above claim with {\cite[Theorem 1.4]{MR2777751}}, we now obtain an injection $\map{\iota}{\RRR}{A}$. 
 Our assumptions then ensure that the function  $$\Map{c}{[\RRR]^2}{\omega_1}{\{x,y\}}{\min\Set{\alpha<\omega_1}{\iota(x)\cap\iota(y)\subseteq\alpha}}$$ is well-defined. 
 Given $\alpha<\omega_1$, set $E_\alpha=c^{{-}1}\{\alpha\}\subseteq\RRR\times\RRR$. Then $\bigcup\Set{E_\alpha}{\alpha<\omega_1}$ is dense open in $\RRR\times\RRR$. 
 
 \begin{claim*}
  There is a $\lambda<\omega_1$ such that the set $\bigcup\Set{E_\alpha}{\alpha<\lambda}$ is comeager in $\RRR\times\RRR$. 
 \end{claim*}
 
 \begin{proof}[Proof of the Claim]
  Assume that there is no $\lambda<\omega_1$ with the property that the set $\bigcup\Set{E_\alpha}{\alpha<\lambda}$ is comeager. 
  Since the ideal of meager subsets of $\RRR\times\RRR$ is closed under well-ordered unions in our setting,  our assumption yields a strictly increasing function $\map{f}{\omega_1}{\omega_1}$ with the property that $E_{f(\alpha)}$ is a non-meager subset of $\RRR\times\RRR$ for all $\alpha<\omega_1$. In this situation, the sequence $\seq{E_{f(\alpha)}}{\alpha<\omega_1}$ consists of pairwise disjoint non-meager subsets of $\RRR\times\RRR$ and, since we assume that $\AD$ holds,  all of these sets possess the property of Baire. This contradicts Proposition \ref{proposition:LongSequenceDisjointBaire}. 
 \end{proof}
 
 By a classical result of Mycielski (see {\cite[Theorem 19.1]{MR1321597}}), we can now find an injection $\map{e}{\RRR}{\RRR}$ such that for all $x,y\in\RRR$ with $x\neq y$, there is an $\alpha<\lambda$ with $\langle e(x),e(y)\rangle\in E_\alpha$. In this situation, we know that $$(\iota\circ e)(x) ~ \cap ~ (\iota\circ e)(y) ~ \subseteq ~ \lambda$$ holds for all $x,y\in\RRR$ with $x\neq y$. 
 In particular, since the set $A$ consists of unbounded subsets of $\omega_1$, we know that the map $$\Map{i}{\RRR}{\omega_1}{x}{\min((\iota\circ e)(x)\setminus\lambda)}$$ is an injection. But this shows that  the reals can be well-ordered, contradicting our assumptions. 
\end{proof}

In order to transfer the above result to models of the form $\HOD(\RRR)$ of $\PPP_{max}$-extensions,  we make use of another axiom introduced by Woodin, called   $\left({*\atop*}\right)$ (see {\cite[Definition 5.69]{MR1713438}}).

\begin{lemma}\label{lemma:MADinPmax}
 Assume that $\AD$ holds in $\LL(\RRR)$ and $\VV$ is a $\PPP_{max}$-generic extension of $\LL(\RRR)$. 
 If $A\in\OD(\RRR)$ is a set of cardinality greater than $\aleph_1$ that consists of unbounded subsets of $\omega_1$,  then there are distinct $x,y\in A$ with the property that $x\cap y$ is unbounded in $\omega_1$.  
\end{lemma}

\begin{proof} 
 Assume, towards a contradiction, that the above conclusion fails. 
 
 \begin{claim*}
  $A\subseteq\LL(\RRR)$. 
 \end{claim*}
 
 \begin{proof}[Proof of the Claim]
  Assume  that $A\nsubseteq\LL(\RRR)$. 
  Since $\VV=\LL(\POT{\omega_1})$ holds and   {\cite[Corollary 5.83]{MR1713438}} shows that our assumptions imply that $\left({*\atop*}\right)$ holds, we can apply  {\cite[Theorem 5.84]{MR1713438}} to find an unbounded subset $U$ of $\omega_1$ and a function $\map{\pi}{{}^{{<}\omega_1}2}{[\omega_1]^\omega}$ such that the following statements hold: 
  \begin{enumerate}
   \item If $s,t\in{}^{{<}\omega_1}2$ with $s\subseteq t$, then $\pi(s)\subseteq\pi(t)$ and $\pi(s)\cap\alpha=\pi(t)\cap\alpha$ for all $\alpha\in\pi(s)$. 
   
   \item Given $s\in{}^{{<}\omega_1}2$ and $\alpha\in\dom{s}\cap U$, we have $\alpha\in\pi(s)$ if and only if $s(\alpha)=1$. 
   
   \item If $x\in{}^{\omega_1}2$, then $\bar{\pi}(x)=\bigcup\Set{\pi(x\restriction\alpha)}{\alpha<\omega_1}\in A$. 
  \end{enumerate}
  
  Pick $x,y\in{}^{\omega_1}2$ such that $x$ has constant value $1$ and $y$ is the characteristic function of $U\setminus\{\min(U)\}$. 
 Since $\bar{\pi}(x),\bar{\pi}(y)\in A$ and $U\setminus\{\min(U)\}\subseteq\bar{\pi}(x)\cap\bar{\pi}(y)$, we know that $\bar{\pi}(x)=\bar{\pi}(y)$ as $\bar{\pi}(x)\cap\bar{\pi}(y)$ is unbounded in $\omega_1$. But $\min(U)\in\bar{\pi}(x)\setminus\bar{\pi}(y)$, a contradiction. 
 \end{proof}

 \begin{claim*}
  $A\in\LL(\RRR)$.  
 \end{claim*}
 
 \begin{proof}[Proof of the Claim]
  Using the homogeneity of $\PPP_{max}$ in $\LL(\RRR)$, this statement follows directly from the previous claim and the fact that the set $A$ is contained in the class  $\OD(\RRR)$.  
 \end{proof}

 Since {\cite[Corollary 2.16]{MR2777751}} shows that $\LL(\RRR)$ is a model of  $\DC+\AD^+$,  we can use Theorem \ref{theorem:AD-MAD-omega_1} in $\LL(\RRR)$ to conclude  that, in  $\LL(\RRR)$, the set $A$ can be well-ordered and its cardinality is smaller than or equal to $\aleph_1$. But this shows that the cardinality of $A$ in $\VV$ is at most $\aleph_1$, a contradiction. 
\end{proof}

\begin{proof}[Proof of Theorem \ref{theorem:ResultsOmega1}.\eqref{item:Omega1-2}]
 Let $A$ be a set of cardinality greater than $\aleph_1$ that consists of unbounded subsets of $\omega_1$ and is definable by a $\Sigma_1$-formula $\varphi(v_0,v_1,v_2)$ and parameters $\omega_1$ and  $z\in\HH{\aleph_1}$.

  First, assume that Woodin's Axiom $(*)$ holds. 
 %
 %
 Then the $\Sigma_1$-Reflection Principle implies that the formula $\varphi$ and the parameters $\omega_1$ and $z$ also define the set $A$ in $\LL(\POT{\omega_1})$. 
 But this shows that  $A\in\OD(\RRR)^{\LL(\POT{\omega_1})}$ and, since $\AD$ holds in $\LL(\RRR)$ and $\LL(\POT{\omega_1})$ is a $\PPP_{max}$-generic extension of $\LL(\RRR)$,  we can now apply Lemma \ref{lemma:MADinPmax} in $\LL(\POT{\omega_1})$ to find distinct $x,y\in A$ with $x\cap y$ unbounded in $\omega_1$.

 Next, assume that there is a measurable cardinal above infinitely many Woodin cardinals. 
 Then $\AD$ holds in $\LL(\RRR)$.
 %
 Note that the formula $\varphi$ and the parameters $\omega_1$ and $z$ also define $A$ in  the structure $\langle\HH{\omega_2},\in\rangle$. 
 Assume, towards a contradiction, that $x\cap y$ is bounded in $\omega_1$ for all distinct $x,y\in A$. 
 Note that, in $\langle\HH{\omega_2},\in\rangle$, the statement that $\varphi$, $\omega_1$ and $z$ 
 define a set of cardinality greater than $\aleph_1$ that consists of unbounded subsets of $\omega_1$ whose pairwise intersections are countable can be expressed by a $\Pi_2$-formula with parameter $z$. 
 Let $G$ be $\PPP_{max}$-generic over $\LL(\RRR)$. 
 Then the $\Pi_2$-maximality of $\LL(\RRR)[G]$ 
 implies that, in the structure $\langle\HH{\aleph_2}^{\LL(\RRR)[G]},\in\rangle$, the formula  $\varphi$ and the parameters $\omega_1$ and $z$ define a set of cardinality greater than $\aleph_1$ that consists of unbounded subsets of $\omega_1$ whose pairwise intersections are countable. 
 In particular, such a subset of $\POT{\omega_1}$ exists in $\OD(\RRR)^{\LL(\RRR)[G]}$, contradicting  Lemma \ref{lemma:MADinPmax}.  
\end{proof}

\section{Concluding remarks and open questions}

 In the following, we discuss several questions raised by the above results, starting with questions about the optimality of the assumption of Theorem \ref{theorem:PerfectSubset}. By Theorem \ref{theorem:PerfectSubsetOptimal}, the consistency strength of this assumption is optimal in the case of singular cardinals. 
 In contrast, results of Schlicht in \cite{MR3743612} show that, if $\kappa$ is an uncountable regular cardinal, $\theta>\kappa$ is inaccessible and $G$ is $\Col{\kappa}{{<}\theta}$-generic over $\VV$, then, in $\VV[G]$, every subset of $\kappa$ in $\OD({}^\kappa\On)$ either has cardinality $\kappa$ or contains a closed subset homeomorphic to ${}^\kappa 2$. 
 In particular, if $\kappa$ is not weakly compact in $\VV$, then, in $\VV[G]$, the cardinal $\kappa$ is not weakly compact, the spaces ${}^\kappa 2$ and ${}^\kappa\kappa$ are homeomorphic (see, for example, {\cite[Corollary 2.3]{MR3557473}}) and for every subset $D$ of $\POT{\kappa}$ of cardinality greater than $\kappa$ that is definable by a $\Sigma_1$-formula with parameters in $\HH{\kappa}\cup\{\kappa\}$, there is a perfect embedding $\map{\iota}{{}^\kappa\kappa}{\POT{\kappa}}$ with $\ran{\iota}\subseteq D$. 
 This shows that our question is only interesting when we also assume that the given cardinal $\kappa$ possesses certain large cardinal properties.\footnote{Note that the assumption that the weak compactness of a cardinal $\kappa$ is preserved by forcing with partial orders of the form $\Col{\kappa}{{<}\theta}$ has high consistency strength (see, for example, \cite{HM19} and \cite{MR2499432}).}

 \begin{question}\label{Q1}
  Assume that $\kappa$ is a weakly compact cardinal with the property that for every subset $D$ of $\POT{\kappa}$ of cardinality greater than $\kappa$ that is definable by a $\Sigma_1$-formula with parameters in $\HH{\kappa}\cup\{\kappa\}$, there is a perfect embedding $\map{\iota}{{}^\kappa\kappa}{\POT{\kappa}}$ with $\ran{\iota}\subseteq D$. Is there an inner model that contains a weakly compact limit of measurable cardinals?   
 \end{question}

 In contrast to the singular case, we may also ask whether the conclusion of Theorem \ref{theorem:PerfectSubset} can be established from a  sequences of measurable cardinals that are bounded in a regular cardinal, but whose order type is equal their minimum.

 \begin{question}
  Does the assumption of Question \ref{theorem:PerfectSubset} imply the existence of a set-sized transitive model of $\ZFC$ containing a weakly compact cardinal $\delta$ and a sequence $S$ of measurable cardinals less than $\delta$ of order-type  $\min(S)$? 
 \end{question}

  We now consider the possibility to strengthen Theorem \ref{MainTheorem:LimitMeasurables}. Since the existence of infinitely many measurable cardinals is compatible with the existence of a $\mathbf{\Sigma}^1_3$-well-ordering of the reals (see {\cite[Theorem 3.6]{MR344123}}), it is natural to ask whether the assumption of this theorem is actually consistent. The model constructed in {\cite[Section 1]{MR344123}} should be  the  natural candidate to look for an affirmative answer to the following question.

\begin{question}
 Is it consistent that there exists a limit $\kappa$ of $\omega$-many measurable cardinals and a well-ordering of a subset of $\POT{\kappa}$ of cardinality greater than $\kappa$ that is definable by a $\Sigma_1$-formula with parameter $\kappa$? 
\end{question}
 
 In addition,  the equiconsistency given by  Theorem \ref{theorem:EquiLongWO} motivates the question whether such implications can be extended to cardinals of higher cofinalities. 

 \begin{question}
  Is it consistent that there exists a limit of measurable cardinals $\kappa$ and a  well-ordering of a subset of $\POT{\kappa}$ of cardinality greater than $\kappa$ that is definable by a $\Sigma_1$-formula with parameters in  $\HH{\kappa}\cup\{\kappa\}$? 
 \end{question}

 Next, we consider simply definable almost disjoint families. The statements of Theorem \ref{theorem:AlmostDisjoint} and  Theorem \ref{theorem:ResultsOmega1}.\eqref{item:Omega1-2}   are motivated by a classical result of Mathias in \cite{MR491197} showing that no maximal almost disjoint family in $\POT{\omega}$ is analytic.  
 In contrast, Miller \cite{MR983001} showed  that that the existence of coanalytic maximal disjoint families in $\POT{\omega}$ is consistent. This motivates the following questions:

 \begin{question}
  \begin{enumerate}
      \item Does the existence of sufficiently strong large cardinals imply that no almost disjoint family  of cardinality greater than $\aleph_1$ in $\POT{\omega_1}$ is definable by a $\Pi_1$-formula with parameters in $\HH{\aleph_1}\cup\{\omega_1\}$? 
      
      \item Do  sufficiently strong large cardinal properties of a cardinal $\kappa$ imply that no almost disjoint family  of cardinality greater than $\kappa$ in $\POT{\kappa}$ is definable by a $\Pi_1$-formula with parameters in $\HH{\kappa}\cup\{\kappa\}$? 
  \end{enumerate}
 \end{question}

 It should be noted that our proof of Theorem \ref{theorem:ResultsOmega1}.\eqref{item:Omega1-2} in Section \ref{section:AlmostDisjointOmega1}  already shows  that Woodin's Axiom $(*)$ (and therefore strong forcing axioms, see \cite{MR4250390}) implies that no almost disjoint family of cardinality greater than $\aleph_1$ in $\POT{\omega_1}$ is definable in the structure $\langle\HH{\aleph_2},\in\rangle$ by a formula with parameters in $\HH{\aleph_1}\cup\omega_2$, because all such families are elements of $\OD(\RRR)^{\LL(\POT{\omega_1})}$. In particular, no such family is definable by a $\Pi_1$-formula with parameters in $\HH{\aleph_1}\cup\{\omega_1\}$ in this setting.

 Finally, we consider the questions whether analogues of the above results hold for other types of uncountable cardinals. 
 The following observation uses ideas from  {\cite[Section 6]{Sigma1Partitions}} and {\cite[Section 5]{MR3694344}} to show that the results of Section \ref{section:AlmostDisjointOmega1} cannot be generalized from $\omega_1$ to $\omega_2$. 
 Moreover, it shows that forcing axioms outright imply the $\Sigma_1$-definability of pathological objects at $\omega_2$. 
 Remember that a sequence $\seq{C_\alpha}{\alpha\in\Lim\cap\omega_1}$ is a \emph{C-sequence} if  $C_\alpha$ is an unbounded subset of $\alpha$ of order-type $\omega$ for every countable limit ordinal $\alpha$.

\begin{proposition}\label{proposition:SimplyDefOmega2}
 \begin{enumerate}
     \item If the \emph{Bounded Proper Forcing Axiom $\BPFA$} holds and $\vec{C}$ is a C-sequence, then there exists an almost disjoint family of cardinality $2^{\aleph_2}$ in $\POT{\omega_2}$ that  is definable by a $\Sigma_1$-formula with parameters $\omega_2$ and $\vec{C}$.  

     \item If $\omega_1=\omega_1^\LL$ and $\BPFA$ holds,\footnote{Note that the results of \cite{MR1324501} show that, if there exists a $\Sigma_1$-reflecting cardinal in $\LL$, then these assumptions hold in a generic extension of $\LL$.} then there exists an almost disjoint family of cardinality $2^{\aleph_2}$ in $\POT{\omega_2}$ that  is definable by a $\Sigma_1$-formula with parameter $\omega_2$. 
     
     \item If there is a supercompact cardinal, then, in a generic extension of the ground model, the \emph{Proper Forcing Axiom $\PFA$} holds and there exists an almost disjoint family of cardinality $2^{\aleph_2}$ in $\POT{\omega_2}$ that  is definable by a $\Sigma_1$-formula and the  parameter  $\omega_2$. 
 \end{enumerate}
\end{proposition}

\begin{proof}
 (i) By {\cite[Theorem 2]{MR2231126}}, our assumption implies the existence of a well-ordering of $\HH{\aleph_2}$ of order-type $\omega_2$ that is definable by a $\Sigma_1$-formula that only uses the sequence $\vec{C}$  as a parameter. 
 In particular, there exists an injection $\map{\iota}{\HH{\aleph_2}}{\omega_2}$ that is definable  in the structure $\langle\HH{\aleph_2},\in\rangle$ by a formula with parameter $\vec{C}$. 
 Since {\cite[Lemma 6.4]{Sigma1Partitions}} shows that our assumption implies that the set $\{\HH{\aleph_2}\}$ is definable by a $\Sigma_1$-formula with parameter $\omega_2$, we know that $\iota$ is definable by a $\Sigma_1$-formula with parameters $\omega_2$ and $\vec{C}$. 
 Given $x\in {}^{\omega_2}2$, we now define $$\bar{x} ~ = ~ \Set{\iota(x\restriction\gamma)}{\gamma<\omega_2} ~ \in ~ \POT{\omega_2}.$$
 
 The above computations then show that the set $A=\Set{\bar{x}}{x\in {}^{\omega_2}2}$ is definable by a $\Sigma_1$-formula with parameters $\omega_2$ and $\vec{C}$, and it is easy to see that $A$ is an almost disjoint family of cardinality $2^{\aleph_2}$ in $\POT{\omega_2}$. 

 (ii) Assume that $\omega_1=\omega_1^\LL$ and $\BPFA$ holds. Let $\vec{C}$ denote the $<_\LL$-least C-sequence in $\LL$. Then $\vec{C}$ is a $C$-sequence and the set $\{\vec{C}\}$ is definable by a $\Sigma_1$-formula with parameter $\omega_2$. Using (i), we can conclude that there exists an almost disjoint family of cardinality $2^{\aleph_2}$ in $\POT{\omega_2}$ that  is definable by a $\Sigma_1$-formula with parameter $\omega_2$. 
 
 (iii) By {\cite[Theorem 5.2]{MR2320944}}, it is possible to start in a model containing a supercompact cardinal and force the validity of $\PFA$ together with the existence of a well-ordering of $\HH{\aleph_2}$ of order-type $\omega_2$ that is definable in $\langle\HH{\aleph_2},\in\rangle$ by a formula without parameters. 
 We can now proceed as in (i) to obtain an almost disjoint family of cardinality $2^{\aleph_2}$ in $\POT{\omega_2}$ that is definable by a $\Sigma_1$-formula with parameter $\omega_2$ in this generic extension. 
\end{proof}


 \bibliographystyle{plain}
\bibliography{references}

\begin{thebibliography}{10}

\bibitem{MR2320944}
David Asper\'o.
\newblock Guessing and non-guessing of canonical functions.
\newblock {\em Annals of Pure and Applied Logic}, 146(2-3):150--179, 2007.

\bibitem{MR4250390}
David Asper\'{o} and Ralf Schindler.
\newblock Martin's {M}aximum{$^{++}$} implies {W}oodin's axiom {$(*)$}.
\newblock {\em Ann. of Math. (2)}, 193(3):793--835, 2021.

\bibitem{doi:10.1142/S0219061321500264}
Karen Bakke~Haga, David Schrittesser, and Asger T\"{o}rnquist.
\newblock Maximal almost disjoint families, determinacy, and forcing.
\newblock {\em J. Math. Log.}, 22(1):Paper No. 2150026, 42, 2022.

\bibitem{MR2777751}
Andr\'{e}s~Eduardo Caicedo and Richard Ketchersid.
\newblock A trichotomy theorem in natural models of {$\rm{AD}^+$}.
\newblock In {\em Set theory and its applications}, volume 533 of {\em Contemp.
  Math.}, pages 227--258. Amer. Math. Soc., Providence, RI, 2011.

\bibitem{MR2231126}
Andr{\'e}s~Eduardo Caicedo and Boban Veli{\v{c}}kovi{\'c}.
\newblock The bounded proper forcing axiom and well orderings of the reals.
\newblock {\em Mathematical Research Letters}, 13(2-3):393--408, 2006.

\bibitem{CJ19}
William Chan and Stephen Jackson.
\newblock {L(R)} with determinacy satisfies the {S}uslin {H}ypothesis.
\newblock {\em Advances in Mathematics}, 346:305--328, 2019.

\bibitem{CJ21}
William Chan and Stephen Jackson.
\newblock Definable combinatorics at the first uncountable cardinal.
\newblock {\em Transactions of the American Mathematical Society},
  374(3):2035--2056, 2021.

\bibitem{AlmostDisjointAD}
William Chan, Stephen Jackson, and Nam Trang.
\newblock Almost disjoint families under determinacy.
\newblock Preprint.

\bibitem{CJT2}
William Chan, Stephen Jackson, and Nam Trang.
\newblock The size of the class of countable sequences of ordinals.
\newblock {\em Trans. Amer. Math. Soc.}, 375(3):1725--1743, 2022.

\bibitem{MR2567927}
Sean Cox.
\newblock Covering theorems for the core model, and an application to
  stationary set reflection.
\newblock {\em Ann. Pure Appl. Logic}, 161(1):66--93, 2009.

\bibitem{DEHORNOY1978109}
Patrick Dehornoy.
\newblock Iterated ultrapowers and {P}rikry forcing.
\newblock {\em Ann. Math. Logic}, 15(2):109--160, 1978.

\bibitem{FISCHER2021102909}
Vera Fischer, David Schrittesser, and Thilo Weinert.
\newblock Definable mad families and forcing axioms.
\newblock {\em Annals of Pure and Applied Logic}, 172(5):102909, 2021.
\newblock 50 Years of Set theory in Toronto.

\bibitem{MR3320593}
Sy-David Friedman, Peter Holy, and Philipp L{\"u}cke.
\newblock Large cardinals and lightface definable well-orders, without the
  {GCH}.
\newblock {\em J. Symb. Log.}, 80(1):251--284, 2015.

\bibitem{MR2193185}
Gunter Fuchs.
\newblock A characterization of generalized {P}\v{r}\'{\i}kr\'{y} sequences.
\newblock {\em Arch. Math. Logic}, 44(8):935--971, 2005.

\bibitem{MR2768695}
Moti Gitik.
\newblock Prikry-type forcings.
\newblock In {\em Handbook of set theory. {V}ols. 1, 2, 3}, pages 1351--1447.
  Springer, Dordrecht, 2010.

\bibitem{MR2830415}
Victoria Gitman.
\newblock Ramsey-like cardinals.
\newblock {\em J. Symbolic Logic}, 76(2):519--540, 2011.

\bibitem{MR2830435}
Victoria Gitman and Philip~D. Welch.
\newblock Ramsey-like cardinals {II}.
\newblock {\em J. Symbolic Logic}, 76(2):541--560, 2011.

\bibitem{MR1324501}
Martin Goldstern and Saharon Shelah.
\newblock The bounded proper forcing axiom.
\newblock {\em J. Symbolic Logic}, 60(1):58--73, 1995.

\bibitem{MR0373889}
Serge Grigorieff.
\newblock Intermediate submodels and generic extensions in set theory.
\newblock {\em Annals of Mathematics (2)}, 101:447--490, 1975.

\bibitem{MR1133077}
Kai Hauser.
\newblock Indescribable cardinals and elementary embeddings.
\newblock {\em J. Symbolic Logic}, 56(2):439--457, 1991.

\bibitem{HM19}
Yair Hayut and Sandra M\"{u}ller.
\newblock Perfect subtree property for weakly compact cardinals.
\newblock {\em Israel J. Math.}, 253(2):865--886, 2023.

\bibitem{HL}
Peter Holy and Philipp L{\"u}cke.
\newblock Locally {$\Sigma_1$}-definable well-orders of
  {${\mathrm{H}}(\kappa^+)$}.
\newblock {\em Fund. Math.}, 226(3):221--236, 2014.

\bibitem{MR3928385}
Haim Horowitz and Saharon Shelah.
\newblock On the non-existence of mad families.
\newblock {\em Arch. Math. Logic}, 58(3-4):325--338, 2019.

\bibitem{Je71}
Thomas~J. Jech.
\newblock Trees.
\newblock {\em J. Symbolic Logic}, 36:1--14, 1971.

\bibitem{MR1940513}
Thomas~J. Jech.
\newblock {\em Set theory}.
\newblock Springer Monographs in Mathematics. Springer-Verlag, Berlin, 2003.
\newblock The third millennium edition, revised and expanded.

\bibitem{JS13}
R.~B. Jensen and J.~R. Steel.
\newblock {K without the measurable}.
\newblock {\em J. Symb. Log.}, 78(3):708--734, 2013.

\bibitem{MR2499432}
Ronald Jensen, Ernest Schimmerling, Ralf Schindler, and John Steel.
\newblock Stacking mice.
\newblock {\em J. Symbolic Logic}, 74(1):315--335, 2009.

\bibitem{MR1994835}
Akihiro Kanamori.
\newblock {\em The higher infinite}.
\newblock Springer Monographs in Mathematics. Springer-Verlag, Berlin, second
  edition, 2003.
\newblock Large cardinals in set theory from their beginnings.

\bibitem{MR526915}
Alexander~S. Kechris.
\newblock {${\rm AD}$} and projective ordinals.
\newblock In {\em Cabal {S}eminar 76--77 ({P}roc. {C}altech-{UCLA} {L}ogic
  {S}em., 1976--77)}, volume 689 of {\em Lecture Notes in Math.}, pages
  91--132. Springer, Berlin, 1978.

\bibitem{MR1321597}
Alexander~S. Kechris.
\newblock {\em Classical descriptive set theory}, volume 156 of {\em Graduate
  Texts in Mathematics}.
\newblock Springer-Verlag, New York, 1995.

\bibitem{MR926749}
Peter Koepke.
\newblock Some applications of short core models.
\newblock {\em Ann. Pure Appl. Logic}, 37(2):179--204, 1988.

\bibitem{MR2768703}
Paul~B. Larson.
\newblock Forcing over models of determinacy.
\newblock In {\em Handbook of set theory. {V}ols. 1, 2, 3}, pages 2121--2177.
  Springer, Dordrecht, 2010.

\bibitem{MR2987148}
Philipp L\"{u}cke.
\newblock {$\Sigma^1_1$}-definability at uncountable regular cardinals.
\newblock {\em J. Symbolic Logic}, 77(3):1011--1046, 2012.

\bibitem{Sigma1Partitions}
Philipp L{\"u}cke.
\newblock Partition properties for simply definable colourings.
\newblock {\em Israel Journal of Mathematics}, 236(2):841--898, 2020.

\bibitem{MR3557473}
Philipp L\"{u}cke, Luca Motto~Ros, and Philipp Schlicht.
\newblock The {H}urewicz dichotomy for generalized {B}aire spaces.
\newblock {\em Israel J. Math.}, 216(2):973--1022, 2016.

\bibitem{ClosureUlt}
Philipp L\"{u}cke and Sandra M\"{u}ller.
\newblock Closure properties of measurable ultrapowers.
\newblock {\em J. Symb. Log.}, 86(2):762--784, 2021.

\bibitem{MR3694344}
Philipp L\"ucke, Ralf Schindler, and Philipp Schlicht.
\newblock {$\Sigma_1(\kappa)$}-definable subsets of {${\rm H}(\kappa^+)$}.
\newblock {\em The Journal of Symbolic Logic}, 82(3):1106--1131, 2017.

\bibitem{MR3845129}
Philipp L\"{u}cke and Philipp Schlicht.
\newblock Measurable cardinals and good {$\Sigma_1(\kappa)$}-wellorderings.
\newblock {\em Mathematical Logic Quarterly}, 64(3):207--217, 2018.

\bibitem{MR491197}
Adrian R.~D. Mathias.
\newblock Happy families.
\newblock {\em Ann. Math. Logic}, 12(1):59--111, 1977.

\bibitem{MR983001}
Arnold~W. Miller.
\newblock Infinite combinatorics and definability.
\newblock {\em Ann. Pure Appl. Logic}, 41(2):179--203, 1989.

\bibitem{MR344123}
William~J. Mitchell.
\newblock Sets constructible from sequences of ultrafilters.
\newblock {\em J. Symbolic Logic}, 39:57--66, 1974.

\bibitem{MR3835078}
Itay Neeman and Zach Norwood.
\newblock Happy and mad families in {$L(\Bbb{R})$}.
\newblock {\em J. Symb. Log.}, 83(2):572--597, 2018.

\bibitem{Schimmerling10}
Ernest Schimmerling.
\newblock A core model toolbox and guide.
\newblock In {\em Handbook of set theory. {V}ols. 1, 2, 3}, pages 1685--1751.
  Springer, Dordrecht, 2010.

\bibitem{MR3743612}
Philipp Schlicht.
\newblock Perfect subsets of generalized {B}aire spaces and long games.
\newblock {\em J. Symb. Log.}, 82(4):1317--1355, 2017.

\bibitem{MR4012549}
David Schrittesser and Asger T\"{o}rnquist.
\newblock The {R}amsey property implies no mad families.
\newblock {\em Proc. Natl. Acad. Sci. USA}, 116(38):18883--18887, 2019.

\bibitem{MR2817562}
I.~Sharpe and Philip~D. Welch.
\newblock Greatly {E}rd{\H{o}}s cardinals with some generalizations to the
  {C}hang and {R}amsey properties.
\newblock {\em Ann. Pure Appl. Logic}, 162(11):863--902, 2011.

\bibitem{St96}
John~R. Steel.
\newblock {\em {The core model iterability problem}}, volume~8 of {\em Lecture
  Notes in Logic}.
\newblock Springer-Verlag, Berlin, New York, 1996.

\bibitem{St10}
John~R. Steel.
\newblock An outline of inner model theory.
\newblock In {\em Handbook of set theory. {V}ols. 1, 2, 3}, pages 1595--1684.
  Springer, Dordrecht, 2010.

\bibitem{MR3411035}
John~R. Steel.
\newblock An introduction to iterated ultrapowers.
\newblock In {\em Forcing, iterated ultrapowers, and {T}uring degrees},
  volume~29 of {\em Lect. Notes Ser. Inst. Math. Sci. Natl. Univ. Singap.},
  pages 123--174. World Sci. Publ., Hackensack, NJ, 2016.

\bibitem{MR3156517}
Asger T\"{o}rnquist.
\newblock {$\Sigma^1_2$} and {$\Pi^1_1$} mad families.
\newblock {\em J. Symbolic Logic}, 78(4):1181--1182, 2013.

\bibitem{welch_2020}
Philip~D. Welch.
\newblock Stably measurable cardinals.
\newblock {\em J. Symb. Log.}, 86(2):448--470, 2021.

\bibitem{MR1713438}
W.~Hugh Woodin.
\newblock {\em The axiom of determinacy, forcing axioms, and the nonstationary
  ideal}, volume~1 of {\em De Gruyter Series in Logic and its Applications}.
\newblock Walter de Gruyter \& Co., Berlin, 1999.

\bibitem{Ze02}
Martin Zeman.
\newblock {\em Inner models and large cardinals}, volume~5 of {\em De Gruyter
  Series in Logic and its Applications}.
\newblock Walter de Gruyter \& Co., Berlin, 2002.

\end{thebibliography}

\end{document}